\documentclass[oneside,english]{amsart}
\usepackage[utf8]{inputenc}
\usepackage{color}
\usepackage{babel}
\usepackage{array}
\usepackage{textcomp}
\usepackage{multirow}
\usepackage{amsthm}
\usepackage{amsbsy}
\usepackage{amstext}
\usepackage{amssymb}
\usepackage[all]{xy}
\usepackage[unicode=true,pdfusetitle,
 bookmarks=true,bookmarksnumbered=true,bookmarksopen=true,bookmarksopenlevel=2,
 breaklinks=true,pdfborder={0 0 1},backref=false,colorlinks=true]
 {hyperref}
\hypersetup{
 linkcolor=blue, urlcolor=blue, citecolor=blue}
\usepackage{breakurl}

\makeatletter

\DeclareFontEncoding{LGR}{}{}
\DeclareTextSymbol{\~}{LGR}{126}
\providecommand{\tabularnewline}{\\}

\numberwithin{equation}{section}
\numberwithin{figure}{section}
\theoremstyle{plain}
\newtheorem{thm}{\protect\theoremname}[section]
  \theoremstyle{remark}
  \newtheorem*{rem*}{\protect\remarkname}
  \theoremstyle{plain}
  \newtheorem{lem}[thm]{\protect\lemmaname}
  \theoremstyle{plain}
  \newtheorem{prop}[thm]{\protect\propositionname}

\usepackage[a4paper]{geometry} 
\usepackage{graphicx}

\makeatother

  \providecommand{\lemmaname}{Lemma}
  \providecommand{\propositionname}{Proposition}
  \providecommand{\remarkname}{Remark}
\providecommand{\theoremname}{Theorem}

\begin{document}
\global\long\def\tri{1\kern-0.77ex  1}

\title{Explicit induction principle and symplectic-orthogonal theta lifts
}

\author{Xiang Fan}

\address{School of Mathematics, Sun Yat-sen University, Guangzhou 510275,
China}


\begin{abstract}
During the last two decades, great efforts have been devoted to the
calculation of the local theta correspondence for reductive dual pairs.
However, uniform formulas remain elusive for real dual pairs of type
I. The purpose of this paper is twofold: to formulate an explicit
version of induction principle for dual pairs $(O(p,q),Sp(2n,\mathbb{R}))$
with $p+q$ even, and to apply it to obtain a complete and explicit
description of the local theta correspondence when $p+q=4$. Our approach
is very elementary by analysis on the infinitesimal characters and
$K$-types under the correspondence.
\end{abstract}

\maketitle

\section{Introduction}

\subsection{Background}

To understand the unitary duals of reductive Lie groups, various 
unitary representations were constructed as local theta lifts for
reductive dual pairs \cite{Howe1987small,Li1989singular}. For example,
for $(G,G')=(O(p,q),Sp(2n,\mathbb{R}))$ which is a dual pair in $Sp(2n(p+q),\mathbb{R})$,
let $p+q$ be even (so that the metaplectic $\mathbb{C}^{\times}$-cover
splits on $G\cdot G'$) and $p+q\leqslant n$ (so that this pair is
in the stable range with $G$ the smaller member), then the local
theta lifting gives an injection from the unitary dual of $G$ to
that of $G'$. Varying $(p,q)$, the unitary duals of diverse $G$'s
of smaller size give rise to families of singular unitary representations
of $G'$.

To study the unitary representations thus obtained, and for other
applications to automorphic forms, there is growing interest to determine
the theta correspondence explicitly in terms of Langlands parameters.
Great efforts have been devoted to this calculation during the last
two decades. For complex groups in dual pairs or real dual pairs of
type II, the correspondence is completely and explicitly described
in \cite{Moeglin1989correspondance,AdamsBarbasch1995reductive,LiPaulTanZhu2003explicit}.
For other real dual pairs of type I, explicit results up to now are
mainly for compact cases \cite{EnrightHoweWallach1983classification}
and ``(almost) equal rank'' cases \cite{AdamsBarbasch1998genuine,Paul1998howe,Paul2005howe},
with the exception of \cite{LiTanZhu2003tensor,LiPaulTanZhu2003explicit}
for a few non-equal-rank cases. In general, however, a full and explicit
description of the local theta correspondence remains elusive.

The purpose of this paper is twofold: to formulate an explicit version
of induction principle for dual pairs $(O(p,q),Sp(2n,\mathbb{R}))$
with $p+q$ even, and to apply it to obtain a complete and explicit
description of the local theta correspondence when $p+q=4$. Our approach
is very elementary by analysis on the infinitesimal characters and
$K$-types under the correspondence, which will also work for other
real reductive dual pairs of type I to obtain similar explicit results.

\subsection{Description of the problem}

We aim to investigate the local theta correspondence for dual pairs
$(G,G')=(O(p,q),Sp(2n,\mathbb{R}))$ with $p+q$ even, which can be
interpreted as a correspondence between representations of $G$ and
those of $G'$. The metaplectic $\mathbb{C}^{\times}$-cover 
\[
1\rightarrow\mathbb{C}^{\times}\to Mp(2n(p+q),\mathbb{R})\to Sp(2n(p+q),\mathbb{R})\to1
\]
splits on $G\cdot G'$ as $p+q$ is even, with a particular splitting
map as described in \cite{Kudla1994splitting}, which indeed fixes
an embedding $G\cdot G'\hookrightarrow Mp(2n(p+q),\mathbb{R})$. Via
this embedding, the \emph{Segal-Shale-Weil oscillator representation}
$\omega$ \cite{LionVergne1980theweil} (associated to the character
of $\mathbb{R}$ that maps $t$ to $e^{2\pi\mathrm{i}t}$) is restricted
to $G\cdot G'$. Let $\mathcal{R}(G)$ denote the \emph{admissible
dual} of $G$, i.e., the set of infinitesimal equivalence classes
of  irreducible admissible representations of $G$.  By abuse of
notation, for a class $\pi\in\mathcal{R}(G)$, let $\pi$ also denote
a representation in this class. Let $\mathcal{R}(G,\omega)$ be the
set of elements $\pi$ of $\mathcal{R}(G)$ such that there exists
a nontrivial $G$-intertwining continuous linear map $\omega\to\pi$.
Similar notations hold for $G'$ and $G\cdot G'$. \cite{Howe1989transcending}
gives a one-to-one correspondence $\mathcal{R}(G,\omega)\leftrightarrow\mathcal{R}(G',\omega)$,
which is usually called the \emph{local theta correspondence} or \emph{Howe
duality correspondence}.

For $\pi\in\mathcal{R}(G)$ and $\pi'\in\mathcal{R}(G')$, they correspond
if and only if $\pi\otimes\pi'\in\mathcal{R}(G\cdot G',\omega)$.
In this case we write $\theta_{n}(\pi)=\pi'$ and $\theta_{p,q}(\pi')=\pi$.
If $\pi\notin\mathcal{R}(G,\omega)$, write $\theta_{n}(\pi)=0$.
Similarly, write $\theta_{p,q}(\pi')=0$ if $\pi'\notin\mathcal{R}(G',\omega)$.
Then $\theta_{n}(\pi)$ is called the \emph{theta $n$-lift} of $\pi$,
and the map $\theta_{n}:\mathcal{R}(G)\to\mathcal{R}(G',\omega)\cup\{0\}$
is called the \emph{theta $n$-lifting} for $G$. Similarly, $\theta_{p,q}(\pi')$
is called the \emph{theta $(p,q)$-lift} of $\pi'$, and $\theta_{p,q}:\mathcal{R}(G')\to\mathcal{R}(G,\omega)\cup\{0\}$
is called the \emph{theta $(p,q)$-lifting} for $G'$.

Our goal is to determine $\theta_{n}:\mathcal{R}(G)\to\mathcal{R}(G',\omega)\cup\{0\}$
explicitly in terms of \emph{Langlands parameters}. We choose Vogan's
version \cite{Vogan1984unitarizability} of Langlands classification,
which is described in \cite{Paul2005howe} and recollected in Subsection
\ref{subsec:Par-Sp} and \ref{subsec:Par-O}. Roughly speaking, $\pi\in\mathcal{R}(G)$
is parametrized in the form $\pi=\pi_{\zeta}(\lambda_{d},\xi,\Psi,\mu,\nu,\varepsilon,\kappa)$,
and $\pi'\in\mathcal{R}(G')$ in the form $\pi'=\pi(\lambda_{d},\Psi,\mu,\nu,\varepsilon,\kappa)$.

For $(G,G')=(O(p,q),Sp(2n,\mathbb{R}))$, \cite{Moeglin1989correspondance,Paul2005howe}
determine the local theta correspondence explicitly when $p+q=2n$
or $2n+2$ (equal rank or almost equal rank cases). The present paper
will give two new results when $p+q$ is even:
\begin{itemize}
\item an explicit version of induction principle for the correspondence;
\item an explicit description of the full correspondence for all $n,p,q$
when $p+q=4$.
\end{itemize}
The following two subsections outline them with basic ideas.

\subsection{Explicit induction principle}

For $\pi\in\mathcal{R}(G)$ and $\pi'\in\mathcal{R}(G')$, \emph{Kudla's
persistence principle} asserts that \cite[I.9]{Moeglin1989correspondance}:
if $\theta_{n}(\pi)\neq0$ then $\theta_{n+1}(\pi)\neq0$; if $\theta_{p,q}(\pi')\neq0$
then $\theta_{p+1,q+1}(\pi')\neq0$. To determine $\theta_{n+1}(\pi)$
from a given nonzero $\theta_{n}(\pi)$ (respectively, $\theta_{p+1,q+1}(\pi')$
from a nonzero $\theta_{p,q}(\pi')$), a most powerful tool is the
``induction principle'' developed by \cite{Kudla1986local,Moeglin1989correspondance,AdamsBarbasch1995reductive,Paul2005howe}.
This paper extends it to an explicit version for $(O(p,q),Sp(2n,\mathbb{R}))$
with $p+q$ even, in terms of Langlands parameters.
\begin{itemize}
\item For convenience, for $x=(x_{1},\dots,x_{m})\in\mathbb{C}^{m}$ and
$y=(y_{1},\dots,y_{l})\in\mathbb{C}^{l}$, write $(x\mid y)=(x_{1},\dots,x_{m},y_{1},\dots,y_{l})\in\mathbb{C}^{m+l}$.\end{itemize}
\begin{thm}
[Explicit induction principle] Let $p+q$ be even, and $1\leqslant k\in\mathbb{Z}$.

(1) If $\pi\in\mathcal{R}(O(p,q))$ and $0\neq\theta_{n}(\pi)=\pi(\lambda_{d},\Psi,\mu,\nu,\varepsilon,\kappa)$
with $p+q\leqslant2n$, then 
\begin{align*}
\theta_{n+k}(\pi) & =\pi(\lambda_{d},\Psi,\mu,\nu,(\varepsilon\mid((-1)^{\frac{p-q}{2}},(-1)^{\frac{p-q}{2}},\dots,(-1)^{\frac{p-q}{2}})),\\
 & \qquad(\kappa\mid(1+n-\frac{p+q}{2},2+n-\frac{p+q}{2},\dots,k+n-\frac{p+q}{2})))
\end{align*}
with a possible modification.

(2) If $\pi'\in\mathcal{R}(Sp(2n,\mathbb{R}))$ and $0\neq\theta_{p,q}(\pi')=\pi_{\zeta}(\lambda_{d},\xi,\Psi,\mu,\nu,\varepsilon,\kappa)$
with $p+q\geqslant2n+2$, then $\zeta=\xi=1$ and
\begin{align*}
\theta_{p+k,q+k}(\pi') & =\pi_{1}(\lambda_{d},1,\Psi,\mu,\nu,(\varepsilon\mid(1,1,\dots,1)),\\
 & \qquad(\kappa\mid(\frac{p+q}{2}-n,1+\frac{p+q}{2}-n,\dots,k-1+\frac{p+q}{2}-n)))
\end{align*}
with a possible modification.

The \emph{``possible modification''} is: if the resulting parameters
contain some entries $\kappa_{i}=\pm\kappa_{j}$ with $\varepsilon_{i}\neq\varepsilon_{j}$,
delete these $\varepsilon_{i}$, $\varepsilon_{j}$, $\kappa_{i}$,
$\kappa_{j}$ from $(\varepsilon,\kappa)$, and add entries $(0,2\kappa_{i})$
into $(\mu,\nu)$.
\end{thm}
This theorem will be proved as special cases of Theorem \ref{thm:EIP_n}
and Theorem \ref{thm:EIP_pq}. Indeed, \cite{Paul2005howe} already
pointed out the result (1) when $p+q=2n$ and the result (2) when
$p+q=2n+2$. Our observation is that Paul's idea can be generalized
to more cases, due to certain patterns of the Langlands parameters
of theta lifts (see Lemma \ref{lem:cond_n} and \ref{lem:cond_pq}).
We call this first main result the \emph{``explicit induction principle''}.

The nonvanishing of theta liftings in the stable range \cite{ProtsakPrzebinda2008occurrence}
asserts that $\theta_{p+q}(\pi)\neq0$ for all $\pi\in\mathcal{R}(G)$.
Given $\theta_{p+q}(\pi)$ explicitly, then $\theta_{p+q+k}(\pi)$
can be read off by our explicit induction principle.  For fixed $(p,q)$
with $p+q$ even, it is strong enough to reduce the calculation of
all $\theta_{n}$ to only finitely many $n\leqslant p+q$.

\subsection{Explicit theta correspondence when $p+q=4$}

We wish to determine $\theta_{n}:\mathcal{R}(O(p,q))\to\mathcal{R}(Sp(2n,\mathbb{R}))\cup\{0\}$
explicitly in terms of Langlands parameters for all $n$ when $p+q=4$.
Notice that by our explicit induction principle, it suffices to calculate
$\theta_{3}$ and $\theta_{4}$ (as $\theta_{1}$ and $\theta_{2}$
are explicitly given by \cite{Paul2005howe} and written in Appendix
\ref{App:Theta12}). Moreover, for $\pi\in\mathcal{R}(O(p,q))$, note
that $\theta_{p+q-1}(\pi)\neq0$ if $\pi$ is not the \emph{determinant
character} $\mathrm{det}_{p,q}$ of $O(p,q)$. Then it further reduces
the calculation to theta $4$-lifts of $\mathrm{det}_{p,q}$ and theta
$3$-lifts of other $\pi\in\mathcal{R}(O(p,q))$ with $\theta_{2}(\pi)=0$.
For every such $\pi$ (or $\mathrm{det}_{p,q}$), we can list its
Langlands parameters, and calculate two invariants: its \emph{infinitesimal
character} and the set $\mathcal{A}(\pi)$ (or $\mathcal{A}(\mathrm{det}_{p,q})$)
of all its \emph{lowest $K$-types} ($K$-types with minimal norm
in the sense of \cite{Vogan1979algebraic}), by algorithms in Subsection
\ref{sub:InfCh} and Appendix \ref{App:CompLKT}. For its theta lift,
we also obtain these two invariants explicitly as follows.

The correspondence of infinitesimal characters under theta correspondence
is clearly known from \cite{Przebinda1996duality} (see Subsection
\ref{sub:InfCh}). It gives the infinitesimal characters of $\theta_{3}(\pi)$
and $\theta_{4}(\mathrm{det}_{p,q})$,  which are of the form $(\beta,0,1)$
(with $\beta\in\mathbb{C}$) and $(0,1,1,2)$ respectively.

Let us obtain $\mathcal{A}(\theta_{3}(\pi))$ in three steps: $\mathcal{A}(\pi)\rightsquigarrow\mathcal{D}(\pi)\rightsquigarrow\mathcal{D}(\theta_{3}(\pi))\rightsquigarrow\mathcal{A}(\theta_{3}(\pi))$.
\begin{itemize}
\item By definition $\mathcal{D}(\pi)=\{K$-types of minimal degree in $\pi\}$.
We only need to check the occurrences in $\pi$ of the $K$-types
with degrees $\leqslant$ the minimal degree of $K$-types in $\mathcal{A}(\pi)$.
This will be done by Frobenius reciprocity and explicit models of
$\pi$.
\item We know $\mathcal{D}(\pi)\leftrightarrow\mathcal{D}(\theta_{3}(\pi))$
under the correspondence of $K$-types in the space of joint harmonics,
which is explicitly expressed in Proposition \ref{prop:CorKT}.
\item We have $\mathcal{A}(\theta_{3}(\pi))\subseteq\mathcal{D}(\theta_{3}(\pi))$,
where $\mathcal{D}(\theta_{3}(\pi))$ is the set of \emph{$K$-types
of minimal degree} in $\theta_{3}(\pi)$ in the sense of \cite{Howe1989transcending}.
(This fact is shown in \cite{Paul2005howe} for $p+q\leqslant2n$,
and is contained in Lemma \ref{lem:cond_n}). So $\mathcal{A}(\theta_{3}(\pi))=\{\sigma\in\mathcal{D}(\theta_{3}(\pi))\text{ with minimal norm}\}$.
\end{itemize}
Similar steps give $\mathcal{A}(\theta_{4}(\mathrm{det}_{p,q}))$,
which is easier as $\mathrm{det}_{p,q}$ contains only one $K$-type.
Fortunately, these two invariants are good enough to determine all
desired theta lifts with only one exception.
\begin{thm}
\label{thm:main2} Let $\det_{p,q}\neq\pi\in\mathcal{R}(O(p,q))$
with $p+q=4$.

(1) If $\theta_{2}(\pi)=0$ and $\pi\neq\pi_{-1}(0,1,\O,0,0,(1,1),(0,2))$,
then the infinitesimal character of $\theta_{3}(\pi)$ and $\mathcal{A}(\theta_{3}(\pi))$
determine a unique element in $\mathcal{R}(Sp(6,\mathbb{R}))$ written
explicitly in Section \ref{sec:3-lifts}.

(2) The infinitesimal character of $\theta_{4}(\det_{p,q})$ and the
set $\mathcal{A}(\theta_{4}(\det_{p,q}))$ determine a unique element
in $\mathcal{R}(Sp(8,\mathbb{R}))$, with Langlands parameters written
explicitly in Section \ref{sec:4-lifts}.\end{thm}
\begin{rem*}
For the exceptional case $\pi=\pi_{-1}(0,1,\O,0,0,(1,1),(0,2))\in\mathcal{R}(O(2,2))$,
there are two elements in $\mathcal{R}(Sp(6,\mathbb{R}))$ with the
desired infinitesimal character and set of lowest $K$-types. It is
not hard to figure out which of them lies in the local theta correspondence.
\end{rem*}
Theorem \ref{thm:main2} is checked by tedious but elementary case-by-case
consideration in Section \ref{sec:4-lifts}, Section \ref{sec:3-lifts},
and Appendix \ref{App:list} which not only lists all $\pi'\in\mathcal{R}(Sp(6,\mathbb{R}))$
with the infinitesimal character $(\beta,0,1)$ for all $\beta\in\mathbb{C}$,
but also gives $\mathcal{A}(\pi')$ explicitly. In this way, together
with our explicit induction principle, we explicitly determine the
local theta correspondence for $(O(p,q),Sp(2n,\mathbb{R}))$ for all
$n$ when $p+q=4$.

\section{\label{sec:Para_KT_Rep}Parametrization}

In this section we introduce some auxiliary notations and parametrizations.
 The advantage of our parametrizations is that there are explicit
algorithms to calculate lowest $K$-types and infinitesimal characters
from Langlands parameters.

\subsection{Notations for $K$-types}

For a compact Lie group $K$, a \emph{$K$-type }means an irreducible
representation (up to equivalence) of $K$. Let $\widehat{K}$ denote
the set of all $K$-types. A $K$-type is automatically finite-dimensional
and unitary, so $\widehat{K}=\mathcal{R}(K)$.

Let $H$ be a reductive Lie group with a maximal compact subgroup
$K_{H}$. We refer to $K_{H}$-types as \emph{$K$-types for $H$},
or simply as \emph{$K$-types} if the group $H$ is clearly understood.
Let $\pi$ be an admissible representation of $H$. Recall that ``$\pi$
is admissible'' means the \emph{multiplicity} $m(\sigma,\pi)=\dim\mathrm{Hom}_{K_{H}}(\sigma,\pi)$
 is finite for any $K_{H}$-type $\sigma$. When $m(\sigma,\pi)>0$,
we say that \emph{$\sigma$ occurs in $\pi$}, or that \emph{$\sigma$
is a $K$-type of $\pi$,} or that \emph{$\pi$ contains $\sigma$}.

For $(G,G')=(O(p,q),Sp(2n,\mathbb{R}))$, take the standard maximal
compact subgroups $O(p)\times O(q)$ and $U(n)$ of $O(p,q)$ and
$Sp(2n,\mathbb{R})$ respectively. Henceforth, for a Lie group, the
corresponding lower case Gothic letter is used to indicate the its
complexified Lie algebra.

\subsection{Parametrization for $U(n)$-types}

For $K=U(n)$, take the maximal torus $T=U(1)^{n}=\mathrm{diag}(U(1),\dots,U(1))$,
and the standard system of positive roots 
\[
\Delta^{+}(\mathfrak{t},\mathfrak{k})=\{e_{i}-e_{j}\mid1\leqslant i<j\leqslant n\}.
\]
Write each weight in $\mathrm{i}\mathfrak{t}_{0}^{*}$ as the $n$-tuple
of coefficients under the basis $e_{1}$, $\dots$, $e_{n}$. A $U(n)$-type
is parametrized by its \emph{highest weight} $(a_{1},a_{2},\dots,a_{n})$
with integers $a_{1}\geqslant a_{2}\geqslant\cdots\geqslant a_{n}$.

\subsection{Parametrization for \texorpdfstring{$O(p)\times O(q)$}{$O(p)$\texttimes $O(q)$}-types}

Notice that
\[
\widehat{O(p)\times O(q)}=\{\sigma\otimes\tau\mid\sigma\in O(p),\ \tau\in O(q)\}\simeq\widehat{O(p)}\times\widehat{O(q)}.
\]

\begin{lem}
[\cite{Weyl1997classical}] \label{lem:sign} Embed $O(p)$ in $U(p)$
as $O(p)=U(p)\cap GL(p,\mathbb{R})$. For any $\sigma\in\widehat{O(p)}$,
there exists a unique $\lambda\in\widehat{U(p)}$ such that the $O(p)$-module
generated by the highest weight vectors of $\lambda$ is equivalent
to $\sigma$. Moreover, these $\lambda$'s obtained from $\widehat{O(p)}$
are exactly those parametrized as $(b_{1},b_{2},\dots,b_{r},\underbrace{1,\dots,1}_{s},\underbrace{0,\dots,0}_{p-r-s})$
with $b_{1}\geqslant b_{2}\geqslant\cdots\geqslant b_{r}\geqslant2$
and $2r+s\leqslant p$. Therefore, we may parametrize $\sigma$ as
\[
\begin{cases}
(b_{1},b_{2},\dots,b_{r},\underbrace{1,\dots,1}_{s},\underbrace{0,\dots,0}_{[\frac{p}{2}]-r-s};+1), & \text{if }r+s\leqslant\frac{p}{2},\\
(b_{1},b_{2},\dots,b_{r},\underbrace{1,\dots,1}_{p-2r-s},\underbrace{0,\dots,0}_{[\frac{p}{2}]-p+r+s};-1), & \text{if }\frac{p}{2}\leqslant r+s\leqslant p-r.
\end{cases}
\]
\end{lem}
\begin{rem*}
When $r+s=\frac{p}{2}$, the two cases coincide and give the same
$\sigma$.
\end{rem*}
An $O(p)$-type is parametrized as $\sigma=(a_{1},a_{2},\dots,a_{x},0,\dots,0;\epsilon)$
with integers $a_{1}\geqslant a_{2}\geqslant\cdots\geqslant a_{x}\geqslant1$,
and $\epsilon\in\{\pm1\}$. Its corresponding $U(p)$-type is
\[
\lambda=(a_{1},a_{2},\dots,a_{x},\underbrace{1,\dots,1}_{\frac{1-\epsilon}{2}(p-2x)},0,\dots,0).
\]

\begin{itemize}
\item When $p$ is even and $p=2x$, the two choices of $\epsilon$ give
the same $\sigma$.
\item $\sigma(-Id)$ acts by the scalar $\epsilon^{p}\cdot(-1)^{\sum_{i=1}^{m}a_{i}}$
for the identity matrix $Id\in O(p)$.
\end{itemize}
An $O(p)\times O(q)$-type is parametrized as $(a_{1},a_{2},\dots,a_{[\frac{p}{2}]};\epsilon)\otimes(b_{1},b_{2},\dots,b_{[\frac{q}{2}]};\eta)$
with integers $a_{1}\geqslant a_{2}\geqslant\cdots\geqslant a_{[\frac{p}{2}]}\geqslant0$,
$b_{1}\geqslant b_{2}\geqslant\cdots\geqslant b_{[\frac{q}{2}]}\geqslant0$,
and $(\epsilon,\eta)\in\{\pm1\}\times\{\pm1\}$. We refer to $(a_{1},a_{2},\dots,a_{[\frac{p}{2}]};b_{1},b_{2},\dots,b_{[\frac{q}{2}]})$
as its \emph{highest weight}, and to $(\epsilon,\eta)$ as its \emph{signs}.

\subsection{Degree of $K$-types}

Consider the dual pair $(G,G')=(O(p,q),Sp(2n,\mathbb{R}))$ with $p+q$
even. Howe \cite{Howe1989transcending} defines a \emph{degree} for
each $K$-type $\sigma$ for $G$ or $G'$ occurring in the Fock space
$\mathcal{F}$ of $\omega$, which is the minimal degree of polynomials
in the $\sigma$-isotypic subspace of $\mathcal{F}$. He also points
out a correspondence of $K$-types in the space $\mathcal{H}$ of
joint harmonics (a certain subspace of $\mathcal{F}$ associated to
the dual pair). Let us describe the degrees and this correspondence
explicitly.
\begin{prop}
[{\cite[Prop.4]{Paul2005howe}}] \label{prop:CorKT} Let $\sigma=(a_{1},\dots,a_{x},0,\dots,0;\epsilon)\otimes(b_{1},\dots,b_{y},$
$0,$ $\dots,$$0;\eta)$ be a $K$-type for $O(p,q)$, with $a_{x}\geqslant1$
and $b_{y}\geqslant1$. Let $\sigma'$ be a $K$-type for $Sp(2n,\mathbb{R})$
written as $\sigma'=(\frac{p-q}{2},\dots,\frac{p-q}{2})+(c_{1},c_{2},\dots,c_{n})$.

(1) The $O(p)\times O(q)$-type $\sigma$ occurs in $\mathcal{H}$
if and only if $n\geqslant x+y+\frac{1-\epsilon}{2}(p-2x)+\frac{1-\eta}{2}(q-2y)$.
In that case, $\sigma$ corresponds to the $U(n)$-type
\[
(\frac{p-q}{2},\dots,\frac{p-q}{2})+(a_{1},\dots,a_{x},\underbrace{1,\dots,1}_{\frac{1-\epsilon}{2}(p-2x)},0,\dots,0,\underbrace{-1,\dots,-1}_{\frac{1-\eta}{2}(q-2y)},-b_{y},\dots,-b_{2},-b_{1}).
\]

(2) The $U(n)$-type $\sigma'$ occurs in $\mathcal{H}$ if and only
if
\[
2\#\{i\mid c_{i}\geqslant2\}+\#\{i\mid c_{i}=1\}\leqslant p\quad\text{and}\quad2\#\{i\mid c_{i}\leqslant-2\}+\#\{i\mid c_{i}=-1\}\leqslant q.
\]

(3) If $\sigma$ occurs in $\mathcal{F}$, then its degree is $\deg(\sigma)=\sum_{i=1}^{x}a_{i}+\sum_{i=1}^{y}b_{i}+\frac{1-\epsilon}{2}(p-2x)+\frac{1-\eta}{2}(q-2y)$.

(4) If $\sigma'$ occurs in $\mathcal{F}$, then its degree is $\deg(\sigma')=\sum_{i=1}^{n}|c_{i}|$.\end{prop}
\begin{rem*}
For $(O(p,q),Sp(2n,\mathbb{R}))$, the degree of a $K$-type for $O(p,q)$
is independent of $n$, and the degree of a $K$-type for $Sp(2n,\mathbb{R})$
depends only on the difference $p-q$.
\end{rem*}
Write the correspondence of $K$-types in $\mathcal{H}$ as 
\[
\xymatrix{\{O(p)\times O(q)\text{-types in }\mathcal{H}\}\ar@<1ex>[rr]^{\quad\ \phi_{n}} &  & \{U(n)\text{-types in }\mathcal{H}\}.\ar@<1ex>[ll]^{\quad\ \phi_{p,q}}}
\]

\begin{itemize}
\item For an $O(p)\times O(q)$-type $\sigma$ and a $U(n)$-type $\sigma'$,
if they correspond to each other in $\mathcal{H}$, write $\phi_{n}(\sigma)=\sigma'$
and $\phi_{p,q}(\sigma')=\sigma$. In that case, $\deg(\sigma)=\deg(\sigma')$.
\item If $\sigma$ (resp. $\sigma'$) does not occur in $\mathcal{H}$,
write $\phi_{n}(\sigma)=0$ (resp. $\phi_{p,q}(\sigma')=0$).
\end{itemize}
Suppose that $\pi\in\mathcal{R}(G)$ and $\pi'\in\mathcal{R}(G')$
correspond to each other in the local theta correspondence. Let $\mathcal{D}(\pi)$
denote the set of $K$-types for $G$ which is of minimal degree in
$\pi$, and similarly for $\pi'$. Define the \emph{degree} of $\pi$
as $\deg(\pi)=\deg(\sigma)$ for any $\sigma\in\mathcal{D}(\pi)$,
and similarly for $\pi'$.
\begin{lem}
[\cite{Howe1989transcending}] \label{lem:min-deg} If  $\pi\in\mathcal{R}(G,\omega)$,
then  $\phi_{n}(\mathcal{D}(\pi))=\mathcal{D}(\theta_{n}(\pi))$
and $\deg(\pi)=\deg(\theta_{n}(\pi))$.
\end{lem}
About the degrees of $K$-types, the following lemma will be quite
useful.
\begin{lem}
\label{lem:parity} If $\pi\in\mathcal{R}(G,\omega)$ contains two
$K$-types $\sigma_{1}$ and $\sigma_{2}$, then $\deg(\sigma_{1})\equiv\deg(\sigma_{2})\ (\mathrm{mod}\ 2)$.\end{lem}
\begin{proof}
Consider the $(\mathfrak{g},K)$-module of $\pi$. Note that $\mathcal{F}=\mathcal{F}_{0}\oplus\mathcal{F}_{1}$,
where $\mathcal{F}_{i}$ (for $i\in\{0,1\}$) is the linear span of
all homogeneous polynomials in $\mathcal{F}$ with degree $\equiv i\ (\mathrm{mod}\ 2)$.
Each $\mathcal{F}_{i}$ is $(\mathfrak{g},K)$-invariant, so the $(\mathfrak{g},K)$-module
of $\pi$ is equivalent to an irreducible $(\mathfrak{g},K)$-quotient
of either $\mathcal{F}_{0}$ or $\mathcal{F}_{1}$. Thus one of $\mathcal{F}_{0}$
and $\mathcal{F}_{1}$ contains both $\sigma_{1}$ and $\sigma_{2}$.

Suppose $\deg(\sigma_{1})\not\equiv\deg(\sigma_{2})\ (\mathrm{mod}\ 2)$.
Then one of $\sigma_{1}$ and $\sigma_{2}$, denoted by $\sigma$,
must occur in both $\mathcal{F}_{0}$ and $\mathcal{F}_{1}$. Let
$\mathcal{F}(d)$ denote the subspace of homogeneous polynomials in
$\mathcal{F}$ with degree $d$. Since the action of $K$ preserves
the degree, $\mathcal{F}(d)$ is $K$-invariant. Write $\mathcal{F}_{\sigma}$
for the $\sigma$-isotypic subspace of $\mathcal{F}$. Then $\mathcal{F}(d)_{\sigma}=\mathcal{F}_{\sigma}\cap\mathcal{F}(d)$
is the $\sigma$-isotypic subspace of $\mathcal{F}(d)$. For $i\in\{0,1\}$,
let 
\[
d_{i}=\min\{d\mid\mathcal{F}(d)_{\sigma}\neq0,\ d\equiv i\ (\mathrm{mod}\ 2)\}
\]
Recall the space of $K$-harmonics (see \cite{Howe1989transcending})
\[
\mathcal{H}(K)=\{P\in\mathcal{F}\mid l(P)=0\text{ for all }l\in\mathfrak{m}^{(0,2)}\},
\]
where $\mathfrak{m}^{(0,2)}=\mathfrak{m}\cap\mathfrak{sp}^{(0,2)}$.
Here $\mathfrak{sp}$ is the complexified Lie algebra of $\mathbf{Sp}=Sp(2n(p+q),\mathbb{R})$,
$\mathfrak{m}$ is the complexified Lie algebra of the centralizer
$M$ of $K$ in $\mathbf{Sp}$, and $\mathfrak{sp}^{(0,2)}$ acts
as the span of $\frac{\partial^{2}}{\partial z_{i}\partial z_{j}}$.
Since $\mathfrak{m}$ and $K$ commute, $\mathcal{H}(K)$ is $K$-stable,
and $\mathcal{F}{}_{\sigma}$ is $\mathfrak{m}$-stable. The action
of $\mathfrak{sp}^{(0,2)}$ reduces the degree of polynomials by $2$,
so the action of $\mathfrak{m}^{(0,2)}$ on $\mathcal{F}(d_{i})_{\sigma}$
gives results in $\mathcal{F}_{\sigma}\cap\mathcal{F}(d_{i}-2)=\mathcal{F}(d_{i}-2)_{\sigma}=0$.
Thus $\mathcal{F}(d_{i})_{\sigma}\subseteq\mathcal{H}(K)$ for both
$i\in\{0,1\}$. So $\mathcal{F}(d_{0})_{\sigma}\cup\mathcal{F}(d_{1})_{\sigma}\subseteq\mathcal{H}(K)_{\sigma}$,
$\mathcal{F}(d_{0})_{\sigma}\cap\mathcal{F}(d_{1})_{\sigma}=\O$,
$\mathcal{F}(d_{0})_{\sigma}\neq\O$ and $\mathcal{F}(d_{1})_{\sigma}\neq\O$,
in contradiction with a result in classical invariant theory (\cite[(3.9)(b)]{Howe1989transcending})
that $\mathcal{H}(K)_{\sigma}=\mathcal{F}(d)_{\sigma}$ for $d=\deg(\sigma)=\min\{d_{0},d_{1}\}$.
\end{proof}

\subsection{Lowest $K$-types}

Let $||\cdot||$ be the norm of $K$-types in the sense of \cite{Vogan1979algebraic}:
a $K$-type $\sigma$ has norm $||\sigma||=\langle\sigma+2\rho_{c},\sigma+2\rho_{c}\rangle$
where $2\rho_{c}$ is the sum of all positive compact roots in $\Delta_{c}^{+}$.

\begin{prop}
If $\sigma=(a_{1},\dots,a_{n})\in\widehat{U(n)}$, then $||\sigma||=\sum_{i=1}^{n}(a_{i}+n+1-2i)^{2}$.\end{prop}
\begin{proof}
As $\Delta_{c}^{+}$ $=$ $\{e_{i}-e_{j}\mid1\leqslant i<j\leqslant n\}$,
\[
2\rho_{c}=\sum_{1\leqslant i<j\leqslant n}(e_{i}-e_{j})=\sum_{i=1}^{n}((n-i)-(i-1)))e_{i}
\]
So $\sigma+2\rho_{c}$ $=$ $\sum_{i=1}^{n}(a_{i}+n+1-2i)e_{i}$,
and $||\sigma||$ $=$ $\sum_{i=1}^{n}(a_{i}+n+1-2i)^{2}$.\end{proof}
\begin{prop}
If $\sigma=(a_{1},\dots,a_{\left[\frac{p}{2}\right]};\epsilon)\otimes(b_{1},\dots,b_{\left[\frac{p}{2}\right]};\eta)$
is an $O(p)\times O(q)$-type with $p+q$ even, then $||\sigma||$
$=$ $\sum_{i=1}^{\left[\frac{p}{2}\right]}(a_{i}+p-2i)^{2}+\sum_{i=1}^{\left[\frac{q}{2}\right]}(b_{i}+q-2i)^{2}$.\end{prop}
\begin{proof}
If $p$, $q$ are both even,
\begin{eqnarray*}
2\rho_{c} & = & \sum_{1\leqslant i<j\leqslant\frac{p}{2}}((e_{i}+e_{j})+(e_{i}-e_{j}))+\sum_{1\leqslant i<j\leqslant\frac{q}{2}}((f_{i}+f_{j})+(f_{i}-f_{j}))\\
 & = & \sum_{i=1}^{\frac{p}{2}}(p-2i)e_{i}+\sum_{i=1}^{\frac{q}{2}}(q-2i)f_{i}.
\end{eqnarray*}
 If $p$, $q$ are both odd,
\begin{eqnarray*}
2\rho_{c} & = & \sum_{1\leqslant i<j\leqslant\left[\frac{p}{2}\right]}((e_{i}+e_{j})+(e_{i}-e_{j}))+\sum_{1\leqslant i<j\leqslant\left[\frac{q}{2}\right]}((f_{i}+f_{j})+(f_{i}-f_{j}))+\sum_{i=1}^{\left[\frac{p}{2}\right]}e_{i}+\sum_{i=1}^{\left[\frac{q}{2}\right]}f_{i}\\
 & = & \sum_{i=1}^{\left[\frac{p}{2}\right]}(p-2i)e_{i}+\sum_{i=1}^{\left[\frac{q}{2}\right]}(q-2i)f_{i}
\end{eqnarray*}
In both cases, $||\sigma||=\langle\sigma+2\rho_{c},\sigma+2\rho_{c}\rangle=\sum_{i=1}^{\left[\frac{p}{2}\right]}(a_{i}+p-2i)^{2}+\sum_{i=1}^{\left[\frac{q}{2}\right]}(b_{i}+q-2i)^{2}$. 
\end{proof}
Let $H$ be a reductive Lie group with a maximal compact subgroup
$K$. Let $\pi$ be an admissible representation of $H$. If $\sigma$
occurs with minimal norm in $\pi$, we call $\sigma$ a \emph{lowest
$K$-type}  of $\pi$. Let $\mathcal{A}(\pi)$ denote the set of
all lowest $K$-types  of $\pi$.

For $\pi\in\mathcal{R}(Sp(2n,\mathbb{R}))$ or $\mathcal{R}(O(p,q))$
with $p+q$ even, there are explicit algorithms to calculate $\mathcal{A}(\pi)$
from Langlands parameters.   which are Proposition 6, 10, and 13
of \cite{Paul2005howe}. We quote them as Proposition \ref{prop:LKT-Sp},
\ref{prop:LKT-O}, and \ref{prop:LKT-sgn} in Appendix \ref{App:CompLKT}.

\subsection{\label{subsec:Par-Sp}Parametrization for $\mathcal{R}(Sp(2n,\mathbb{R}))$}

To ``explicitly'' describe or study the local theta correspondence,
we have to use some parametrization of the representations. In this
paper, we choose Vogan's version of Langlands classification in \cite{Vogan1984unitarizability},
as described in \cite{Paul2005howe}. Irreducible admissible representations
of $Sp(2n,\mathbb{R})$ are parametrized by \emph{Langlands parameters}
$(\lambda_{d},\Psi,\mu,\nu,\varepsilon,\kappa)$ as follows.

Let $(W,\langle,\rangle)$ be a symplectic space over $\mathbb{R}$
of dimension $2n$ with the isometry group $Sp(W)\cong Sp(2n,\mathbb{R})$.
Let 
\[
\{0\}\subset W_{1}\subset W_{2}\subset\cdots\subset W_{r}
\]
be an isotropic flag in $W$, with $\dim W_{i}=d_{i}$. Set $d_{0}=0$,
and $n_{i}=d_{i}-d_{i-1}$. Let $P$ be the stabilizer of this flag
in $Sp(W)$. Then $P=MAN$ is a parabolic subgroup of $Sp(W)$ with
Levi factor 
\[
MA\cong Sp(2(n-d_{r}),\mathbb{R})\times\prod_{i=1}^{r}GL(n_{i},\mathbb{R}).
\]
Especially, let $(n_{1},\dots,n_{r})=(\underbrace{2,\dots,2}_{s},\underbrace{1,\dots,1}_{t})$
and $n-d_{r}=n-2s-t=v$. Then 
\[
MA\cong Sp(2v,\mathbb{R})\times GL(2,\mathbb{R})^{s}\times GL(1,\mathbb{R})^{t}.
\]
Take $\lambda_{d}\in\mathbb{Z}^{v}$, $\mu\in(\mathbb{Z}_{\geqslant0})^{s}$,
$\nu\in\mathbb{C}^{s}$, $\varepsilon\in\{\pm1\}^{t}$, and $\kappa\in\mathbb{C}^{t}$
subject to the following conditions.
\begin{itemize}
\item Let $(\lambda_{d},\Psi)$ parametrize a limit of discrete series $\rho=\rho(\lambda_{d},\Psi)$
of $Sp(2v,\mathbb{R})$. Here $\lambda_{d}$ is the Harish-Chandra
parameter of $\rho$, of the form
\begin{eqnarray*}
\lambda_{d} & = & (\underbrace{a_{1},\dots,a_{1}}_{k_{1}},\underbrace{a_{2},\dots,a_{2}}_{k_{2}},\dots,\underbrace{a_{b},\dots,a_{b}}_{k_{b}},\underbrace{0,\dots,0}_{z},\\
 &  & \ \ \underbrace{-a_{b},\dots,-a_{b}}_{l_{b}},\dots,\underbrace{-a_{2},\dots,-a_{2}}_{l_{2}},\underbrace{-a_{1},\dots,-a_{1}}_{l_{1}}),
\end{eqnarray*}
with integers $a_{1}>a_{2}>\cdots>a_{b}>0$, and $|k_{i}-l_{i}|\leqslant1$
for all $i$. Furthermore, $\Psi$ is a system of positive roots for
$Sp(2v,\mathbb{R})$ containing the standard set of positive compact
roots $\Delta_{c}^{+}=\{e_{i}-e_{j}\mid1\leqslant i<j\leqslant v\}$,
such that $\lambda_{d}$ is dominant with respect to $\Psi$, and
satisfies the condition (F-1): 
\[
\text{(F-1)}\quad\text{if }\alpha\in\Delta_{c}^{+}\text{ is a simple root in }\Psi\text{, then }\langle\lambda_{d},\alpha\rangle>0.
\]

\item Let $(\mu,\nu)$ parametrize a relative limit of discrete series
\[
\tau=\tau(\mu,\nu)=\tau(\mu_{1},\nu_{1})\otimes\tau(\mu_{2},\nu_{2})\otimes\cdots\otimes\tau(\mu_{s},\nu_{s})
\]
of $GL(2,\mathbb{R})^{s}$, where $\mu=(\mu_{1},\mu_{2},\dots\mu_{s})\in(\mathbb{Z}_{\geqslant0})^{s}$,
$\nu=(\nu_{1},\nu_{2},\dots,\nu_{s})\in\mathbb{C}^{s}$, and $\tau(\mu_{i},\nu_{i})$
is the relative limit of discrete series of $GL(2,\mathbb{R})$ with
infinitesimal character $(\frac{1}{2}(\mu_{i}+\nu_{i}),\frac{1}{2}(-\mu_{i}+\nu_{i}))$
and  lowest $K$-type $(\mu_{i}+1;1)$.
\item Let $(\varepsilon,\kappa)$ parametrize a character 
\[
\chi=\chi(\varepsilon,\kappa)=\chi(\varepsilon_{1},\kappa_{1})\otimes\chi(\varepsilon_{2},\kappa_{2})\otimes\cdots\otimes\chi(\varepsilon_{t},\kappa_{t}).
\]
 of $GL(1,\mathbb{R})^{t}$, where $\varepsilon=(\varepsilon_{1},\varepsilon_{2},\dots,\varepsilon_{t})\in\{\pm1\}^{t}$,
$\kappa=(\kappa_{1},\kappa_{2},\dots,\kappa_{t})\in\mathbb{C}^{t}$,
and $\chi(\varepsilon_{j},\kappa_{j})$ is the character of $GL(1,\mathbb{R})$
that maps $x$ to $\mathrm{sgn}(x)^{\frac{1-\varepsilon_{j}}{2}}|x|^{\kappa_{j}}$.
\end{itemize}
From above parameters, we obtain a standard module  $\mathrm{Ind}_{MAN}^{Sp(2n,\mathbb{R})}(\rho\otimes\tau\otimes\chi\otimes\tri_{N})$,
where $\tri_{N}$ is the trivial representation of $N$. Here the
induction is normalized so that infinitesimal characters are preserved. 
\begin{lem}
[{\cite{Vogan1984unitarizability,Paul2005howe}}] \label{lem:non-par-Sp}
(1) Let the parameters $(\lambda_{d},\Psi,\mu,\nu,\varepsilon,\kappa)$
be chosen as above, and satisfy the ``non-parity condition (F-2)''
for $Sp(2n,\mathbb{R})$:
\begin{align*}
 & \text{for }1\leqslant i\leqslant s\text{, if }\nu_{i}=0\text{, then }\mu_{i}\text{ is odd;}\\
 & \text{for }1\leqslant i,j\leqslant t\text{, if }\kappa_{i}=\pm\kappa_{j}\text{, then }\varepsilon_{i}=\varepsilon_{j}\text{;}\\
 & \text{for }1\leqslant i\leqslant t\text{, if }\kappa_{i}=0\text{, then }\varepsilon_{i}=(-1)^{v}\text{.}
\end{align*}
Then for a well chosen $N$, the standard module $\mathrm{Ind}_{MAN}^{Sp(2n,\mathbb{R})}(\rho\otimes\tau\otimes\chi\otimes\tri_{N})$
has a unique irreducible quotient, denoted by $\pi(\lambda_{d},\Psi,\mu,\nu,\varepsilon,\kappa)$.

(2) Each irreducible admissible representation of $Sp(2n,\mathbb{R})$
is infinitesimally equivalent to some $\pi(\lambda_{d},\Psi,\mu,\nu,\varepsilon,\kappa)$
obtained in (1).

(3) $\pi(\lambda_{d},\Psi,\mu,\nu,\varepsilon,\kappa)$ and $\pi(\lambda'_{d},\Psi',\mu',\nu',\varepsilon',\kappa')$
are infinitesimally equivalent if and only if: $\lambda_{d}=\lambda'_{d}$,
$\Psi=\Psi'$, $(\mu',\nu')$ is obtained from $(\mu,\nu)$ by a simultaneous
permutation of the coordinates of $\mu$ and $\nu$, and by possibly
multiplying some of the entries of $\nu$ by $-1$, and similarly
$(\varepsilon',\kappa')$ is obtained from $(\varepsilon,\kappa)$
by a simultaneous permutation of the coordinates of $\varepsilon$
and $\kappa$, and by possibly multiplying some of the entries of
$\kappa$ by $-1$. \end{lem}
\begin{rem*}
Parameters $\lambda_{d}$, $\mu$, $\nu$, $\varepsilon$, $\kappa$
that do not occur will be written as $0$. To avoid confusion, each
occurring parameter must be written in the form of $(\cdot)$, even
if it has only one entry. Thus ``$\kappa=0$'' and ``$\kappa=(0)$''
have totally different meanings.
\end{rem*}

\subsection{\label{subsec:Par-O}Parametrization for $\mathcal{R}(O(p,q))$}

$\mathcal{R}(O(p,q))$ with $p+q$ even is parametrized in the same
way as in \cite{Paul2005howe}. This parametrization is similar to
that in the last subsection, but with two more parameters $(\xi,\zeta)$
about the signs of lowest $K$-types. Irreducible admissible representations
of $O(p,q)$ (with $p+q$ even) are parametrized by Langlands parameters
 $(\zeta,\lambda_{d},\xi,\Psi,\mu,\nu,\varepsilon,\kappa)$ as follows.

Let $(V,(,))$ be a real vector space with nondegenerate symmetric
bilinear form $(,)$ of signature $(p,q)$, with the isometry group
$O(V)\cong O(p,q)$. Let 
\[
\{0\}\subset V_{1}\subset V_{2}\subset\cdots\subset V_{r}
\]
be an isotropic flag in $V$, with $\dim V_{i}=d_{i}$. Set $d_{0}=0$,
and $n_{i}=d_{i}-d_{i-1}$. Let $P$ be the stabilizer of this flag
in $O(V)$. Then $P=MAN$ is a parabolic subgroup of $O(V)$ with
Levi factor 
\[
MA\cong O(p-d_{r},q-d_{r})\times\prod_{i=1}^{r}GL(n_{i},\mathbb{R}).
\]
Especially, let $(n_{1},\dots,n_{r})=(\underbrace{2,\dots,2}_{s},\underbrace{1,\dots,1}_{t})$,
and let $p-d_{r}=p-2s-t$ be even. Write $p-d_{r}=2a$, $q-d_{r}=2d$
with $a,d\in\mathbb{Z}$. Then 
\[
MA\cong O(2a,2d)\times GL(2,\mathbb{R})^{s}\times GL(1,\mathbb{R})^{t}.
\]
Take $\lambda_{d}\in\mathbb{Z}^{a+d}$, $\mu\in(\mathbb{Z}_{\geqslant0})^{s}$,
$\nu\in\mathbb{C}^{s}$, $\varepsilon\in\{\pm1\}^{t}$, $\kappa\in\mathbb{C}^{t}$,
$\xi\in\{\pm1\}$, and $\zeta\in\{\pm1\}$ subject to the following
conditions.
\begin{itemize}
\item Let $(\lambda_{d},\xi,\Psi)$ parametrize a limit of discrete series
$\rho=\rho(\lambda_{d},\xi,\Psi)$ of $O(2a,$ $2d)$. Here $\lambda_{d}$
is the Harish-Chandra parameter of $\rho$, of the form 
\begin{eqnarray*}
\lambda_{d} & = & (\underbrace{a_{1},\dots,a_{1}}_{k_{1}},\underbrace{a_{2},\dots,a_{2}}_{k_{2}},\dots,\underbrace{a_{b},\dots,a_{b}}_{k_{b}},\underbrace{0,\dots,0}_{z};\\
 &  & \ \ \underbrace{a_{1},\dots,a_{1}}_{l_{1}},\underbrace{a_{2},\dots,a_{2}}_{l_{2}},\dots,\underbrace{a_{b},\dots,a_{b}}_{l_{b}},\underbrace{0,\dots,0}_{z'}),
\end{eqnarray*}
with integers $a_{1}>a_{2}>\cdots>a_{b}>0$, $|k_{i}-l_{i}|\leqslant1$
for all $i$, and $|z-z'|\leqslant1$. Furthermore, $\Psi$ is a system
of positive roots for $O(2a,2d)$ containing the standard set of positive
compact roots
\[
\Delta_{c}^{+}=\{e_{i}\pm e_{j}\mid1\leqslant i<j\leqslant a\}\cup\{f_{i}\pm f_{j}\mid1\leqslant i<j\leqslant d\}
\]
such that $\lambda_{d}$ is dominant with respect to $\Psi$, and
satisfies the condition (F-1):
\[
\text{(F-1)}\quad\text{if }\alpha\in\Delta_{c}^{+}\text{ is a simple root in }\Psi\text{, then }\langle\lambda_{d},\alpha\rangle>0.
\]
Indeed, the Harish-Chandra parameter $\lambda_{d}$ and positive root
system $\Psi$ determine a limit of discrete series of $SO(2a,2d)$
denoted by $\rho(\lambda_{d},\Psi)$. When $z+z'=0$, $\mathrm{Ind}_{SO(2a,2d)}^{O(2a,2d)}\rho(\lambda_{d},\Psi)$
is irreducible, and thus is a limit of discrete series of $O(2a,2d)$
denoted by $\rho(\lambda_{d},1,\Psi)$ (and let $\xi=1$ in that case).
 When $z+z'>0$, $\mathrm{Ind}_{SO(2a,2d)}^{O(2a,2d)}\rho(\lambda_{d},\Psi)$
is the direct sum of two limits of discrete series of $O(2a,2d)$,
precisely one of which has the lowest $K$-type with signs $(1;1)$,
denoted by $\rho(\lambda_{d},1,\Psi)$, and the other one by $\rho(\lambda_{d},-1,\Psi)$.
\item Let $(\mu,\nu)$ parametrize a relative limit of discrete series
\[
\tau=\tau(\mu,\nu)=\tau(\mu_{1},\nu_{1})\otimes\tau(\mu_{2},\nu_{2})\otimes\cdots\otimes\tau(\mu_{s},\nu_{s})
\]
of $GL(2,\mathbb{R})^{s}$, where $\mu=(\mu_{1},\mu_{2},\dots\mu_{s})\in(\mathbb{Z}_{\geqslant0})^{s}$,
$\nu=(\nu_{1},\nu_{2},\dots,\nu_{s})\in\mathbb{C}^{s}$, and $\tau(\mu_{i},\nu_{i})$
is the relative limit of discrete series of $GL(2,\mathbb{R})$ with
infinitesimal character $(\frac{1}{2}(\mu_{i}+\nu_{i}),\frac{1}{2}(-\mu_{i}+\nu_{i}))$
and lowest $K$-type $(\mu_{i}+1;1)$.
\item Let $(\varepsilon,\kappa)$ parametrize a character 
\[
\chi=\chi(\varepsilon,\kappa)=\chi(\varepsilon_{1},\kappa_{1})\otimes\chi(\varepsilon_{2},\kappa_{2})\otimes\cdots\otimes\chi(\varepsilon_{t},\kappa_{t}).
\]
 of $GL(1,\mathbb{R})^{t}$, where $\varepsilon=(\varepsilon_{1},\varepsilon_{2},\dots,\varepsilon_{t})\in\{\pm1\}^{t}$,
$\kappa=(\kappa_{1},\kappa_{2},\dots,\kappa_{t})\in\mathbb{C}^{t}$,
and $\chi(\varepsilon_{j},\kappa_{j})$ is the character of $GL(1,\mathbb{R})$
that maps $x$ to $\mathrm{sgn}(x)^{\frac{1-\varepsilon_{j}}{2}}|x|^{\kappa_{j}}$.
\end{itemize}
From above parameters, we obtain a standard module $\mathrm{Ind}_{MAN}^{O(p,q)}(\rho\otimes\tau\otimes\chi\otimes\tri_{N})$
with normalized induction, where $\tri_{N}$ is the trivial representation
of $N$.
\begin{lem}
[{\cite{Vogan1984unitarizability,Paul2005howe}}] \label{lem:non-par-O}
(1) Let the parameters $(\lambda_{d},\xi,\Psi,\mu,\nu,\varepsilon,\kappa)$
be chosen as above, and satisfy the ``non-parity condition (F-2)''
 for $O(p,q)$ with $p+q$ even:
\begin{align*}
 & \text{for }1\leqslant i\leqslant s\text{, if }\nu_{i}=0\text{, then }\mu_{i}\text{ is odd;}\\
 & \text{for }1\leqslant i,j\leqslant t\text{, if }\kappa_{i}=\pm\kappa_{j}\text{, then }\varepsilon_{i}=\varepsilon_{j}.
\end{align*}
Then for a well chosen $N$, the standard module $\mathrm{Ind}_{MAN}^{O(p,q)}(\rho\otimes\tau\otimes\chi\otimes\tri_{N})$
has a unique irreducible quotient, denoted by $\pi_{1}(\lambda_{d},\xi,\Psi,\mu,\nu,\varepsilon,\kappa)$,
if $\lambda_{d}$ contains a zero entry or $\kappa$ contains no zero
entry. Otherwise, we get two irreducible quotients distinguished by
the signs of their lowest $K$-types (as described in Proposition
\ref{prop:LKT-sgn}), denoted by $\pi_{1}(\lambda_{d},1,\Psi,\mu,\nu,\varepsilon,\kappa)$
and $\pi_{-1}(\lambda_{d},1,\Psi,\mu,\nu,\varepsilon,\kappa)$.

(2) Each irreducible admissible representation of $O(p,q)$ is infinitesimally
equivalent to some $\pi_{\zeta}(\lambda_{d},\xi,\Psi,\mu,\nu,\varepsilon,\kappa)$
obtained in (1).

(3) Two such representations $\pi_{\zeta}(\lambda_{d},\xi,\Psi,\mu,\nu,\varepsilon,\kappa)$
and $\pi_{\zeta'}(\lambda'_{d},\xi',\Psi',\mu',\nu',\varepsilon',\kappa')$
are infinitesimally  equivalent if and only if: $\lambda_{d}=\lambda'_{d}$,
$\Psi=\Psi'$, $\xi=\xi'$, $\zeta=\zeta'$, $(\mu',\nu')$ is obtained
from $(\mu,\nu)$ by a simultaneous permutation of the coordinates
of $\mu$ and $\nu$, and by possibly multiplying some of the entries
of $\nu$ by $-1$, and similarly $(\varepsilon',\kappa')$ is obtained
from $(\varepsilon,\kappa)$ by a simultaneous permutation of the
coordinates of $\varepsilon$ and $\kappa$, and by possibly multiplying
some of the entries of $\kappa$ by $-1$.\end{lem}
\begin{rem*}
$\zeta=-1\Rightarrow$ $\lambda_{d}$ contains no zero entry $\Rightarrow$
 $\xi=1$. So $(\xi,\zeta)\neq(-1,-1)$.
\end{rem*}

\subsection{\label{sub:InfCh}Infinitesimal character}

Let $G$ be a reductive Lie group with complexified Lie algebra $\mathfrak{g}$,
and $Z(\mathfrak{g})$ the center of the universal enveloping algebra
$U(\mathfrak{g})$. For $\pi\in\mathcal{R}(G)$ and $z\in Z(\mathfrak{g})$,
Schur's Lemma asserts that $\pi(z)$ acts by a scalar $\lambda(z)\in\mathbb{C}$.
Then $z\mapsto\lambda(z)$ gives a homomorphism of algebras $\lambda:Z(\mathfrak{g})\to\mathbb{C}$,
which is called the infinitesimal character of $\pi$.

Let $\mathfrak{h}$ be a Cartan subalgebra of $\mathfrak{g}$. Let
$W$ be the Weyl group $\mathrm{N}_{G}(\mathfrak{h})/\mathrm{Z}_{G}(\mathfrak{h})$.
Via the Harish-Chandra isomorphism $Z(\mathfrak{g})\simeq S(\mathfrak{h})^{W}$
(the set of $W$-fixed elements of the symmetric algebra $S(\mathfrak{h})$),
the infinitesimal character of $\pi$ is written as an element of
$\mathfrak{h}^{*}$, up to the action of $W$.
\begin{prop}
[\cite{Paul2005howe}] \label{prop:inf.char} The infinitesimal
character of $\pi(\lambda_{d},\Psi,\mu,\nu,\varepsilon,\kappa)\in\mathcal{R}(Sp(2n,\mathbb{R}))$
or $\pi_{\zeta}(\lambda_{d},\xi,\Psi,\mu,\nu,\varepsilon,\kappa)\in\mathcal{R}(O(p,q))$
with $p+q$ even is
\[
(\lambda_{d}\mid(\frac{\mu_{1}+\nu_{1}}{2},\dots,\frac{\mu_{s}+\nu_{s}}{2},\frac{-\mu_{1}+\nu_{1}}{2},\dots,\frac{-\mu_{s}+\nu_{s}}{2},\kappa_{1},\dots,\kappa_{t})),
\]
up to the action of $W$ which consists of all permutations and sign-changes
of coordinates.
\end{prop}
\cite{Przebinda1996duality} gives the duality correspondence of infinitesimal
characters under the local theta correspondence for all reductive
dual pairs over $\mathbb{R}$.
\begin{prop}
[\cite{Przebinda1996duality}] \label{prop:CrspIF} Let $\pi\in\mathcal{R}(O(p,q))$
with $p+q$ even corresponds to $\pi'\in\mathcal{R}(Sp(2n,\mathbb{R}))$
in the local theta correspondence. Write the infinitesimal character
of $\pi$ as $x=(x_{1},\dots,x_{m})\in\mathbb{C}^{m}$ with $m=\frac{p+q}{2}$,
and that of $\pi'$ as $y=(y_{1},\dots,y_{n})\in\mathbb{C}^{n}$.
Let $\sim$ denote the equivalence up to permutations and sign-changes
of coordinates.

(1) If $m=n$, then $x\sim y$.

(2) If $m>n$, then $x\sim(y\mid(0,1,2,\dots,m-n-1))$. 

(3) If $m<n$, then $y\sim(x\mid(1,2,3,\dots,n-m))$.
\end{prop}

\section{Explicit induction principle}

The \emph{explicit induction principle} for $(O(p,q),Sp(2n,\mathbb{R}))$
with $p+q$ even is formulated in this section.

\subsection{Explicit induction principle on $n$}

Let $\pi\in\mathcal{R}(O(p,q))$ with $p+q$ even and $\theta_{n}(\pi)=\pi'\in\mathcal{R}(Sp(2n,\mathbb{R}))$.
Let $(W',\langle,\rangle)$ be a symplectic space over $\mathbb{R}$
of dimension $2(n+1)$ with the isometry group $Sp(W')\cong Sp(2(n+1),\mathbb{R})$.
Take an isotropic subspace $W_{0}$ of dimension $1$ in $W'$. 
Then the stabilizer of $W_{0}$ in $Sp(W')$ is a parabolic subgroup
$M'A'N'$ with Levi factor
\[
M'A'\cong Sp(2n,\mathbb{R})\times GL(1,\mathbb{R}).
\]
 
\begin{lem}
[{\cite[Th.30(1)]{Paul2005howe}}] There exists a nontrivial $O(p,q)\times Sp(2(n+1),\mathbb{R})$-map
(on the level of Harish-Chandra modules)
\[
\omega\longrightarrow\pi\otimes\mathrm{Ind}_{M'A'N'}^{Sp(2(n+1),\mathbb{R})}(\pi'\otimes\xi'\otimes\tri_{N'}),
\]
where $\xi'$ is the character of $GL(1,\mathbb{R})$ with $\xi'(g)=|\det(g)|^{\frac{p+q}{2}-n-1}\mathrm{sgn}(\det(g))^{\frac{p-q}{2}}$,
and $\tri_{N'}$ is the trivial representation of $N'$. Let $I'$
denote the above normalized induction. Then $\theta_{n+1}(\pi)$ is
an irreducible subquotient of $I'$.
\end{lem}
Let $\theta_{n}(\pi)=\pi'=\pi(\lambda_{d},\Psi,\mu,\nu,\varepsilon,\kappa)$,
and $\boldsymbol{I}'=\boldsymbol{I}(\lambda_{d},\Psi,\mu,\nu,(\varepsilon|((-1)^{\frac{p-q}{2}})),(\kappa|(1+n-\frac{p+q}{2})))$
be the standard module of parabolic induction of $Sp(2(n+1),\mathbb{R})$
with the given Langlands parameters. Then $I'$ is a subquotient of
$\boldsymbol{I}'$, and $\theta_{n+1}(\pi)$ is an irreducible subquotient
of $\boldsymbol{I}'$.

It is well-known that in a standard modules of parabolic induction,
each lowest $K$-type occurs with multiplicity one (cf. \cite{Vogan1979algebraic}).
Therefore, if $\tau\in\mathcal{A}(\boldsymbol{I}')$, then there is
a unique irreducible subquotient of $\boldsymbol{I}'$ containing
$\tau$.

When $p+q\neq2n+2$, let $\pi'_{1}=\pi(\lambda_{d},\Psi,\mu,\nu,(\varepsilon|((-1)^{\frac{p-q}{2}})),(\kappa|(1+n-\frac{p+q}{2})))$
with a possible modification on parameters: if $\pm\kappa_{i}=1+n-\frac{p+q}{2}$
and $\varepsilon_{i}\neq(-1)^{\frac{p-q}{2}}$ for some $i$, delete
these four entries, and add $(\mu_{s+1},\nu_{s+1})=(0,2\kappa_{i})$
into $(\mu,\nu)=((\mu_{1},\dots,\mu_{s}),(\nu_{1},\dots,\nu_{s}))$.
Since $1+n-\frac{p+q}{2}\neq0$, no zero entry is added. So the resulting
parameters satisfy the condition (F-2). Whether modified or not, $\pi'_{1}$
is infinitesimal equivalent to an irreducible subquotient of $\boldsymbol{I}'$
(see \cite{Vogan1981representations,Vogan1984unitarizability}).
\begin{prop}
\label{prop:LKT+1}  If $\theta_{n+1}(\pi)$ contains a lowest $K$-type
of $\pi'_{1}$, then $\theta_{n+1}(\pi)=\pi'_{1}$.\end{prop}
\begin{proof}
The $K$-structure of a standard module depends only on the discrete
Langlands parameters, so $\mathcal{A}(\boldsymbol{I}')=\mathcal{A}(\pi_{1}')$
 by the algorithm in Proposition \ref{prop:LKT-Sp}. Both $\theta_{n+1}(\pi)$
and $\pi_{1}'$ are infinitesimal equivalent to irreducible subquotients
of $\boldsymbol{I}'$, and they contain the same lowest $K$-type
of $\boldsymbol{I}'$ which occurs with multiplicity one, so they
are infinitesimal equivalent to each other.
\end{proof}
Look for a sufficient condition so that $\theta_{n+1}(\pi)$ contains
a lowest $K$-type of $\pi'_{1}$. On the one hand,  $\mathcal{D}(\theta_{n+1}(\pi))=\phi_{n+1}\circ\phi_{p,q}(\mathcal{D}(\pi'))$.
On the other hand, $\mathcal{A}(\pi'_{1})$ can be got from $\mathcal{A}(\pi')$
by a simple operation when $\pi'$ satisfies certain conditions. If
$\mathcal{D}(\pi')\cap\mathcal{A}(\pi')$ is nonempty, and $\phi_{n+1}\circ\phi_{p,q}$
coincides with this operation, then it is done.
\begin{prop}
Let $\sigma'$ be a $U(n)$-type with $\phi_{p,q}(\sigma')\neq0$.
Then $\phi_{n+1}\circ\phi_{p,q}(\sigma')=\sigma'_{1}$ is the $U(n+1)$-type
parametrized by the same entries of $\sigma'$ with one more entry
$\frac{p-q}{2}$ added.\end{prop}
\begin{proof}
Check by the explicit descriptions of $\phi_{n}$ and $\phi_{n+1}$
in Proposition \ref{prop:CorKT}.
\end{proof}
Let $\lambda_{d}$ contain exactly $k(\lambda_{d})$ positive entries,
$l(\lambda_{d})$ negative ones, and $z(\lambda_{d})$ zero ones.
\begin{prop}
\label{prop:cond_lambda} Let $(\lambda_{d},\Psi)$ satisfy one of
the following conditions:

(1) $\frac{p-q}{2}=k(\lambda_{d})-l(\lambda_{d})$;

(2) $\frac{p-q}{2}=k(\lambda_{d})-l(\lambda_{d})+1$, $z(\lambda_{d})>0$,
and $e_{k(\lambda_{d})+1}+e_{k(\lambda_{d})+z(\lambda_{d})}\in\Psi$; 

(3) $\frac{p-q}{2}=k(\lambda_{d})-l(\lambda_{d})-1$, $z(\lambda_{d})>0$,
and $e_{k(\lambda_{d})+1}+e_{k(\lambda_{d})+z(\lambda_{d})}\notin\Psi$.

 Then $\mathcal{A}(\pi'_{1})=\{\sigma'_{1}\mid\sigma'\in\mathcal{A}(\pi')\}$,
where $\sigma'_{1}$ is got from $\sigma'$ by adding an entry $\frac{p-q}{2}$.\end{prop}
\begin{proof}
By the algorithm (Theorem \ref{prop:LKT-Sp}) to calculate $\mathcal{A}(\pi')$
and $\mathcal{A}(\pi'_{1})$ respectively, the only change is to add
one more entry. Note that $k(\lambda_{d})-l(\lambda_{d})=u-r$ in
the algorithm. The added entry is $u-r$ in case (1), $u-r+1$ in
case (2), $u-r-1$ in case (3), which is always $\frac{p-q}{2}$.\end{proof}
\begin{lem}
\label{lem:cond_n} Let $\pi'=\pi(\lambda_{d},\Psi,\mu,\nu,\varepsilon,\kappa)\in\mathcal{R}(Sp(2n,\mathbb{R}))$.
Let $p$ and $q$ be nonnegative integers with $p+q$ even. Consider
the following conditions.

$(i)$ $\theta_{p,q}(\pi')\neq0$, and $p+q\leqslant2n$. 

$(ii)$ $\theta_{p',q'}(\pi')\neq0$ for $p'=p+n-\frac{p+q}{2}$ and
$q'=q+n-\frac{p+q}{2}$.

$(iii)$ $(\lambda_{d},\Psi)$ satisfy the hypothesis of Proposition
\ref{prop:cond_lambda} (concerning only $p-q$). 

$(iv)$ $\mathcal{A}(\pi')\subseteq\mathcal{D}(\pi')$ for the dual
pair $(O(p,q),Sp(2n,\mathbb{R}))$.

Then $(i)\Rightarrow(ii)\Leftrightarrow(iii)\Rightarrow(iv)$.\end{lem}
\begin{proof}
$(i)\Rightarrow(ii)$: By Kudla's Persistence Principle as $n-\frac{p+q}{2}\geqslant0$.

$(ii)\Leftrightarrow(iii)$ From \cite[Th.18]{Paul2005howe} we can
read off the full description of $\theta_{p',q'}$ for $p'+q'=2n$,
and see that $\theta_{p',q'}(\pi')\neq0$ if and only if $(\lambda_{d},\Psi)$
satisfy the hypothesis of Proposition \ref{prop:cond_lambda}.

$(ii)\Rightarrow(iv)$ Note that the degree of a $K$-type for $Sp(2n,\mathbb{R})$
depends only on the difference $p'-q'=p-q$ (by Proposition \ref{prop:CorKT}).
By \cite[Cor.37]{Paul2005howe}, $\mathcal{A}(\pi')\subseteq\mathcal{D}(\pi')$.\end{proof}
\begin{prop}
\label{prop:IndP}If $p+q\neq2n+2$ and $(\lambda_{d},\Psi)$ satisfy
the hypothesis of Proposition \ref{prop:cond_lambda}, then $\theta_{n+1}(\pi)=\pi'_{1}$.\end{prop}
\begin{proof}
Take any $\sigma\in\mathcal{A}(\pi')\subseteq\mathcal{D}(\pi')$,
we get $\sigma'_{1}\in\mathcal{A}(\pi'_{1})\cap\mathcal{D}(\theta_{n+1}(\pi))$.
So $\theta_{n+1}(\pi)$ contains a lowest $K$-type of $\pi'_{1}$,
and $\theta_{n+1}(\pi)=\pi'_{1}$ by Proposition \ref{prop:LKT+1}.\end{proof}
\begin{thm}
[Explicit induction principle on $n$] \label{thm:EIP_n} Let $\pi'=\pi(\lambda_{d},\Psi,\mu,\nu,\varepsilon,\kappa)\in\mathcal{R}(Sp(2n,\mathbb{R}))$,
$\theta_{p,q}(\pi')\neq0$ with even $p+q\neq2n+2$. Suppose that
$\pi'$ satisfies the hypothesis of Proposition \ref{prop:cond_lambda}
 (which is automatically true if $p+q\leqslant2n$ by Lemma \ref{lem:cond_n}).
Then for any positive integer $k$ such that $\frac{p+q}{2}\notin[n+1,n+k]$
(if $p+q\leqslant2n$, then for all $k\geqslant1$),
\begin{align*}
\theta_{n+k}(\theta_{p,q}(\pi')) & =\pi(\lambda_{d},\Psi,\mu,v,(\varepsilon|((-1)^{\frac{p-q}{2}},(-1)^{\frac{p-q}{2}},\dots,(-1)^{\frac{p-q}{2}})),\\
 & \qquad(\kappa|(1+n-\frac{p+q}{2},2+n-\frac{p+q}{2},\dots,k+n-\frac{p+q}{2})))
\end{align*}
with a possible modification: if the resulting parameters contain
some entries $\kappa_{i}=\pm\kappa_{j}$ with $\varepsilon_{i}\neq\varepsilon_{j}$,
delete $\varepsilon_{i}$, $\varepsilon_{j}$, $\kappa_{i}$, $\kappa_{j}$
from $(\varepsilon,\kappa)$, and add entries $(0,2\kappa_{i})$ into
$(\mu,\nu)$.\end{thm}
\begin{proof}
Use Corollary \ref{prop:IndP} repeatedly for $k$ times.
\end{proof}
For $\pi\in\mathcal{R}(O(p,q))$ with $p+q$ even, define the \emph{first
occurrence index}
\[
n(\pi)=\min\{0\leqslant k\in\mathbb{Z}\mid\theta_{k}(\pi)\neq0\}.
\]
  By the nonvanishing of the stable range theta liftings, $n(\pi)\leqslant p+q$.
Therefore, to calculate $\theta_{n}$ for all $n$, by the explicit
induction principle on $n$, it suffices to consider for $1\leqslant n\leqslant p+q$.

\subsection{Explicit induction principle on $(p,q)$}

Let $\pi'\in\mathcal{R}(Sp(2n,\mathbb{R}))$ and $\theta_{p,q}(\pi')=\pi\in\mathcal{R}(O(p,q))$
with $p+q$ even. Let $(V',(,))$ be a real vector space with nondegenerate
symmetric bilinear form $(,)$ of signature $(p+1,q+1)$ with the
isometry group $O(V')\cong O(p+1,q+1)$. Take an isotropic subspace
$V_{0}$ of dimension $1$ in $V'$.  Then the stabilizer of $V_{0}$
in $O(V')$ is a parabolic subgroup $M'A'N'$ of $O(p+1,q+1)$ with
Levi factor
\[
M'A'\cong O(p,q)\times GL(1,\mathbb{R}).
\]

\begin{lem}
[{\cite[Th.30(2)]{Paul2005howe}}] There exists a nontrivial $O(p+1,q+1)\times Sp(2n,\mathbb{R})$-map
(on the level Harish-Chandra modules)
\[
\omega\longrightarrow\mathrm{Ind}_{M'A'N'}^{O(p+1,q+1)}(\pi\otimes\xi\otimes\tri_{N'})\otimes\pi',
\]
where $\xi$ is the character of $GL(1,\mathbb{R})$ with $\xi(g)=|\det(g)|^{n-\frac{p+q}{2}}$,
and $\tri_{N'}$ is the trivial representation of $N'$. Let $I$
denote the above normalized induction. Then $\theta_{p+1,q+1}(\pi')$
is an irreducible subquotient of $I$.
\end{lem}
Let $\theta_{p,q}(\pi')=\pi=\pi_{\zeta}(\lambda_{d},\xi,\Psi,\mu,\nu,\varepsilon,\kappa)$,
and $\boldsymbol{I}=\boldsymbol{I}(\lambda_{d},\xi,\Psi,\mu,\nu,(\varepsilon|(1)),(\kappa|(n-\frac{p+q}{2})))$
be the standard module of parabolic induction of $O(p+1,q+1)$ with
the given Langlands parameters. Then $I$ is a subquotient of $\boldsymbol{I}$,
and $\theta_{p+1,q+1}(\pi')$ is an irreducible subquotient of $\boldsymbol{I}$. 

When $p+q\neq2n$, let $\pi_{1,1}=\pi_{\zeta}(\lambda_{d},\xi,\Psi,\mu,\nu,(\varepsilon|(1)),(\kappa|(n-\frac{p+q}{2})))$
with a possible modification: if $\pm\kappa_{i}=n-\frac{p+q}{2}$
and $\varepsilon_{i}\neq1$ for some $i$, delete these four entries,
and add $(\mu_{s+1},\nu_{s+1})=(0,2\kappa_{i})$ into $(\mu,\nu)=((\mu_{1},\dots,\mu_{s}),(\nu_{1},\dots,\nu_{s}))$.
Since $n-\frac{p+q}{2}\neq0$, no zero entry is added. So the resulting
parameters satisfy the condition (F-2). Whether modified or not,
$\pi_{1,1}$ is infinitesimal equivalent to an irreducible subquotient
of $\boldsymbol{I}$ (see \cite{Vogan1981representations,Vogan1984unitarizability}).
\begin{prop}
\label{prop:LKT+11}  If $\theta_{p+1,q+1}(\pi')$ contains a lowest
$K$-type of $\pi_{1,1}$, then $\theta_{p+1,q+1}(\pi')=\pi_{1,1}$.\end{prop}
\begin{proof}
The $K$-structure of a standard module depends only on the discrete
Langlands parameters, so $\mathcal{A}(\boldsymbol{I})$ is calculated
from $(\lambda_{d},\xi,\Psi,\mu,\nu,(\varepsilon|(1)),$ $(\kappa|(n-\frac{p+q}{2})))$
by the algorithm in Proposition \ref{prop:LKT-O} and \ref{prop:LKT-sgn}
(with all the signs obtained from all choices of $\zeta$). Comparing
to the algorithm to compute $\mathcal{A}(\pi_{1,1})$, we see $\mathcal{A}(\pi_{1,1})\subseteq\mathcal{A}(\boldsymbol{I})$.
Both $\theta_{p+1,q+1}(\pi')$ and $\pi_{1,1}$ are infinitesimal
equivalent to irreducible subquotients of $\boldsymbol{I}$, and they
contain the same lowest $K$-type of $\boldsymbol{I}$ which occurs
with multiplicity one, so they are infinitesimal equivalent to each
other.\end{proof}
\begin{prop}
\label{prop:sig_11}Suppose $\sigma=(a_{1},\dots,a_{x},0,\dots,0;\epsilon)\otimes(b_{1},\dots,b_{x},0,\dots,0;\eta)$
is an $O(p)\times O(q)$-type, with $a_{x}>0$ and $b_{y}>0$. (When
$p=2x$, $\epsilon=\pm1$ give the same $O(p)$-type, but we choose
$\epsilon=1$ for the convenience. We write $\epsilon=-1$ only if
$p>2x$. Similarly, we write $\eta=-1$ only if $q>2y$.) If $\sigma'=\phi_{n}(\sigma)\neq0$.
Then $\phi_{p+1,q+1}(\sigma')=\sigma_{1,1}$, where $\sigma_{1,1}$
is defined as 
\[
\sigma_{1,1}=(a_{1},\dots,a_{x},\frac{1-\epsilon}{2},\underbrace{0,\dots,0}_{[\frac{p+1}{2}]-x-1};\epsilon)\otimes(b_{1},\dots,b_{y},\frac{1-\eta}{2},\underbrace{0,\dots,0}_{[\frac{q+1}{2}]-y-1};\eta).
\]
\end{prop}
\begin{rem*}
Notice that when $p=2x+1$ and $\epsilon=-1$, the resulting $(a_{1},\dots,a_{x},1;-1)$
should be rewritten as $(a_{1},\dots,a_{x},1;1)$ if we want to repeat
this algorithm to get $\phi_{p+2,q+2}(\sigma')$. Similarly for $q=2y+1$
and $\eta=-1$.\end{rem*}
\begin{prop}
\label{prop:cond2} Let $p+q\neq2n$, $\pi=\pi_{\zeta}(\lambda_{d},\xi,\Psi,\mu,\nu,\varepsilon,\kappa)$,
and $\pi_{1,1}$ be as above. Suppose that $\zeta=\xi=1$, and either
$\lambda_{d}$ contains a zero entry or some $(\varepsilon_{i},\kappa_{i})=(1,0)$.
Then for any $\sigma\in\mathcal{A}(\pi)$, we have $\sigma_{1,1}\in\mathcal{A}(\pi_{1,1})$,
where $\sigma_{1,1}$ is obtained from $\sigma$ by the above algorithm.\end{prop}
\begin{proof}
Let $\beta=\#\{i\mid\varepsilon_{i}=1\}$ and $\gamma=\#\{i\mid\varepsilon_{i}=-1\}$.
For any $\sigma\in\mathcal{A}(\pi)$, under the assumption for parameters
of $\pi$, the signs of $\sigma$ is determined by $\beta$, $\gamma$
and the highest weight $(\Lambda_{1};\Lambda_{2})$ of $\sigma$ as
in Proposition \ref{prop:LKT-sgn} of Appendix \ref{App:CompLKT}:

If $\beta\geqslant\gamma$, then the signs of $\sigma$ are $(1,1)$.

If $\beta<\gamma$, then the signs of $\sigma$ are $(1;-1)$ if $\Lambda_{1}$
has more zeros than $\Lambda_{2}$, and $(-1;1)$ otherwise. 

(Note that when $\beta<\gamma$, by Proposition \ref{prop:LKT-O},
$\Lambda_{1}$ or $\Lambda_{2}$ contains no zero entry $\Rightarrow\beta+1=\gamma\Rightarrow p$
and $q$ are odd. So the signs of $\sigma$ are well written as in
the Proposition \ref{prop:sig_11}.) 

Compare the algorithms (Proposition \ref{prop:LKT-O}, \ref{prop:LKT-sgn})
to get $\mathcal{A}(\pi)$ and $\mathcal{A}(\pi_{1,1})$. The only
change is from $\beta$ to $\beta+1$. If $\beta\geqslant\gamma$
for $\sigma$, we get a lowest $K$-type of $\pi_{1,1}$ with the
same nonzero entries. If $\beta<\gamma$ for $\sigma$, we get a lowest
$K$-type of $\pi_{1,1}$ with the same nonzero entries, and one more
entry $1$ in the left or right part that contains less zero entries.
Always the resulting lowest $K$-type of $\pi_{1,1}$ is $\sigma_{1,1}$.

(Note that if $\beta+1=\gamma$, then the the resulting lowest $K$-type
has signs $(1,1)$, which coincides with the remark after Proposition
\ref{prop:sig_11}.)\end{proof}
\begin{lem}
\label{lem:cond_pq} Let $\pi=\pi_{\zeta}(\lambda_{d},\xi,\Psi,\mu,\nu,\varepsilon,\kappa)\in\mathcal{R}(O(p,q))$
with $p+q$ even. Consider the following conditions.

$(i)$ $n(\pi)\leqslant\frac{p+q}{2}-1$.

$(ii)$ $\zeta=\xi=1$, and either $\lambda_{d}$ contains a zero
entry or some $(\varepsilon_{i},\kappa_{i})=(1,0)$.

$(iii)$ $\mathcal{A}(\pi)\subseteq\mathcal{D}(\pi)$.

Then $(i)\Leftrightarrow(ii)\Rightarrow(iii)$.\end{lem}
\begin{proof}
$(i)\Rightarrow(ii)$: Let $n=\frac{p+q}{2}-1\geqslant n(\pi)$. Then
$\theta_{n}(\pi)\neq0$ by Kudla's Persistence Principle. So $(ii)$
can read off from the full description of $\theta_{n}$ with $2n+2=p+q$
in \cite[Th.15, Th.18]{Paul2005howe}.

$(ii)\Rightarrow(i)$: \cite[Th.15, Th.18]{Paul2005howe} explicitly
gives $\theta_{n}(\pi)\neq0$ for $n=\frac{p+q}{2}-1$ when $(ii)$
holds.

$(i)$ and $(ii)\Rightarrow(iii)$: By \cite[Cor.37]{Paul2005howe},
$(i)$ implies that for any $\sigma\in\mathcal{A}(\pi)$, there is
$\delta\in\mathcal{A}(\pi)\cap\mathcal{D}(\pi)$ with the same highest
weight as that of $\sigma$. However, under $(ii)$ the choice for
signs of lowest $K$-types of $\pi$ with the given highest weight
is unique (see Proposition \ref{prop:LKT-sgn}), so $\delta=\sigma$.\end{proof}
\begin{thm}
[Explicit induction principle on $(p,q)$] \label{thm:EIP_pq} Let
$\pi=\pi_{\zeta}(\lambda_{d},\xi,\Psi,\mu,\nu,\varepsilon,\kappa)$
$\in\mathcal{R}(O(p,q))$ with $p+q$ even. Suppose that $n(\pi)\leqslant\frac{p+q}{2}-1$
(or equivalently, $\zeta=\xi=1$ and either $\lambda_{d}$ contains
a zero entry or some $(\varepsilon_{i},\kappa_{i})=(1,0)$). Let integers
$n\geqslant n(\pi)$ and $k\geqslant1$ satisfy $n\notin[\frac{p+q}{2},\frac{p+q}{2}+k-1]$.
 (If $n(\pi)\leqslant n\leqslant\frac{p+q}{2}-1$, we may take any
$k\geqslant1$.) Then 
\begin{align*}
\theta_{p+k,q+k}(\theta_{n}(\pi)) & =\pi_{1}(\lambda_{d},1,\Psi,\mu,\nu,(\varepsilon|(1,1,\dots,1)),\\
 & \qquad(\kappa|(n-\frac{p+q}{2},n-1-\frac{p+q}{2},\dots,n-k+1-\frac{p+q}{2})))
\end{align*}
with a possible modification: if the resulting parameters contain
some entries $\kappa_{i}=\pm\kappa_{j}$ with $\varepsilon_{i}\neq\varepsilon_{j}$,
delete $\varepsilon_{i}$, $\varepsilon_{j}$, $\kappa_{i}$, $\kappa_{j}$
from $(\varepsilon,\kappa)$, and add entries $(0,2\kappa_{i})$ into
$(\mu,\nu)$.\end{thm}
\begin{proof}
Since $n(\pi)\leqslant\frac{p+q}{2}-1$, the parameters of $\pi$
satisfy $(ii)$ of Lemma \ref{lem:cond_pq} and $\mathcal{A}(\pi)\subseteq\mathcal{D}(\pi)$.
Take any $\sigma\in\mathcal{A}(\pi)$, we get $\sigma_{1,1}\in\mathcal{D}(\theta_{p+1,q+1}(\theta_{n}(\pi)))$.
 By Proposition \ref{prop:cond2}, $\sigma_{1,1}$ is also a lowest
$K$-type of $\pi_{1,1}$. By Proposition \ref{prop:LKT+11}, $\theta_{p+1,q+1}(\theta_{n}(\pi))=\pi_{1}(\lambda_{d},1,\Psi,\mu,\nu,(\varepsilon|(1)),(\kappa|(n-\frac{p+q}{2})))$
with a possible modification of parameters. Repeat this process for
$k$ times.
\end{proof}

\section{Reducing cases when $p+q=4$}

The following parts of this paper aim to explicitly calculate the
theta lifting
\[
\theta_{n}:\mathcal{R}(O(p,q))\to\mathcal{R}(Sp(2n,\mathbb{R}))\cup\{0\},
\]
for all $n\geqslant1$ when $p+q=4$, in terms of Langlands parameters.
This section reduces the cases that need consideration.

\subsection{Reducing $(p,q)$}

As $O(p,q)=O(q,p)$, there is a bijection $\varphi:\mathcal{R}(O(p,q))\stackrel{\simeq}{\longrightarrow}\mathcal{R}(O(q,p))$.
Under our parametrization $\varphi$ is indeed the following operation
on $(\lambda_{d},\Psi)$ (preserving other parameters):
\begin{itemize}
\item interchange the two parts of $\lambda_{d}=\{*;*\}$,
\item interchange the roles of $\{e_{i}\}$ and $\{f_{j}\}$ for $\Psi$.\end{itemize}
\begin{lem}
[{\cite[Lemma 20]{Paul2005howe}}] \label{lem:exchange-pq} Let
$n$, $p$, and $q$ be nonnegative integers with $p+q$ even. Let
$\pi'\in\mathcal{R}(Sp(2n,\mathbb{R}))$ with contragredient $\pi'^{*}$.
If $\theta_{p,q}(\pi')\neq0$, then
\[
\varphi(\theta_{p,q}(\pi'))=\theta_{q,p}(\pi'^{*}).
\]

\end{lem}
To transfer between the Langlands parameters of $\pi'$ and $\pi'^{*}$,
we only need to:
\begin{itemize}
\item replace $\lambda_{d}=(c_{1},c_{2},\dots,c_{v})$ by $(-c_{v},-c_{v-1},\dots,-c_{1})$,
\item replace $e_{i}$ by $-e_{v+1-i}$ for $\Psi$.
\end{itemize}
Therefore, to calculate $\theta_{n}$ for $O(p,q)$ when $p+q=4$,
it suffices to calculate when $(p,q)=(4,0),(3,1),(2,2)$.

\subsection{First occurrence}

For $\pi\in\mathcal{R}(O(p,q))$ with $p+q$ even, its \emph{first
occurrence index}  is
\[
n(\pi)=\min\{k\geqslant0\mid\theta_{k}(\pi)\neq0\}.
\]
Taking the oscillator representation of $\widetilde{Sp}(0,\mathbb{R})\cong\mu_{2}=\{\pm1\}$
to be the nontrivial character, we get the local theta correspondence
for $(O(p,q),Sp(0,\mathbb{R}))$ as $\tri\leftrightarrow\tri$, where
$\tri$ denotes the trivial representation. In this sense, it is assumed
that $n(\tri)=0$, and $n(\pi)\geqslant1$ for any nontrivial $\pi$.

We say that a reductive dual pair $(G,G')$ of type I is in the \emph{stable
range} with $G$ the smaller member if the defining module of $G'$
has an isotropic subspace of the same dimension as that of the defining
module of $G$. For $(G,G')=(O(p,q),Sp(2n,\mathbb{R}))$, it is in
the stable range with $G$ the smaller member if $n\geqslant p+q$.
The nonvanishing of theta liftings in the stable range (cf. \cite{Li1989singular,ProtsakPrzebinda2008occurrence})
states that:
\begin{prop}
$n(\pi)\leqslant p+q$ for all $\pi\in\mathcal{R}(O(p,q))$ with $p+q$
even.
\end{prop}
Let $\mathrm{det}_{p,q}$  be the determinant character of $O(p,q)$.
Its only $O(p)\times O(q)$-type is $(0,\dots,0;-1)\otimes(0,\dots,0;-1)$,
which does not occur in the space of joint harmonics when $n<p+q$
by Proposition \ref{prop:CorKT}. Hence by Lemma \ref{lem:min-deg},
$n(\mathrm{det}_{p,q})\geqslant p+q$, and thus $n(\mathrm{det}_{p,q})=p+q$.

Recently B. Sun and C.-B. Zhu \cite{SunZhu2012conservation} proved
some ``conservation relations'' conjectured by Kudla and Rallis
(\cite{KudlaRallis2005first}) about the first occurrence for local
theta correspondence. Especially, the following relation for $(O(p,q),Sp(2n,\mathbb{R}))$
holds.
\begin{lem}
[\cite{SunZhu2012conservation}] \label{thm:Conser} For $\pi\in\mathcal{R}(O(p,q))$
with $p+q$ even, 
\[
n(\pi)+n(\pi\otimes\mathrm{det}_{p,q})=p+q.
\]
\end{lem}
\begin{prop}
If $\mathrm{det}_{p,q}\neq\pi\in\mathcal{R}(O(p,q))$ with $p+q$
even, then $n(\pi)\leqslant p+q-1$. Therefore, $\mathrm{det}_{p,q}$
is the only element in $\mathcal{R}(O(p,q))$ with first occurrence
index $p+q$.\end{prop}
\begin{proof}
$\pi\otimes\mathrm{det}_{p,q}\neq\tri$ $\Rightarrow n(\pi\otimes\mathrm{det}_{p,q})\geqslant1$
$\Rightarrow n(\pi)=p+q-n(\pi\otimes\mathrm{det}_{p,q})\leqslant p+q-1$.
\end{proof}
Fix $(p,q)$ with $p+q$ even and $\pi\in\mathcal{R}(O(p,q))$. To
explicitly calculate $\theta_{n}(\pi)$ for all $n$, it suffices
to calculate $\theta_{n}(\pi)$ for $n(\pi)\leqslant n\leqslant\max\{n(\pi),\frac{p+q}{2}\}$
by the explicit induction principle Theorem \ref{thm:EIP_n}.

When $p+q=4$, $\theta_{1}$ and $\theta_{2}$ can be explicitly read
off from \cite{Paul2005howe} as (almost) equal rank cases, and will
be written down in Appendix \ref{App:Theta12}. So we only need to
calculate $\theta_{4}(\mathrm{det}_{p,q})$ and $\theta_{3}(\pi)$
for all $\pi\in\mathcal{R}(O(p,q))$ with $n(\pi)=3$.

\section{\label{sec:3-lifts}Theta $3$-lifts when $p+q=4$ and \texorpdfstring{
$n(\pi)=3$}{\textup{$n($π}$)=3$}}

In this section, we will calculate $\theta_{3}(\pi)$ explicitly for
any $\pi\in\mathcal{R}(O(p,q))$ with $n(\pi)=3$ when $(p,q)=(4,0)$,
$(3,1)$, $(2,2)$.
\begin{lem}
[{\cite[Cor.24]{Paul2005howe}}] Let $\pi=\pi_{\zeta}(\lambda_{d},\xi,\Psi,\mu,\nu,\varepsilon,\kappa)\in\mathcal{R}(O(p,q))$
with $p+q$ even. Then $n(\pi)\leqslant\frac{p+q}{2}$ if and only
if either $\xi=\zeta=1$, or $\lambda_{d}$ contains no zero entry
and some $(\varepsilon_{i},\kappa_{i})=(-1,0)$.\end{lem}
\begin{prop}
\label{prop:n=00003D3} Let $\pi=\pi_{\zeta}(\lambda_{d},\xi,\Psi,\mu,\nu,\varepsilon,\kappa)\in\mathcal{R}(O(p,q))$
with $p+q=4$. Then $n(\pi)=3$ if and only if $\pi\neq\mathrm{det}_{p,q}$
and either $\xi=-1$, or $\zeta=-1$ and some $(\varepsilon_{i},\kappa_{i})=(1,0)$.\end{prop}
\begin{proof}
The previous lemma asserts that $n(\pi)\geqslant3\Leftrightarrow(\xi,\zeta)\neq(1,1)$,
and either $\lambda_{d}$ contains a zero entry or all $(\varepsilon_{i},\kappa_{i})\neq(-1,0)$.
Note that $\xi=-1$ only if $\lambda_{d}$ contains a zero entry.
Also note that $\zeta=-1$ only if $\lambda_{d}$ contains no zero
entry and $\kappa$ contains a zero entry.
\end{proof}
When $p+q=4$, this proposition let us list explicitly all $\pi\in\mathcal{R}(O(p,q))$
with $n(\pi)=3$.

\subsection{For $O(4,0)$}

Each $\pi\in\mathcal{R}(O(4,0))$ is parametrized as $\pi=\pi_{1}((m,l;),\xi,\Psi,0,0,$
$0,0)$ with integers $m>l\geqslant0$ and $\xi\in\{\pm1\}$. Now
$\Delta=\{\pm e_{1}\pm e_{2}\}$, and $\Delta_{c}^{+}=\{e_{1}\pm e_{2}\}=\Psi$.

As $O(4,0)=O(4)\times O(0)$ is compact, $\pi$ itself is a $K$-type.
By Proposition \ref{prop:LKT-O} and \ref{prop:LKT-sgn} to compute
lowest $K$-types, this $K$-type is $(m-1,l;\xi)\otimes(;)$. So
\[
\mathrm{det}_{4,0}=\pi_{1}((1,0;),-1,\{e_{1}\pm e_{2}\},0,0,0,0).
\]
By Proposition \ref{prop:n=00003D3}, all $\pi\in\mathcal{R}(O(4,0))$
with $n(\pi)=3$ are 
\[
\pi=\pi_{1}((m,0;),-1,\{e_{1}\pm e_{2}\},0,0,0,0)\quad\text{with }2\leqslant m\in\mathbb{Z}.
\]
For such $\pi$, by Lemma \ref{lem:cond_n}, 
\begin{align*}
\mathcal{A}(\theta_{3}(\pi)) & =\mathcal{D}(\theta_{3}(\pi))=\{\phi_{3}((m-1,0;-1)\otimes(;))\}=\{(m+1,3,3)\}.
\end{align*}
 By Proposition \ref{prop:inf.char} and \ref{prop:CrspIF}, the infinitesimal
character of $\theta_{3}(\pi)$ is $(m,0,1)$. 
\begin{prop}
For $2\leqslant m\in\mathbb{Z}$, there is a unique $\pi'\in\mathcal{R}(Sp(6,\mathbb{R}))$
with the infinitesimal character $(m,0,1)$ and $\mathcal{A}(\pi')=\{(m+1,3,3)\}$.
It is 
\[
\theta_{3}(\pi_{1}((m,0;),-1,\{e_{1}\pm e_{2}\},0,0,0,0))=\pi((m,1,0),\Psi,0,0,0,0),
\]
with $\Psi=\{e_{1}\pm e_{2},e_{2}\pm e_{3},e_{1}\pm e_{3},2e_{1},2e_{2},2e_{3}\}$.\end{prop}
\begin{proof}
Go through Appendix \ref{App:list} which lists $\mathcal{A}(\pi')$
for all $\pi'\in\mathcal{R}(Sp(6,\mathbb{R}))$ with the infinitesimal
character $(\beta,0,1)$, where $\beta\in\mathbb{C}$.
\end{proof}

\subsection{For $O(2,2)$}

Let $\pi=\pi_{\zeta}(\lambda_{d},\xi,\Psi,\mu,\nu,\varepsilon,\kappa)\in\mathcal{R}(O(2,2))$
with $n(\pi)=3$. By Proposition \ref{prop:n=00003D3}, all these
$\pi$ are listed in the following table.

\smallskip{}

\begin{center}
\begin{tabular}{c|c|c|c|c|c|c|c|c}
\hline 
$\zeta$ & $\xi$ & $\lambda_{d}$ & $\Psi$ & $\mu$ & $\nu$ & $\varepsilon$ & $\kappa$ & with\tabularnewline
\hline 
\hline 
\multirow{2}{*}{$1$} & \multirow{2}{*}{$-1$} & $(n;0)$ & $\{e_{1}\pm f_{1}\}$ & \multirow{4}{*}{$0$} & \multirow{4}{*}{$0$} & \multirow{2}{*}{$0$} & \multirow{2}{*}{$0$} & \multirow{2}{*}{$0\leqslant n\in\mathbb{Z}$}\tabularnewline
\cline{3-4} 
 &  & $(0;n)$ & $\{\pm e_{1}+f_{1}\}$ &  &  &  &  & \tabularnewline
\cline{1-4} \cline{7-9} 
\multirow{2}{*}{$-1$} & \multirow{2}{*}{$1$} & \multirow{2}{*}{$0$} & \multirow{2}{*}{$\O$} &  &  & $(1,-1)$ & $(0,\beta)$ & $\beta\in\mathbb{C}\backslash\{0\}$ by (F-2)\tabularnewline
\cline{7-9} 
 &  &  &  &  &  & $(1,1)$ & $(0,\beta)$ & $\beta\in\mathbb{C}\backslash\{\pm1\}$\tabularnewline
\hline 
\end{tabular}
\par\end{center}

\smallskip{}
\begin{itemize}
\item $\det_{2,2}=\pi_{-1}(0,1,\O,0,0,(1,1),(0,1))$. Indeed, it is of the
form as in this table but with the infinitesimal character $(0,1)$
and $\mathcal{A}(\det_{2,2})=\{(0;-1)\otimes(0;-1)\}$.\end{itemize}
\begin{prop}
If $\pi=\pi_{1}((n;0),-1,\{e_{1}\pm f_{1}\},0,0,0,0)\in\mathcal{R}(O(2,2))$
with $0\leqslant n\in\mathbb{Z}$, then
\[
\mathcal{A}(\theta_{3}(\pi))=\{(n+1,-1,-1)\}.
\]
\end{prop}
\begin{proof}
By Proposition \ref{prop:LKT-O} and \ref{prop:LKT-sgn}, $\mathcal{A}(\pi)=\{(n+1;1)\otimes(0;-1)\}$.
The lowest $K$-type of $\pi$ has degree $n+3$, so $\mathrm{deg}(\pi)\leqslant n+3$.
We claim that $\mathrm{deg}(\pi)=n+3$. Otherwise, by Lemma \ref{lem:parity}
on the parity of degrees, there is a $K$-type $(m;\epsilon)\otimes(l;\eta)$
of $\pi$ with degree $\leqslant n+1$. So $m+l\leqslant n+1$. But
$(m;\epsilon)\otimes(l;\eta)$ is not a lowest $K$-type, so $m^{2}+l^{2}>(n+1)^{2}\geqslant(m+l)^{2}$,
which makes a contradiction.

All $K$-types for $O(2,2)$ occurring in the space of joint harmonics
with degree $n+3$ are: $(n+1;1)\otimes(0;-1)$, $(0;-1)\otimes(n+1;1)$,
and $(m;1)\otimes(l;1)$ with $m+l=n+3$. As $(0;-1)\otimes(n+1;1)$
has the same norm as that of the lowest $K$-type, it cannot occur
in $\pi$. Note that $(m;1)\otimes(l;1)$ occurs in $\pi$ only if
$m^{2}+l^{2}>(n+1)^{2}$. Therefore, all $U(3)$-types in $\mathcal{D}(\theta_{3}(\pi))$
other than $(n+1,-1,-1)$ are of the form $(m,0,-l)$ with $m+l=n+3$
and $m^{2}+l^{2}>(n+1)^{2}$. Then
\begin{align*}
||(m,0,-l)|| & =m^{2}+l^{2}+4(m+l)+8\\
 & >(n+1)^{2}+4(n+3)+8\\
 & >||(n+1,-1,-1)||.
\end{align*}
By Lemma \ref{lem:cond_n}, the set $\mathcal{A}(\theta_{3}(\pi))$
consists of $U(3)$-types with minimal norm in $\mathcal{D}(\theta_{3}(\pi))$,
so $\mathcal{A}(\theta_{3}(\pi))=\{(n+1,-1,-1)\}$.\end{proof}
\begin{rem*}
Similarly, if $\pi=\pi_{1}((0;n),-1,\{\pm e_{1}+f_{1}\},0,0,0,0)$
with $0\leqslant n\in\mathbb{Z}$, then $\mathcal{A}(\theta_{3}(\pi))=\{(1,1,-n-1)\}$. \end{rem*}
\begin{prop}
If $\pi=\pi_{-1}(0,1,\O,0,0,(1,-1),(0,\beta))$ with $\beta\neq0$,
then
\[
\mathcal{A}(\theta_{3}(\pi))=\{(1,-1,-1),(1,1,-1)\}.
\]
\end{prop}
\begin{proof}
By Proposition \ref{prop:LKT-O} and \ref{prop:LKT-sgn}, $\mathcal{A}(\pi)=\{(1;1)\otimes(0;-1),(0;-1)\otimes(1;1)\}$.
The lowest $K$-types of $\pi$ have degree $3$, so $\deg(\pi)=1$
or $3$ by Lemma \ref{lem:parity} on the parity of degrees. All $K$-types
for $O(2,2)$ with degree $1$ are $(1;1)\otimes(0;1)$ and $(0;1)\otimes(1;1)$.
They have the same norm as that of the lowest $K$-types of $\pi$,
thus cannot occur in $\pi$. So $\deg(\pi)=3$ and $\mathcal{A}(\pi)\subset\mathcal{D}(\pi)$.
Then
\begin{align*}
\phi_{3}(\mathcal{A}(\pi)) & =\{(1,-1,-1),(1,1,-1)\}\\
 & \subset\phi_{3}(\mathcal{D}(\pi))=\mathcal{D}(\theta_{3}(\pi)).
\end{align*}
All $K$-types for $Sp(6,\mathbb{R})$ occurring in the space of joint
harmonics with degree $3$ are: $(1,-1,-1)$, $(1,1,-1)$, $(2,0,-1)$,
$(1,0,-2)$, $(3,0,0)$, and $(0,0,-3)$, among which the first two
take the minimal norm and belong to $\mathcal{D}(\theta_{3}(\pi))$,
so $\mathcal{A}(\theta_{3}(\pi_{c}))=\{(1,-1,-1),(1,1,-1)\}$ by Lemma
\ref{lem:cond_n}.\end{proof}
\begin{prop}
If $\pi=\pi_{-1}(0,1,\O,0,0,(1,1),(0,\beta))$ with $\beta\neq\pm1$,
then 
\[
\mathcal{A}(\theta_{3}(\pi))=\{(1,0,-1)\}.
\]
\end{prop}
\begin{proof}
By Proposition \ref{prop:LKT-O} and \ref{prop:LKT-sgn}, $\mathcal{A}(\pi)=\{(0;-1)\otimes(0;-1)\}$.
The lowest $K$-type of $\pi$ has degree $4$, but does not occur
in the space of joint harmonics for $(O(2,2),Sp(6,\mathbb{R}))$.
Thus $\deg(\pi)<4$. So $\deg(\pi)=0$ or $2$ by Lemma \ref{lem:parity}
on the parity of degrees. The only $K$-type for $O(2,2)$ with degree
$0$ is $(0;1)\otimes(0;1)$. It has the same norm as that of the
lowest $K$-type of $\pi$, thus cannot occur in $\pi$. So $\deg(\pi)=2$.

All $K$-types for $O(2,2)$ with degree $2$ are: $(0;-1)\otimes(0;1)$,
$(0;1)\otimes(0;-1)$, $(1;1)\otimes(1;1)$, $(2;1)\otimes(0;1)$
and $(0;1)\otimes(2;1)$. The first two have the same norm as that
of the lowest $K$-type of $\pi$, and thus cannot occur in $\pi$.
Therefore,
\begin{gather*}
\mathcal{D}(\pi)\subset\{(1;1)\otimes(1;1),(2;1)\otimes(0;1),(0;1)\otimes(2;1)\},\\
\mathcal{D}(\theta_{3}(\pi))=\phi_{3}(\mathcal{D}(\pi))\subset\{(1,0,-1),(2,0,0),(0,0,-2)\}.
\end{gather*}
And the norms $||(1,0,-1)||<||(2,0,0)||=||(0,0,-2)||$.

By Lemma \ref{lem:exchange-pq}, $\theta_{3}(\pi)^{*}=\theta_{3}(\varphi(\pi))=\theta_{3}(\pi)$,
which asserts that $\theta_{3}(\pi)$ has parameter $\lambda_{d}$
of the form $(c,0,-c)$ and $\Psi$ preserved after replacing $e_{i}$
by $-e_{v+1-i}$. By the algorithm in Proposition \ref{prop:LKT-Sp},
the set $\mathcal{A}(\theta_{3}(\pi))$ remains the same after replacing
each $U(3)$-type $(m,n,l)$ by $(-l,-n,-m)$.

By Lemma \ref{lem:cond_n}, the set $\mathcal{A}(\theta_{3}(\pi))$
consists of $K$-types in $\mathcal{D}(\theta_{3}(\pi))$ with minimal
norm, so $\mathcal{A}(\theta_{3}(\pi))=\{(1,0,-1)\}$ or $\{(2,0,0),(0,0,-2)\}$.
The infinitesimal character of $\pi$ is $(\beta,0)$, thus that of
$\theta_{3}(\pi)$ is $(\beta,0,1)$. According to the Appendix \ref{App:list}
which lists $\mathcal{A}(\pi')$ for all $\pi'\in\mathcal{R}(Sp(6,\mathbb{R}))$
with the infinitesimal character $(\beta,0,1)$, we see $\mathcal{A}(\pi')\neq\{(2,0,0),$
$(0,0,-2)\}$.\end{proof}
\begin{thm}
Let $\pi\in\mathcal{R}(O(2,2))$ with $n(\pi)=3$. 

(1) If $\pi\neq\pi_{-1}(0,1,\O,0,0,(1,1),(0,2))$, then the infinitesimal
character of $\theta_{3}(\pi)$ and the set $\mathcal{A}(\theta_{3}(\pi))$
determine a unique element in $\mathcal{R}(Sp(6,\mathbb{R}))$, which
is $\theta_{3}(\pi)$ listed in the following table. 

(2) If $\pi=\pi_{-1}(0,1,\O,0,0,(1,1),(0,2))$, then $\theta_{3}(\pi)=\pi(0,\O,(1),(1),(1),(2))$.
\end{thm}
\smallskip{}\noindent\resizebox{\textwidth}{!}{%

\begin{tabular}{c|c|c|c|c}
\hline 
\multirow{2}{*}{$\pi$} & inf. char. & \multirow{2}{*}{$\mathcal{A}(\theta_{3}(\pi))$} & \multirow{2}{*}{$\theta_{3}(\pi)$} & \multirow{2}{*}{with}\tabularnewline
 & of $\theta_{3}(\pi)$ &  &  & \tabularnewline
\hline 
\hline 
$\pi_{1}((0;0),-1,$ & \multirow{4}{*}{$(0,0,1)$} & \multirow{2}{*}{$\{(1,-1,-1)\}$} & \multirow{2}{*}{$\pi((0),\{-2e_{1}\},(1),(1),0,0)$} & \multirow{4}{*}{}\tabularnewline
$\{e_{1}\pm f_{1}\},0,0,0,0)$ &  &  &  & \tabularnewline
\cline{1-1} \cline{3-4} 
$\pi_{1}((0;0),-1,$ &  & \multirow{2}{*}{$\{(1,1,-1)\}$} & \multirow{2}{*}{$\pi((0),\{2e_{1}\},(1),(1),0,0)$} & \tabularnewline
$\{\pm e_{1}+f_{1}\},0,0,0,0)$ &  &  &  & \tabularnewline
\hline 
$\pi_{1}((n;0),-1,$ & \multirow{4}{*}{$(n,0,1)$} & \multirow{2}{*}{$\{(n+1,-1,-1)\}$} & $\pi((n,0,-1),\{e_{1}\pm e_{2},\pm e_{2}-e_{3},$ & \multirow{4}{*}{$1\leqslant n\in\mathbb{Z}$}\tabularnewline
$\{e_{1}\pm f_{1}\},0,0,0,0)$ &  &  & $e_{1}\pm e_{3},2e_{1},-2e_{2},-2e_{3}\},0,0,0,0)$ & \tabularnewline
\cline{1-1} \cline{3-4} 
$\pi_{1}((0;n),-1,$ &  & \multirow{2}{*}{$\{(1,1,-n-1)\}$} & $\pi((1,0,-n),\Psi,\{e_{1}\pm e_{2},\pm e_{2}-e_{3},$ & \tabularnewline
$\{\pm e_{1}+f_{1}\},0,0,0,0)$ &  &  & $\pm e_{1}-e_{3},2e_{1},2e_{2},-2e_{3}\},0,0,0,0)$ & \tabularnewline
\hline 
$\pi_{-1}(0,1,\O,0,0,$ & \multirow{4}{*}{$(\beta,0,1)$} & $\{(1,-1,-1),$ & \multirow{2}{*}{$\pi(0,\O,(1),(1),(-1),(\beta))$} & \multirow{2}{*}{$\beta\in\mathbb{C}\backslash\{0\}$}\tabularnewline
$(1,-1),(0,\beta))$ &  & $(1,1,-1)\}$ &  & \tabularnewline
\cline{1-1} \cline{3-5} 
$\pi_{-1}(0,1,\O,0,0,$ &  & \multirow{2}{*}{$\{(1,0,-1)\}$} & \multirow{2}{*}{$\pi(0,\O,(1),(1),(1),(\beta))$} & \multirow{2}{*}{$\beta\in\mathbb{C}\backslash\{\pm1\}$}\tabularnewline
$(1,1),(0,\beta))$ &  &  &  & \tabularnewline
\hline 
\end{tabular}

}\smallskip{}
\begin{proof}
(1) can be checked case-by-case according to Appendix \ref{App:list}
which lists $\mathcal{A}(\pi')$ for all $\pi'\in\mathcal{R}(Sp(6,\mathbb{R}))$
with the infinitesimal character $(\beta,0,1)$.

For (2), now $\theta_{3}(\pi)$ has the infinitesimal character $(2,0,1)$
and $\mathcal{A}(\theta_{3}(\pi))=\{(1,0,-1)\}$. According to Appendix
\ref{App:list}, there are exactly two elements in $\mathcal{R}(Sp(6,\mathbb{R}))$
with such infinitesimal character and set of lowest $K$-types: $\pi_{1}'=\pi(0,\O,(1),(1),(1),(2))$
and $\pi_{2}'=\pi(0,\O,(1),(3),(1),(0))$. Note that $\theta_{2}(\det_{1,1})=\pi(0,\O,(1),(1),0,0)$
 by Proposition \ref{prop:det_pp}, and $\theta_{3}(\det_{1,1})=\pi_{1}'$
by the explicit induction principle on $n$. Then $\theta_{2,2}(\pi_{1}')\neq0$
by Kudla's persistence principle.

We claim that $\theta_{2}(\theta_{2,2}(\pi_{1}'))=0$. Otherwise,
$\pi_{1}'=\theta_{3}(\theta_{2,2}(\pi_{1}'))$ is got from $\theta_{2}(\theta_{2,2}(\pi_{1}'))$
by explicit induction principle on $n$ (Theorem \ref{thm:EIP_n}),
which must contain an entry $1$ in $\kappa$ or $2$ in $\nu$. But
the Langlands parameters of $\pi_{1}'$ are not of this form. Therefore
$n(\theta_{2,2}(\pi_{1}'))=3$. Yet $\pi_{1}'$ is not in the list
of theta $3$-lifts in (1), so $\theta_{2,2}(\pi_{1}')=\pi_{-1}(0,1,\O,0,0,(1,1),(0,2))$.
\end{proof}

\subsection{For $O(3,1)$}

Let $\pi=\pi_{\zeta}(\lambda_{d},\xi,\Psi,\mu,\nu,\varepsilon,\kappa)\in\mathcal{R}(O(3,1))$
with $n(\pi)=3$. By Proposition \ref{prop:n=00003D3}, all these
$\pi$ are listed in the following table.

\smallskip{}

\begin{center}
\begin{tabular}{c|c|c|c|c|c|c|c|c}
\hline 
$\zeta$ & $\xi$ & $\lambda_{d}$ & $\Psi$ & $\mu$ & $\nu$ & $\varepsilon$ & $\kappa$ & with\tabularnewline
\hline 
\hline 
$-1$ & $1$ & $(m;)$ & \multirow{3}{*}{$\O$} & \multirow{3}{*}{$0$} & \multirow{3}{*}{$0$} & $(1)$ & $(0)$ & $1\leqslant m\in\mathbb{Z}$\tabularnewline
\cline{1-3} \cline{7-9} 
\multirow{2}{*}{$1$} & \multirow{2}{*}{$-1$} & \multirow{2}{*}{$(0;)$} &  &  &  & $(1)$ & $(\beta)$ & $\beta\in\mathbb{C}\backslash\{\pm1\}$\tabularnewline
\cline{7-9} 
 &  &  &  &  &  & $(-1)$ & $(\beta)$ & $\beta\in\mathbb{C}$\tabularnewline
\hline 
\end{tabular}
\par\end{center}

\smallskip{}
\begin{itemize}
\item $\det_{3,1}=\pi_{1}((0;),-1,\O,0,0,(1),(1))$.  Indeed, it is of
the form as in this table but with the infinitesimal character $(0,1)$
and $\mathcal{A}(\det_{3,1})=\{(0;-1)\otimes(;-1)\}$.\end{itemize}
\begin{prop}
If $\pi=\pi_{-1}((m;),1,\O,0,0,(1),(0))$ with $1\leqslant m\in\mathbb{Z}$,
then 
\[
\mathcal{A}(\theta_{3}(\pi))=\{(m+1,2,0)\}.
\]
\end{prop}
\begin{proof}
By Proposition \ref{prop:LKT-O} and \ref{prop:LKT-sgn}, $\mathcal{A}(\pi)=\{(m;-1)\otimes(;-1)\}$.
So $\deg(\pi)\leqslant\deg((m;-1)\otimes(;-1))=m+2$. We claim that
$\deg(\pi)=m+2$. Otherwise, $\deg(\pi)\leqslant m$ by Lemma \ref{lem:parity}
on the parity of degrees. Let $\delta=(l;\epsilon)\otimes(;\eta)\in\mathcal{D}(\pi)$.
Then $m\geqslant\deg(\delta)\geqslant l$. But $||\delta||=|l+1|^{2}>|m+1|^{2}$,
so $l>m$, which makes a contradiction.

Hence $(m;-1)\otimes(;-1)\in\mathcal{D}(\pi)$, and $(m+1,2,0)\in\mathcal{D}(\theta_{3}(\pi))$.
Other $K$-types for $O(3,1)$ of degree $m+2$ are $(m+1;+1)\otimes(;-1)$,
$(m+1;-1)\otimes(;+1)$, and $(m+2;+1)\otimes(;+1)$. So
\begin{align*}
(m+1,2,0) & \in\mathcal{D}(\theta_{3}(\pi))\\
 & \subseteq\{(m+1,2,0),\ (m+2,1,0),\ (m+2,2,1),\ (m+3,1,1)\}.
\end{align*}
As $||(m+1,2,0)||<||(m+2,1,0)||=||(m+2,2,1)||<||(m+3,1,1)||$, by
Lemma \ref{lem:cond_n}, we have $\mathcal{A}(\theta_{3}(\pi))=\{(m+1,2,0)\}$.\end{proof}
\begin{prop}
Let $\pi=\pi_{1}((0),-1,\O,0,0,(1),(\beta))$ with $\beta\in\mathbb{C}\backslash\{\pm1\}$.
Let $I$ be the standard module of parabolic induction to get $\pi$
as in Subsection \ref{subsec:Par-O}.

(1) For an $O(3)\times O(1)$-type $\sigma=(l;\epsilon)\otimes(;\eta)$,
the multiplicity of $\sigma$ in $I$ is\textup{
\[
m(\sigma,I)=\begin{cases}
1, & \text{if }\epsilon=-1\text{ and }\eta=(-1)^{l+1},\\
0, & \text{otherwise}.
\end{cases}
\]
}

(2) The set $\mathcal{D}(\pi)=\{(1;-1)\otimes(;+1)\}$, and $\mathcal{A}(\theta_{3}(\pi))=\mathcal{D}(\theta_{3}(\pi))=\{(2,2,1)\}$.\end{prop}
\begin{proof}
(1) By definition $I=\mathrm{Ind}_{MAN}^{O(3,1)}(\det\otimes\chi_{1,\beta}\otimes\tri_{N})$,
where $MA\cong O(2)\times GL(1,\mathbb{R})$, $\det$ is the determinant
character of $O(2)$, $\chi_{1,\beta}$ is the character of $GL(1,\mathbb{R})$
that maps $x$ to $|x|^{\beta}$, and $\tri_{N}$ is the trivial representation
of $N$. For $K=O(3)\times O(1)$, notice the $K$-intertwining isomorphism
$I|_{K}\cong\mathrm{Ind}_{K\cap MA}^{K}(\det\otimes\chi_{1,\beta})$
(cf. \cite[Prop.4.1.12]{Vogan1981representations}). By Frobenius
reciprocity, the multiplicity
\begin{align*}
m(\sigma,I)=\dim\mathrm{Hom}_{K}(\sigma,I|_{K}) & =\dim\mathrm{Hom}_{K}(\sigma,\mathrm{Ind}_{K\cap MA}^{K}(\det\otimes\chi_{1,\beta}))\\
 & =\dim\mathrm{Hom}_{K\cap MA}(\sigma,\det\otimes\chi_{1,\beta}),
\end{align*}
where $K\cap MA=\{\mathrm{diag}(X,t,t)\mid X\in O(2),\ t=\pm1\}$.
Let $K_{0}=K\cap MA$ and $\lambda=(\det\otimes\chi_{1,\beta})|_{K_{0}}$.
Then $\lambda$ is the character of $K_{0}$ defined by $\lambda(\mathrm{diag}(X,t,t))=\det(X)$.
Write $V$ for the underlying space of $\sigma$. Then 
\[
m(\sigma,I)=\dim\mathrm{Hom}_{K_{0}}(V|_{K_{0}},\lambda)=\dim V(\lambda),
\]
where $V(\lambda)$ is the $\lambda$-isotypic invariant subspace
of $V|_{K_{0}}.$

Let us realize $\sigma$ in a concrete underlying space. Let $L=\{\mathrm{diag}(Y,1)\mid Y\in SO(3)\}$.
Then $L$ is a subgroup of $O(3)\times O(1)$ isomorphic to $SO(3)$,
and $\sigma|_{L}$ is irreducible and of dimension $2l+1$. Recall
that there is a unique $(2l+1)$-dimensional irreducible representation
$(\sigma_{l},V_{l})$ of $SO(3)$, which can be realized via the double
cover 
\[
\mathbf{P}:SU(2)\twoheadrightarrow SO(3),
\]
on the vector space $V_{l}$ of polynomials in $Z_{1}$, $Z_{2}$
that are homogeneous of degree $2l$, with actions defined by 
\[
\left(\sigma_{l}\left(\mathbf{P}\left[\begin{array}{cc}
\alpha & \beta\\
-\bar{\beta} & \bar{\alpha}
\end{array}\right]\right)f\right)\left[\begin{array}{c}
Z_{1}\\
Z_{2}
\end{array}\right]=f\left(\left[\begin{array}{cc}
\alpha & \beta\\
-\bar{\beta} & \bar{\alpha}
\end{array}\right]^{-1}\left[\begin{array}{c}
Z_{1}\\
Z_{2}
\end{array}\right]\right)\quad\text{ for }f\in V_{l}.
\]
Here $f$ is of even degree, so the action is trivial on $\{\pm Id\}=\ker(\mathbf{P})$,
and $\sigma_{l}$ is well defined. Now $\sigma$ can be realized on
$V_{l}$, with actions
\begin{align*}
\sigma(\mathrm{diag}(Y,1)) & =\sigma_{l}(Y),\\
\sigma(\mathrm{diag}(-1,-1,-1,1)) & =(-1)^{l}\epsilon,\\
\sigma(\mathrm{diag}(1,1,1,-1)) & =\eta.
\end{align*}
The actions of the last two matrices are read off from the signs of
$\sigma$.

We can find the $\lambda$-isotypic invariant subspace $V(\lambda)$
in $V=V_{l}$. Let 
\[
r(\theta)=\left[\begin{array}{cc}
\cos\theta & \sin\theta\\
-\sin\theta & \cos\theta
\end{array}\right]\in SO(2).
\]
Let $K_{1}=\{\mathrm{diag}(r(\theta),1,1)\}\cong SO(2)$. Since $K_{1}$
is a subgroup of $K_{0}$ on which $\lambda$ acts trivially, $V(\lambda)\subseteq V^{K_{1}}$
(the subspace of $K_{1}$-fixed vectors in $V$). As 
\[
\mathbf{P}\left[\begin{array}{cc}
e^{\mathrm{i}\theta} & 0\\
0 & e^{-\mathrm{i}\theta}
\end{array}\right]=\left[\begin{array}{cc}
\cos\theta & \sin\theta\\
-\sin\theta & \cos\theta
\end{array}\right]=r(\theta),
\]
 it is easy to see that $V^{K_{1}}=\mathbb{C}Z_{1}^{l}Z_{2}^{l}$,
which is of dimension $1$. Note that $K_{0}$ is generated by $K_{1}$,
$\mathrm{diag}(1,-1,1,1)$, and $\mathrm{diag}(-1,-1,-1,-1)$. So
$V(\lambda)=\mathbb{C}Z_{1}^{l}Z_{2}^{l}$ or $0$, depending on whether
or not the following two equations hold: 
\begin{align*}
\sigma|_{\mathbb{C}Z_{1}^{l}Z_{2}^{l}}(\mathrm{diag}(1,-1,1,1)) & =\lambda(\mathrm{diag}(1,-1,1,1))=-1,\\
\sigma|_{\mathbb{C}Z_{1}^{l}Z_{2}^{l}}\sigma(\mathrm{diag}(-1,-1,-1,-1)) & =\lambda(\mathrm{diag}(-1,-1,-1,-1))=1.
\end{align*}
Note that $\mathrm{diag}(1,-1,1,1)=\mathrm{diag}(-1,-1,-1,1)\cdot\mathrm{diag}(-1,1,-1,1)$.
Since 
\[
\mathbf{P}\left[\begin{array}{cc}
0 & 1\\
-1 & 0
\end{array}\right]=\mathrm{diag}(-1,1,-1)\in SO(3),
\]
$\sigma|_{\mathbb{C}Z_{1}^{l}Z_{2}^{l}}(\mathrm{diag}(-1,1,-1,1))=(-1)^{l}$,
and thus 
\[
\sigma|_{\mathbb{C}Z_{1}^{l}Z_{2}^{l}}(\mathrm{diag}(1,-1,1,1))=(-1)^{l}\epsilon(-1)^{l}=\epsilon.
\]
Moreover,
\[
\sigma(\mathrm{diag}(-1,-1,-1,-1))=\sigma(\mathrm{diag}(-1,-1,-1,1))\cdot\sigma(\mathrm{diag}(1,1,1,-1))=(-1)^{l}\epsilon\eta.
\]
So $V(\lambda)=\begin{cases}
\mathbb{C}Z_{1}^{l}Z_{2}^{l}, & \text{if }\epsilon=-1\text{ and }(-1)^{l}\epsilon\eta=1,\\
0, & \text{otherwise},
\end{cases}$ and we get $m(\sigma,I)=\dim V(\lambda)$.

(2) By Proposition \ref{prop:LKT-O} and \ref{prop:LKT-sgn}, $\mathcal{A}(\pi)=\{\Lambda\}$
with $\Lambda=(0;-1)\otimes(;-1)$. So $\deg(\pi)\leqslant\deg(\Lambda)=4$.
Since $\phi_{3}(\Lambda)=0$, we have $\Lambda\notin\mathcal{D}(\pi)$.
By Lemma \ref{lem:parity} on the parity of degrees, $\deg(\pi)=2$
or $0$. The only $K$-type for $O(3,1)$  of degree $0$ is $(0;+1)\otimes(;+1)$,
which cannot occur in $\pi$ since it has the same norm as that of
the lowest $K$-type of $\pi$. So $\deg(\pi)=2$. 

All $K$-types for $O(3,1)$ of degree $2$ are $(1;+1)\otimes(;-1)$,
$(1;-1)\otimes(;+1)$, and $(2;+1)\otimes(;+1)$. But $(1;+1)\otimes(;-1)$
and $(2;+1)\otimes(;+1)$ do not occur in $I$ by (1), and thus not
in $\pi$. So $\mathcal{D}(\pi)=\{(1;-1)\otimes(;+1)\}$. So $\mathcal{A}(\theta_{3}(\pi))=\mathcal{D}(\theta_{3}(\pi))=\{(2,2,1)\}$
by Lemma \ref{lem:cond_n}.\end{proof}
\begin{prop}
Let $\pi=\pi_{1}((0),-1,\O,0,0,(-1),(\beta))$ with $\beta\in\mathbb{C}$.
Let $I$ be the standard module of parabolic induction to get $\pi$
as in Subsection \ref{subsec:Par-O}.

(1) For an $O(3)\times O(1)$-type $\sigma=(l;\epsilon)\otimes(;\eta)$,
the multiplicity of $\sigma$ in $I$ is \textup{
\[
m(\sigma,I)=\begin{cases}
1, & \text{if }\epsilon=-1\text{ and }\eta=(-1)^{l},\\
0, & \text{otherwise}.
\end{cases}
\]
}

(2) \textup{$\mathcal{A}(\theta_{3}(\pi))=\{(2,2,2)\}$.}\end{prop}
\begin{proof}
(1) Now $I=\mathrm{Ind}_{MAN}^{O(3,1)}(\det\otimes\chi_{-1,\beta}\otimes\tri_{N})$,
where $MA\cong O(2)\times GL(1,\mathbb{R})$, $\det$ is the determinant
character of $O(2)$, $\chi_{-1,\beta}$ is the character of $GL(1,\mathbb{R})$
that maps $x$ to $\mathrm{sgn}(x)|x|^{\beta}$, and $\tri_{N}$ is
the trivial representation of $N$. For $K=O(3)\times O(1)$, notice
the $K$-intertwining isomorphism $I|_{K}\cong\mathrm{Ind}_{K\cap MA}^{K}(\det\otimes\chi_{-1,\beta})$
(cf. \cite[Prop.4.1.12]{Vogan1981representations}). By Frobenius
reciprocity,
\begin{align*}
m(\sigma,I)=\dim\mathrm{Hom}_{K}(\sigma,I|_{K}) & =\dim\mathrm{Hom}_{K}(\sigma,\mathrm{Ind}_{K\cap MA}^{K}(\det\otimes\chi_{-1,\beta}))\\
 & =\dim\mathrm{Hom}_{K\cap MA}(\sigma,\det\otimes\chi_{-1,\beta}),
\end{align*}
where $K\cap MA=\{\mathrm{diag}(X,t,t)\mid X\in O(2),\ t=\pm1\}$.
Let $K_{0}=K\cap MA$, $\lambda'=(\det\otimes\chi_{-1,\beta})|_{K_{0}}$.
Then $\lambda'$ is the character of $K_{0}$ defined by $\lambda'(\mathrm{diag}(X,t,t))=\det(X)t$.
Write $V$ for the underlying space of $\sigma$. Then
\[
m(\sigma,I)=\dim_{K_{0}}\mathrm{Hom}(V|_{K_{0}},\lambda')=\dim V(\lambda'),
\]
where $V(\lambda')$ is the $\lambda'$-isotypic invariant subspace
of $V|_{K_{0}}.$ By a similar argument as the proof of the last proposition,
replacing $\lambda$ by $\lambda'$, we get
\[
V(\lambda')=\begin{cases}
\mathbb{C}Z_{1}^{l}Z_{2}^{l}, & \text{if }\epsilon=-1\text{ and }(-1)^{l+1}\epsilon\eta=1,\\
0, & \text{otherwise}.
\end{cases}
\]

(2) By Proposition \ref{prop:LKT-O} and \ref{prop:LKT-sgn}, $\mathcal{A}(\pi)=\{\Lambda\}$,
where $\Lambda=(0;-1)\otimes(;+1)$. So $\deg(\pi)\leqslant\deg(\Lambda)=3$.
By Lemma \ref{lem:parity} on the parity of degrees, $\deg(\pi)=1$
or $3$. List all $O(3)\times O(1)$-types of degree $1$ or $3$:
$(0;+1)\otimes(;-1)$, $(1;+1)\otimes(;+1)$, $(2;+1)\otimes(;-1)$,
$(3;+1)\otimes(;+1)$, $(0;-1)\otimes(;+1)$, $(1;-1)\otimes(;-1)$,
$(2;-1)\otimes(;+1)$. 

By (1), the first four $O(3)\times O(1)$-types in this list do not
occur in $\pi$, and the last three ones have degree $3$. So $\deg(\pi)=3$
and 
\begin{alignat*}{1}
(0;-1)\otimes(;+1)\in\mathcal{D}(\pi)\subseteq & \ \{(0;-1)\otimes(;+1),\ (1;-1)\otimes(;-1),\ (2;-1)\otimes(;+1)\},\\
(2,2,2)\in\mathcal{D}(\theta_{3}(\pi))\subseteq & \ \{(2,2,2),\ (2,2,0),\ (3,2,1)\},
\end{alignat*}
with norms $||(2,2,2)||<||(2,2,0)||<||(3,2,1)||$. By Lemma \ref{lem:cond_n},
 $\mathcal{A}(\theta_{3}(\pi))=\{(2,2,2)\}$.
\end{proof}

\begin{thm}
Let $\pi\in\mathcal{R}(O(3,1))$ with $n(\pi)=3$. Then the infinitesimal
character of $\theta_{3}(\pi)$ and the set $\mathcal{A}(\theta_{3}(\pi))$
determine a unique element in $\mathcal{R}(Sp(6,\mathbb{R}))$, which
is $\theta_{3}(\pi)$ listed in the following table.
\end{thm}
\noindent\resizebox{\textwidth}{!}{%

\begin{tabular}{c|c|c|c|c}
\hline 
\multirow{2}{*}{$\pi$} & inf.char. & \multirow{2}{*}{$\mathcal{A}(\theta_{3}(\pi))$} & \multirow{2}{*}{$\theta_{3}(\pi)$} & \multirow{2}{*}{with}\tabularnewline
 & of $\theta_{3}(\pi)$ &  &  & \tabularnewline
\hline 
\hline 
\multirow{2}{*}{$\pi_{-1}((m;),1,\O,0,0,(1),(0))$} & \multirow{2}{*}{$(m,0,1)$} & \multirow{2}{*}{$\{(m+1,2,0)\}$} & \multirow{2}{*}{$\pi((m),\{2e_{1}\},(1),(1),0,0)$} & \multirow{2}{*}{$1\leqslant m\in\mathbb{Z}$}\tabularnewline
 &  &  &  & \tabularnewline
\hline 
\multirow{2}{*}{$\pi_{1}((0;),-1,\O,0,0,(1),(0))$} & \multirow{2}{*}{$(0,0,1)$} & \multirow{4}{*}{$\{(2,2,1)\}$} & $\pi((1,0,0),\{e_{1}\pm e_{2},e_{2}\pm e_{3},$ & \multirow{2}{*}{}\tabularnewline
 &  &  & $e_{1}\pm e_{3},2e_{1},2e_{2},-2e_{3}\},0,0,0,0)$ & \tabularnewline
\cline{1-2} \cline{4-5} 
\multirow{2}{*}{$\pi_{1}((0;),-1,\O,0,0,(1),(\beta))$} & \multirow{4}{*}{$(\beta,0,1)$} &  & $\pi((1,0),\{2e_{1},2e_{2},e_{1}\pm e_{2}\},$ & $\beta\in$\tabularnewline
 &  &  & $0,0,(-1),(\beta))$ & $\mathbb{C}\backslash\{0,\pm1\}$\tabularnewline
\cline{1-1} \cline{3-5} 
\multirow{2}{*}{$\pi_{1}((0),-1,\O,0,0,(-1),(\beta))$} &  & \multirow{2}{*}{$\{(2,2,2)\}$} & $\pi((1,0),\{2e_{1},2e_{2},e_{1}\pm e_{2}\},$ & \multirow{2}{*}{$\beta\in\mathbb{C}$}\tabularnewline
 &  &  & $0,0,(1),(\beta))$ & \tabularnewline
\hline 
\end{tabular}

}
\begin{proof}
This can be checked case-by-case according to Appendix \ref{App:list}
which lists $\mathcal{A}(\pi')$ for all $\pi'\in\mathcal{R}(Sp(6,\mathbb{R}))$
with the infinitesimal character $(\beta,0,1)$ for all $\beta\in\mathbb{C}$.
\end{proof}

\section{\label{sec:4-lifts}Theta $4$-lifts of $\mathrm{det}_{p,q}$ when
$p+q=4$}

Let $\det_{p,q}$ denote the determinant character of $O(p,q)$. In
this section, we will calculate $\theta_{4}(\det_{p,q})$ explicitly
when $(p,q)=(4,0)$, $(3,1)$, $(2,2)$. The strategy is to determine
the desired theta lifts by its infinitesimal character and lowest
$K$-types.

First note that the infinitesimal character of $\det_{p,q}$ is $(0,1)$.
By Proposition \ref{prop:CrspIF}, the infinitesimal character of
$\theta_{4}(\det_{p,q})$ is $(0,1,1,2)$. Also note that $\mathcal{A}(\theta_{4}(\det_{p,q}))\subseteq\mathcal{D}(\theta_{4}(\det_{p,q}))$
by Lemma \ref{lem:cond_n}. The only $K$-type of $\det_{2,2}$ is
$(0;-1)\otimes(0;-1)$. By Proposition \ref{prop:CorKT} and Lemma
\ref{lem:min-deg},
\[
\mathcal{A}(\theta_{4}(\mathrm{det}_{2,2}))=\mathcal{D}(\theta_{4}(\mathrm{det}_{2,2}))=\{\phi_{4}((0;-1)\otimes(0;-1))\}=\{(1,1,-1,-1)\}.
\]
Similarly, the only $K$-type of $\det_{3,1}$ is $(0;-1)\otimes(;-1)$
and 
\[
\mathcal{A}(\theta_{4}(\mathrm{det}_{3,1}))=\mathcal{D}(\theta_{4}(\mathrm{det}_{3,1}))=\{\phi_{4}((0;-1)\otimes(;-1))\}=\{(2,2,2,0\};
\]
the only $K$-type of $\det_{4,0}$ is $(0,0;-1)\otimes(;)$ and 
\[
\mathcal{A}(\theta_{4}(\mathrm{det}_{4,0}))=\mathcal{D}(\theta_{4}(\mathrm{det}_{4,0}))=\{\phi_{4}((0,0;-1)\otimes(;))\}=\{(3,3,3,3)\}.
\]

\subsection{Determinant for $O(2,2)$}

Let us calculate $\theta_{2p}(\det_{p,p})\in\mathcal{R}(Sp(4p,\mathbb{R}))$
for any $p\geqslant1$. The infinitesimal character of $\det_{p,p}$
is $(0,1,2,\dots,p-1)$. By Proposition \ref{prop:CrspIF}, the infinitesimal
character of $\theta_{2p}(\mathrm{det}_{p,p})$ is 
\[
(0,1,1,2,2,\dots,p-1,p-1,p)
\]
The only $O(p)\times O(p)$-type of $\det_{p,p}$ is the $(0,\dots,0;-1)\otimes(0,\dots,0;-1)$,
so 
\[
\mathcal{A}(\theta_{2p}(\mathrm{det}_{p,p}))=\mathcal{D}(\theta_{2p}(\mathrm{det}_{p,p}))=\{(\underbrace{1,\dots,1}_{p},\underbrace{-1,\dots,-1}_{p})\}.
\]

\begin{prop}
\label{prop:det_pp} There exists a unique $\pi'\in\mathcal{R}(Sp(4p,\mathbb{R}))$
with $\mathcal{A}(\pi')=\{(\underbrace{1,\dots,1}_{p},\underbrace{-1,\dots,-1}_{p})\}$
and the infinitesimal character $(0,1,1,2,2,\dots,p-1,p-1,p)$. Indeed,
\textup{
\[
\pi'=\theta_{2p}(\mathrm{det}_{p,p})=\pi(0,\O,(1,1,\dots,1),(1,3,5,\dots,2p-1),0,0).
\]
}\end{prop}
\begin{proof}
Let $\pi'=\pi(\lambda_{d},\Psi,\mu,\nu,\varepsilon,\kappa)$. In Proposition
\ref{prop:LKT-Sp} to calculate $\mathcal{A}(\pi')$,
\begin{align*}
 & \lambda_{a}+\rho(\mathfrak{u}\cap\mathfrak{p})-\rho(\mathfrak{u}\cap\mathfrak{k})\\
=\  & (\underbrace{\beta_{1},\dots,\beta_{1}}_{u_{1}},\dots,\underbrace{\beta_{m},\dots\beta_{m}}_{u_{m}},\underbrace{u-r,\dots,u-r}_{w},\underbrace{\gamma_{m},\dots,\gamma_{m}}_{r_{1}},\dots,\underbrace{\gamma_{1},\dots,\gamma_{1}}_{r_{m}})
\end{align*}
with $\beta_{1}\geqslant\cdots\geqslant\beta_{m}\geqslant u-r+1>u-r-1\geqslant\gamma_{m}\geqslant\cdots\geqslant\gamma_{1}$.
Here $u=\sum_{i=1}^{m}u_{i}$, $r=\sum_{i=1}^{m}r_{i}$, and $w$
are the numbers of positive, negative, and zero entries in $\lambda_{a}$
respectively. To get the lowest $K$-type $\sigma=(\underbrace{1,\dots,1}_{p},\underbrace{-1,\dots,-1}_{p})$,
every $\beta_{i}$ with $u_{i}>0$ or $\gamma_{i}$ with $r_{i}>0$
should lie in $\{0,\pm1,\pm\frac{3}{2}\}$. 

First we show $u=r$. If $u>r$, then all $\beta_{i}\geqslant u-r+1\geqslant2$
cannot occur, and thus $0=u>r$. If $u<r$ then all $\gamma_{i}\leqslant u-r-1\leqslant-2$
and thus $0=r>u$. There is always a contradiction.

As $u=r$, the lowest $K$-type $\sigma$ is of the form
\[
(\underbrace{*,\dots,*}_{u},\underbrace{1,\dots,1}_{h},\underbrace{0,\dots,0}_{w-h},\underbrace{*,\dots,*}_{u})\text{ or }(\underbrace{*,\dots,*}_{u},\underbrace{0,\dots,0}_{w-h},\underbrace{-1,\dots,-1}_{h},\underbrace{*,\dots,*}_{u}).
\]
Since $\sigma$ contains only $p$ entries $1$ and $p$ entries $-1$,
we have $w=h=0$. Therefore, $\lambda_{d}$ and $\mu$ contain no
zero entry, and $(\varepsilon,\kappa)$ do not occur.

Let $\alpha_{1}$ be the maximal absolute value of entries of $\lambda_{a}$.
Clearly, $\frac{1}{2}\leqslant\alpha_{1}\in\frac{1}{2}\mathbb{Z}$.
Since $u_{1}>0$ or $r_{1}>0$, we have $\beta_{1}=\alpha_{1}+\frac{u_{1}-r_{1}}{2}+\frac{1}{2}\leqslant\frac{3}{2}$
or $\gamma_{1}=-\alpha_{1}+\frac{u_{1}-r_{1}}{2}-\frac{1}{2}\geqslant-\frac{3}{2}$.
As $|u_{1}-r_{1}|\leqslant1$, we get $\alpha_{1}\leqslant\frac{3}{2}$.
Moreover, if $\alpha_{1}\not\in\mathbb{Z}$, then $u_{1}=r_{1}$ and
thus $\alpha_{1}\leqslant1$. So $\alpha_{1}\in\{\frac{1}{2},1\}$.

 So $\lambda_{d}=(\underbrace{1,\dots,1}_{k},\underbrace{-1,\dots,-1}_{l})$.
As $k-l=u-r=0$, we have $k=l$. Therefore, $u_{1}=r_{1}>0$.

We show $\alpha_{1}\neq1$. Otherwise, suppose $\alpha_{1}=1$. Then
$\beta_{1}=\alpha_{1}+\frac{u_{1}-r_{1}}{2}+\frac{1}{2}=\frac{3}{2}$
and $\gamma_{1}=-\frac{3}{2}$, both of which occur as entries of
$\lambda_{a}+\rho(\mathfrak{u}\cap\mathfrak{p})-\rho(\mathfrak{u}\cap\mathfrak{k})$.
In spite of the choice of $\pm\frac{1}{2}$ adding into $\beta_{1}$
and $\gamma_{1}$, we get $2$ or $-2$ as an entry of the lowest
$K$-type. This is not true for $\sigma$.

So $\alpha_{1}=\frac{1}{2}$. Therefore, $\lambda_{d}$ does not occur
and $\mu_{i}=1$ for all $i$.

In all, $\pi'=\pi(0,\O,\mu,\nu,0,0)$ with $\mu=(1,1,\dots,1)$. The
infinitesimal character of $\pi'$ is
\[
(0,1,1,\dots,p-1,p-1,p)\sim(\frac{1+\nu_{1}}{2},\frac{1-\nu_{1}}{2},\frac{1+\nu_{2}}{2},\frac{1-\nu_{2}}{2},\dots,\frac{1+\nu_{p}}{2},\frac{1-\nu_{p}}{2}).
\]
Here $\sim$ denotes the equivalence up to permutations and sign-changes
of coordinates. So $\nu=(1,3,5,\dots,2p-1)$ is the only choice (up
to equivalence). By uniqueness, $\pi'=\theta_{2p}(\det_{p,p})$.
\end{proof}

\subsection{Determinant for $O(3,1)$}
\begin{prop}
There exists a unique $\pi'\in\mathcal{R}(Sp(8,\mathbb{R}))$ with
infinitesimal character $(0,1,1,2)$ and $\mathcal{A}(\pi')=\{(2,2,2,0)\}$,
which is \textup{
\[
\theta_{4}(\mathrm{det}_{3,1})=\pi((1,0),\Psi,(1),(3),0,0)
\]
with} $\Psi=\{e_{1}\pm e_{2},2e_{1},2e_{2}\}$.\end{prop}
\begin{proof}
Let $\pi'=\pi(\lambda,\Psi,\mu,\nu,\varepsilon,\kappa)$. Recall the
algorithm Proposition \ref{prop:LKT-Sp} to calculate $\mathcal{A}(\pi')=\{\sigma\}$.

First we show $u-r=1$. If $u-r\geqslant2$, then all $\beta_{i}\geqslant u-r+1\geqslant3$
cannot occur, and thus $0=u>r$, which makes a contradiction. If $u-r\leqslant0$,
then all $\gamma_{i}\leqslant u-r-1\leqslant-1$ cannot occur and
$0=r\geqslant u$, and thus $\sigma$ contains only entries in $\{0,\pm1\}$,
which makes a contradiction.

As $u-r=1$ and $u+w+r=4$, we see $(u,w,r)=(2,1,1)$ or $(1,3,0)$.
Note that $\sigma=(2,2,2,0)$ is of the form
\[
(\underbrace{*,\dots,*,}_{u}\underbrace{2,\dots,2}_{h},\underbrace{1,\dots,1}_{w-h},\underbrace{*,\dots,*}_{r})\text{ or }(\underbrace{*,\dots,*,}_{u}\underbrace{1,\dots,1}_{w-h},\underbrace{0,\dots,0}_{h},\underbrace{*,\dots,*}_{r}),
\]
so $h=w$. The second form cannot happen, otherwise $w+r=1$ (the
number of zero entries). So $r=1$ and $(u,w,r)=(2,1,1)$. If $z=0$,
both forms above should occur, which makes lowest $K$-types not unique.
So $z>0$. As $z\leqslant w$, we get $z=w=1$. Therefore, $\lambda_{d}$
contains exactly one zero entry, $\mu$ contains no zero entry, and
$(\varepsilon,\kappa)$ do not occur. Hence $(v,s,t)=(2,1,0)$ or
$(4,0,0)$. 

We show $(v,s,t)\neq(4,0,0)$. Otherwise, $(k,l)=(u,r)=(2,1)$, and
thus $\lambda_{d}=(2,1,0,-1)$ and $\lambda_{a}+\rho(\mathfrak{u}\cap\mathfrak{p})-\rho(\mathfrak{u}\cap\mathfrak{k})=(3,\frac{5}{2},1,-\frac{1}{2})$,
which cannot give the lowest $K$-type $\sigma$.

Thus $(v,s,t)=(2,1,0)$. As $k-l=u-r=1$ and $k+z+l=v$, we get $(k,z,l)=(1,1,0)$.
So $\lambda_{d}=(a_{1},0)$ for $1\leqslant a_{1}\in\mathbb{Z}$.
Now $(a_{1},\frac{\mu_{1}+\nu_{1}}{2},\frac{-\mu_{1}+\nu_{1}}{2})\sim(1,1,2)$.
Note that $\mu_{1}>0$, and that $\mu_{1}$ should be odd if $\nu_{1}=0$.
So $(\frac{\mu_{1}+\nu_{1}}{2},\frac{-\mu_{1}+\nu_{1}}{2})\not\sim(1,1)$.
Thus $(\frac{\mu_{1}+\nu_{1}}{2},\frac{-\mu_{1}+\nu_{1}}{2})\sim(1,2)$
and $\lambda_{d}=(1,0)$. Then $(a_{1},\mu_{1},\nu_{1})=(1,1,\pm3)$
or $(1,3,\pm1)$, and $\lambda_{a}=(1,\frac{1}{2},0,-\frac{1}{2})$
or $(\frac{3}{2},1,0,-\frac{3}{2})$, with the corresponding $\lambda_{a}+\rho(\mathfrak{u}\cap\mathfrak{p})-\rho(\mathfrak{u}\cap\mathfrak{k})=(2,2,1,0)$
or $(2,2,1,-2)$ respectively. But the latter one cannot give the
lowest $K$-type $\sigma$. So $\lambda_{a}=(1,\frac{1}{2},0,-\frac{1}{2})$,
$\mu=(1)$, and $\nu\sim(3)$. In the algorithm, to add $h$ entries
$1$ on the $w$ entries, we have $2e_{2}\in\Psi$. 

By uniqueness, $\pi'=\theta_{4}(\det_{3,1})=\pi((1,0),\Psi,(1),(3),0,0)$
with $\Psi=\{e_{1}\pm e_{2},2e_{1},2e_{2}\}$.
\end{proof}

\subsection{Determinant for $O(4,0)$}
\begin{prop}
\label{prop:lift-det40} There exists a unique $\pi'\in\mathcal{R}(Sp(8,\mathbb{R}))$
with infinitesimal character $(0,1,1,2)$ and $\mathcal{A}(\pi')=\{(3,3,3,3)\}$,
which is 
\[
\theta_{4}(\mathrm{det}_{4,0})=\pi((2,1,0),\Psi,0,0,(-1),(1)),
\]
with $\Psi=\{e_{1}\pm e_{2},e_{2}\pm e_{3},e_{1}\pm e_{3},2e_{1},2e_{2},2e_{3}\}$.\end{prop}
\begin{proof}
Let $\pi'=\pi(\lambda,\Psi,\mu,\nu,\varepsilon,\kappa)$. Recall the
algorithm Proposition \ref{prop:LKT-Sp} to calculate $\mathcal{A}(\pi')=\{\sigma\}$.

First we show $u-r=2$. If $u-r\geqslant3$, then all $\beta_{i}\geqslant u-r+1\geqslant4$
cannot occur, and thus $0=u>r$, which makes a contradiction. If $u-r\leqslant1$,
then all $\gamma_{i}\leqslant u-r-1\leqslant0$ cannot occur, and
thus $0=r\geqslant u-1$ and $w=4-u\geqslant3$. Therefore, $w$ entries
$(u-r,\dots,u-r)$ occur, and thus $u-r\in\{2,3,4\}$, which makes
a contradiction.

As all $\gamma_{i}\leqslant u-r-1=1$ cannot occur, $r=0$. So $u=2$,
$w=2$. Clearly $h=w=2$. If $z=0$, the lowest $K$-types will not
be unique. So $z\geqslant1$. And $z-\left[\frac{z+1}{2}\right]\leqslant w-h=0$,
so $z=1$. Since $w=h=2$, we get all $\varepsilon_{i}=-1$.

As $k-l=u-r=2$ and $z=1$, the odd integer $v=k+z+l\geqslant3$.
So $(v,s,t)=(3,0,1)$  and $(k,l)=(2,0)$. Note that $(\lambda_{d}\mid\kappa)\sim(0,1,1,2)$,
and $\lambda_{d}$ contains an entry $0$. So $\lambda_{d}=(2,1,0)$,
and $\kappa\sim(1)$. In the algorithm, to add $h$ entries $1$ on
the $w$ entries, we have $2e_{3}\in\Psi$.

By uniqueness, $\pi'=\theta_{4}(\mathrm{det}_{4,0})=\pi((2,1,0),\Psi,0,0,(-1),(1))$
with $\Psi=\{e_{1}\pm e_{2},e_{2}\pm e_{3},e_{1}\pm e_{3},2e_{1},2e_{2},2e_{3}\}$.
\end{proof}
We collect the results of this section in the following table.

\smallskip{}\noindent\resizebox{\textwidth}{!}{%

\begin{tabular}{c|c|c|c|c|c}
\hline 
\multirow{2}{*}{$(p,q)$} & \multirow{2}{*}{$\mathrm{det}_{p,q}$} & inf. char.  & \multirow{2}{*}{$\mathcal{A}(\theta_{4}(\mathrm{det}_{p,q}))$} & \multirow{2}{*}{$\theta_{4}(\mathrm{det}_{p,q})$} & \multirow{2}{*}{$\Psi$}\tabularnewline
 &  & of $\theta_{4}(\mathrm{det}_{p,q})$ &  &  & \tabularnewline
\hline 
\hline 
\multirow{2}{*}{$(2,2)$} & $\pi_{-1}(0,1,\O,0,0,$ & \multirow{6}{*}{$(0,1,1,2)$} & $\{(1,1,-1,$ & \multirow{2}{*}{$\pi(0,\O,(1,1),(1,3),0,0)$} & \multirow{2}{*}{$\O$}\tabularnewline
 & $(1,1),(0,1))$ &  & $-1)\}$ &  & \tabularnewline
\cline{1-2} \cline{4-6} 
\multirow{2}{*}{$(3,1)$} & $\pi_{1}((0;),-1,\O,$ &  & \multirow{2}{*}{$\{(2,2,2,0\}$} & \multirow{2}{*}{$\pi((1,0),\Psi,(1),(3),0,0)$} & \multirow{2}{*}{$\{e_{1}\pm e_{2},2e_{1},2e_{2}\}$}\tabularnewline
 & $0,0,(1),(1))$ &  &  &  & \tabularnewline
\cline{1-2} \cline{4-6} 
\multirow{2}{*}{$(4,0)$} & $\pi_{1}((1,0;),-1,$ &  & \multirow{2}{*}{$\{(3,3,3,3)\}$} & \multirow{2}{*}{$\pi((2,1,0),\Psi,0,0,(-1),(1))$} & $\{e_{1}\pm e_{2},e_{2}\pm e_{3},$\tabularnewline
 & $\{e_{1}\pm e_{2}\},0,0,0,0)$ &  &  &  & $e_{1}\pm e_{3},2e_{1},2e_{2},2e_{3}\}$\tabularnewline
\hline 
\end{tabular}

}\smallskip{}

\section*{Acknowledgments}

This paper covers the main parts of the author's doctoral dissertation
in HKUST. It was written during his stay as a post-doctor in Peking
University, and completed in Sun Yat-sen University. The work was
done partially while the author was visiting the Institute for Mathematical
Sciences, National University of Singapore in 2016. The visit was
supported by the Institute. The author would like to express his deep
gratitude to his thesis advisor Professor Jian-Shu Li for guidance
to this problem and enlightening discussions, and for providing various
help during and after his PhD life. The author would also like to
thank Professor Roger Howe, Professor Binyong Sun and Doctor Yixin
Bao for inspiring conversations about this subject. 

\appendix

\section{\label{App:CompLKT}Algorithms to compute Lowest $K$-types}

In this appendix, we quote from \cite{Paul2005howe} the algorithms
for $\mathcal{R}(Sp(2n,\mathbb{R}))$ and $\mathcal{R}(O(p,q))$ (with
$p+q$ even) to calculate the set of lowest $K$-types from the Langlands
parameters.

\subsection{For $Sp(2n,\mathbb{R})$}
\begin{prop}
[{\cite[Prop.6]{Paul2005howe}}] \label{prop:LKT-Sp} Let $\pi=\pi(\lambda_{d},\Psi,\mu,\nu,\varepsilon,\kappa)\in\mathcal{R}(Sp(2n,\mathbb{R}))$,
\begin{align*}
\lambda_{d} & =(\underbrace{a_{1},\dots,a_{1}}_{k_{1}},\underbrace{a_{2},\dots,a_{2}}_{k_{2}},\dots,\underbrace{a_{b},\dots a_{b}}_{k_{b}},\underbrace{0,\dots,0}_{z},\\
 & \quad\ \underbrace{-a_{b},\dots,-a_{b}}_{l_{b}},\dots,\underbrace{-a_{2},\dots,-a_{2}}_{l_{2}},\underbrace{-a_{1},\dots,-a_{1}}_{l_{1}})\in\mathbb{Z}^{v},
\end{align*}
with integers $a_{1}>a_{2}>\cdots>a_{b}>0$, and $|k_{i}-l_{i}|\leqslant1$
for all $i$. Let $k=\sum_{i=1}^{b}k_{i}$, $l=\sum_{i=1}^{b}l_{i}$,
$\tilde{k}_{j}=\sum_{i=1}^{j}k_{i}$, and $\tilde{l}_{j}=\sum_{i=1}^{j}l_{i}$.
Notice that $k+z+l=v$.

Let $\lambda_{a}$ be obtained from $(\lambda_{d}\mid\frac{\mu_{1}}{2},\dots,\frac{\mu_{s}}{2},\underbrace{0,\dots,0}_{t},-\frac{\mu_{s}}{2},\dots,-\frac{\mu_{1}}{2})$
by reordering of the coordinates so that the resulting entries are
nonincreasing. Write 
\begin{align*}
\lambda_{a} & =(\underbrace{\alpha_{1},\dots,\alpha_{1}}_{u_{1}},\underbrace{\alpha_{2},\dots,\alpha_{2}}_{u_{2}},\dots,\underbrace{\alpha_{m},\dots\alpha_{m}}_{u_{m}},\underbrace{0,\dots,0}_{w},\\
 & \quad\ \underbrace{-\alpha_{m},\dots,-\alpha_{m}}_{r_{m}},\dots,\underbrace{-\alpha_{2},\dots,-\alpha_{2}}_{r_{2}},\underbrace{-\alpha_{1},\dots,-\alpha_{1}}_{r_{1}})\in\left(\frac{1}{2}\mathbb{Z}\right)^{n},
\end{align*}
with $\alpha_{1}>\alpha_{2}>\cdots>\alpha_{m}>0$. Then for all $i$
we have $\alpha_{i}\in\frac{1}{2}\mathbb{Z}$ and $|u_{i}-r_{i}|\leqslant1$.
If $u_{i}\neq r_{i}$, then $\alpha_{i}\in\mathbb{Z}$. Let $u=\sum_{i=1}^{m}u_{i}$
and $r=\sum_{i=1}^{m}r_{i}$. Notice that $u-r=k-l$.

For $Sp(2n,\mathbb{R})$, the root system is 
\[
\Delta=\{\pm e_{i}\pm e_{j}\mid1\leqslant i<j\leqslant n\}\cup\{\pm2e_{i}\mid1\leqslant i\leqslant n\}.
\]
We fix the standard set of positive compact roots 
\[
\Delta_{c}^{+}=\{e_{i}-e_{j}\mid1\leqslant i<j\leqslant n\}.
\]
Let $\rho(\mathfrak{u}\cap\mathfrak{p})$ and $\rho(\mathfrak{u}\cap\mathfrak{k})$
be one-half sums of the noncompact and compact roots with respect
to which $\lambda_{a}$ is strictly dominant, respectively. Then $\lambda_{a}+\rho(\mathfrak{u}\cap\mathfrak{p})-\rho(\mathfrak{u}\cap\mathfrak{k})$
\begin{align*}
=(\underbrace{\beta_{1},\dots,\beta_{1}}_{u_{1}},\dots, & \underbrace{\beta_{m},\dots\beta_{m}}_{u_{m}},\underbrace{u-r,\dots,u-r}_{w},\underbrace{\gamma_{m},\dots,\gamma_{m}}_{r_{m}},\dots,\underbrace{\gamma_{1},\dots,\gamma_{1}}_{r_{1}}),\\
with\quad\beta_{i} & =\alpha_{i}+\frac{1}{2}+\frac{u_{i}-r_{i}}{2}+\sum_{1\leqslant j<i}(u_{j}-r_{j}),\\
\gamma_{i} & =-\alpha_{i}-\frac{1}{2}+\frac{u_{i}-r_{i}}{2}+\sum_{1\leqslant j<i}(u_{j}-r_{j}).
\end{align*}
Then the lowest $K$-types of $\pi$ are precisely those of the form
\begin{eqnarray*}
\Lambda & = & \lambda_{a}+\rho(\mathfrak{u}\cap\mathfrak{p})-\rho(\mathfrak{u}\cap\mathfrak{k})+\delta_{L},\\
\text{with}\quad\delta_{L} & = & (\underbrace{\delta_{1},\dots,\delta_{1}}_{u_{1}},\dots,\underbrace{\delta_{m},\dots\delta_{m}}_{u_{m}},\eta_{1},\dots,\eta_{w},\underbrace{\delta_{m},\dots,\delta_{m}}_{r_{m}},\dots,\underbrace{\delta_{1},\dots,\delta_{1}}_{r_{1}}),
\end{eqnarray*}
satisfying the following conditions:

(1) If $\beta_{i}\in\mathbb{Z}$, then $\delta_{i}=0$.

(2) Suppose $\beta_{i}\in\mathbb{Z}+\frac{1}{2}$. Then $\delta_{i}\in\{\pm\frac{1}{2}\}$.
If $\alpha_{i}$ does not occur as an entry in $\lambda_{d}$ then
both choices occur. If $\alpha_{i}=a_{j}$, then $\delta_{i}=\frac{1}{2}$
if $e_{\tilde{k}_{j-1}+1}+e_{v-\tilde{l}_{j}+1}\in\Psi$, and $\delta_{i}=-\frac{1}{2}$
otherwise. 

(3) We have $(\eta_{1},\dots,\eta_{w})=(\underbrace{1,\dots1}_{h},0,\dots,0)$
or $(0,\dots,\underbrace{-1,\dots,-1}_{h})$, with $h=\left[\frac{z+1}{2}\right]+\#\{j\mid\mu_{j}=0\}+\#\{j\mid\varepsilon_{j}=(-1)^{u-r+1}\}$.
If $z=0$ then both choices occur. If $z>0$, then $(\eta_{1},\dots,\eta_{w})$
is of the first form whenever $e_{k+1}+e_{k+z}\in\Psi$ (this includes
the case $z=1$ where the condition becomes $2e_{k+1}\in\Psi$), and
of the second form otherwise.
\end{prop}

\subsection{For $O(p,q)$ with $p+q$ even}
\begin{prop}
[{\cite[Prop.10]{Paul2005howe}}] \label{prop:LKT-O} Let $\pi=\pi_{\zeta}(\lambda_{d},\xi,\Psi,\mu,\nu,\varepsilon,\kappa)\in\mathcal{R}(O(p,q))$
with $p+q$ even. Let $ $
\begin{align*}
\lambda_{d} & =(\underbrace{a_{1},\dots,a_{1}}_{k_{1}},\underbrace{a_{2},\dots,a_{2}}_{k_{2}},\dots,\underbrace{a_{b},\dots a_{b}}_{k_{b}},\underbrace{0,\dots,0}_{z};\\
 & \quad\ \underbrace{a_{1},\dots,a_{1}}_{l_{1}},\underbrace{a_{2},\dots,a_{2}}_{l_{2}},\dots,\underbrace{a_{b},\dots,a_{b}}_{l_{b}},\underbrace{0,\dots,0}_{z'})
\end{align*}
 with integers $a_{1}>a_{2}>\cdots>a_{b}>0$, $|k_{i}-l_{i}|\leqslant1$
for all $i$, and $|z-z'|\leqslant1$. Let $k=\sum_{i=1}^{b}k_{i}$,
$l=\sum_{i=1}^{b}l_{i}$, $\tilde{k}_{j}=\sum_{i=1}^{j}k_{i}$, and
$\tilde{l}_{j}=\sum_{i=1}^{j}l_{i}$.

Let $\lambda_{a}$ be obtained from $(\lambda_{d}\mid(\frac{\mu_{1}}{2},\dots,\frac{\mu_{s}}{2},\underbrace{0,\dots,0}_{\left[\frac{t}{2}\right]};\ \frac{\mu_{1}}{2},\dots,\frac{\mu_{s}}{2},\underbrace{0,\dots,0}_{\left[\frac{t}{2}\right]})$,
by reordering of the coordinates so that the resulting entries in
both parts are nonincreasing. Write 
\begin{align*}
\lambda_{a} & =(\underbrace{\alpha_{1},\dots,\alpha_{1}}_{u_{1}},\underbrace{\alpha_{2},\dots,\alpha_{2}}_{u_{2}},\dots,\underbrace{\alpha_{m},\dots a_{m}}_{u_{m}},\underbrace{0,\dots,0}_{x};\\
 & \quad\ \underbrace{\alpha_{1},\dots,\alpha_{1}}_{r_{1}},\underbrace{\alpha_{2},\dots,\alpha_{2}}_{r_{2}},\dots,\underbrace{\alpha_{m},\dots a_{m}}_{r_{m}},\underbrace{0,\dots,0}_{y})
\end{align*}
with $\alpha_{1}>\alpha_{2}>\cdots>\alpha_{m}>0$. Then for all $i$
we have $\alpha_{i}\in\frac{1}{2}\mathbb{Z}$, $|u_{i}-r_{i}|\leqslant1$
and $|x-y|\leqslant1$. If $u_{i}\neq r_{i}$, then $\alpha_{i}\in\mathbb{Z}$.
Let $u=\sum_{i=1}^{m}u_{i}$ and $r=\sum_{i=1}^{m}r_{i}$. Notice
that $u-r=k-l$ and $x-y=z-z'$. 

We take the set of roots $\Delta$ for $O(p,q)$, and fix the standard
set of positive compact roots $\Delta_{c}^{+}$, which are as in the
following table (with $p_{0}=\left[\frac{p}{2}\right]$ and $q_{0}=\left[\frac{q}{2}\right]$):

\smallskip{}\noindent\resizebox{\textwidth}{!}{%

\begin{tabular}{c|c|c}
\hline 
 & if $p,q$ are both even & if $p,q$ are both odd\tabularnewline
\hline 
\hline 
\multirow{3}{*}{$\Delta$} & $\{\pm e_{i}\pm e_{j}\mid1\leqslant i<j\leqslant p_{0}\}$ & $\{\pm e_{i}\pm e_{j}\mid1\leqslant i<j\leqslant p_{0}\}$\tabularnewline
 & $\cup\{\pm f_{i}\pm f_{j}\mid1\leqslant i<j\leqslant q_{0}\}$ & $\cup\{\pm f_{i}\pm f_{j}\mid1\leqslant i<j\leqslant q_{0}\}$\tabularnewline
 & $\cup\{\pm e_{i}\pm f_{j}\mid1\leqslant i\leqslant p_{0},\ 1\leqslant j\leqslant q_{0}\}$ & $\cup\{\pm e_{i}\pm f_{j},\pm e_{i},\pm f_{j}\mid1\leqslant i\leqslant p_{0},\ 1\leqslant j\leqslant q_{0}\}$\tabularnewline
\hline 
\multirow{3}{*}{$\Delta_{c}^{+}$} & $\{e_{i}\pm e_{j}\mid1\leqslant i<j\leqslant p_{0}\}$ & $\{e_{i}\pm e_{j}\mid1\leqslant i<j\leqslant p_{0}\}$\tabularnewline
 & $\cup\{\pm f_{i}\pm f_{j}\mid1\leqslant i<j\leqslant q_{0}\}$ & $\cup\{f_{i}\pm f_{j}\mid1\leqslant i<j\leqslant q_{0}\}$\tabularnewline
 &  & $\cup\{e_{i},f_{j}\mid1\leqslant i\leqslant p_{0},\ 1\leqslant j\leqslant q_{0}\}$\tabularnewline
\hline 
\end{tabular}

}\smallskip{}

Let $\rho(\mathfrak{u}\cap\mathfrak{p})$ and $\rho(\mathfrak{u}\cap\mathfrak{k})$
be one-half sums of the noncompact and compact roots with respect
to which $\lambda_{a}$ is strictly dominant, respectively. Then $\lambda_{a}+\rho(\mathfrak{u}\cap\mathfrak{p})-\rho(\mathfrak{u}\cap\mathfrak{k})$
\begin{gather*}
=(\underbrace{\beta_{1},\dots,\beta_{1}}_{u_{1}},\dots,\underbrace{\beta_{m},\dots\beta_{m}}_{u_{m}},\underbrace{0,\dots,0}_{x};\underbrace{\gamma_{1},\dots,\gamma_{1}}_{r_{1}},\dots,\underbrace{\gamma_{m},\dots,\gamma_{m}}_{r_{m}},\underbrace{0,\dots,0}_{y}),
\end{gather*}
\begin{align*}
\beta_{i} & =\alpha_{i}+\frac{1}{2}-u+r+\frac{u_{i}-r_{i}}{2}+\sum_{1\leqslant j<i}(u_{j}-r_{j}),\\
\gamma_{i} & =\alpha_{i}+\frac{1}{2}+u-r-\frac{u_{i}-r_{i}}{2}-\sum_{1\leqslant j<i}(u_{j}-r_{j}).
\end{align*}
Then the highest weights of the lowest $K$-types of $\pi$ are precisely
those of the form 
\begin{align*}
\Lambda_{0} & =\lambda_{a}+\rho(\mathfrak{u}\cap\mathfrak{p})-\rho(\mathfrak{u}\cap\mathfrak{k})+\delta_{L},\\
\text{with }\delta_{L} & =(\underbrace{\delta_{1},\dots,\delta_{1}}_{u_{1}},\dots,\underbrace{\delta_{m},\dots\delta_{m}}_{u_{m}},\eta_{1},\dots,\eta_{x};\underbrace{-\delta_{1},\dots,-\delta_{1}}_{r_{1}},\dots,\underbrace{-\delta_{m},\dots,-\delta_{m}}_{r_{m}},\xi_{1},\dots\xi_{y}),
\end{align*}
satisfying the following conditions:

(1) If $\beta_{i}\in\mathbb{Z}$, then $\delta_{i}=0$.

(2) Suppose $\beta_{i}\in\mathbb{Z}+\frac{1}{2}$. Then $\delta_{i}\in\{\pm\frac{1}{2}\}$.
If $\alpha_{i}$ does not occur as an entry in $\lambda_{d}$ then
both choices occur. If $\alpha_{i}=a_{j}$, then $\delta_{i}=\frac{1}{2}$
if $e_{\tilde{k}_{j}}-f_{\tilde{l}_{j}}\in\Psi$, and $\delta_{i}=-\frac{1}{2}$
otherwise.

(3) We have $(\eta_{1},\dots,\eta_{x};\xi_{1},\dots\xi_{y})=(\underbrace{1,\dots,1}_{h},0,\dots,0;0,\dots,0)$
or $(0,\dots0;\underbrace{1,\dots,1}_{h},0,\dots,0)$, with $h=\min\{z,z'\}+\#\{j\mid\mu_{j}=0\}+\min\{\beta,\gamma\}$,
where $\beta=\#\{j\mid\varepsilon_{j}=1\}$ and $\gamma=\#\{j\mid\varepsilon_{j}=-1\}$.
If $z+z'=0$ then both choices occur. If $z+z'>0$, then $(\eta_{1},\dots,\eta_{w})$
is of the first form whenever $e_{k+z}-f_{l+z'}\in\Psi$, and of the
second form otherwise.
\end{prop}

\begin{prop}
[{\cite[Prop.13]{Paul2005howe}}] \label{prop:LKT-sgn} Let $\pi=\pi_{\zeta}(\lambda_{d},\xi,\Psi,\mu,\nu,\varepsilon,\kappa)\in\mathcal{R}(O(p,q))$
with $p+q$ even. Let $\Lambda_{0}=(\Lambda_{1},\Lambda_{2})$ be
the highest weight of a lowest $K$-type of $\pi$ as in Proposition
\ref{prop:LKT-O}, with associated $z$, $z'$, $\beta$ and $\gamma$.
 Then the signs of any lowest $K$-type of $\pi$ with highest $\Lambda_{0}$
are given as follows:

(1) Suppose $z+z'=0$ and all $\kappa_{i}\neq0$. If $\beta\geqslant\gamma$
then both $(1;1)$ and $(-1;-1)$ occur as signs. (The resulting two
K-types may coincide.) If $\beta<\gamma$ then both $(1;-1)$ and
$(-1;1)$ occur as signs.

(2) Suppose $z+z'=0$ and some $(\varepsilon_{i},\kappa_{i})=(1,0)$.
If $\beta\geqslant\gamma$ then the signs are $(\zeta;\zeta)$. If
$\beta<\gamma$ then the signs are $(\zeta;-\zeta)$ if $\Lambda_{1}$
has more zeros then $\Lambda_{2}$, and $(-\zeta;\zeta)$ otherwise.

(3) Suppose $z+z'=0$ and some $(\varepsilon_{i},\kappa_{i})=(-1,0)$.
If $\beta\geqslant\gamma$ then the signs are $(\zeta;\zeta)$ if
$\Lambda_{1}$ has more zeros then $\Lambda_{2}$, and $(-\zeta,-\zeta)$
otherwise. If $\beta<\gamma$, then the signs are $(\zeta;-\zeta)$.

(4) Suppose $z+z'>0$ and $\beta\geqslant\gamma$. Then the signs
are $(\xi;\xi)$.

(5) Suppose $z+z'>0$ and $\beta<\gamma$. Then the signs are $(\xi;-\xi)$
if $\Lambda_{1}$ has more zeros then $\Lambda_{2}$, and $(-\xi;\xi)$
otherwise.
\end{prop}

\section{\label{App:Theta12}Theta $1$-lifts and $2$-lifts when $p+q=4$}

For the sake of completeness, in this appendix we list all nonzero
theta $1$-lifts and $2$-lifts for $O(p,q)$ when $(p,q)=(4,0)$,
$(3,1)$, $(2,2)$, which is read off from \cite[Th.15, Th.18]{Paul2005howe}.

\subsection{For $O(4,0)$}

Let $\pi=\pi_{1}((m,l;),\xi,\{e_{1}\pm e_{2}\},0,0,0,0)\in\mathcal{R}(O(4,0))$,
with integers $m>l\geqslant0$.  All nonzero theta $1$-lifts and
$2$-lifts of $\pi$ are in the following table:

\smallskip{}

\begin{center}
\begin{tabular}{c|c|c|c|c}
\hline 
 & $m$ & $l$ & $\xi$ & \tabularnewline
\hline 
\hline 
$\theta_{1}(\pi)\neq0$ & $\geqslant1$ & $0$ & \multirow{2}{*}{$1$} & $\theta_{1}(\pi)=\pi((m),\{2e_{1}\},0,0,0,0)$\tabularnewline
\cline{1-3} \cline{5-5} 
$\theta_{2}(\pi)\neq0$ & $>l$ & $\geqslant0$ &  & $\theta_{2}(\pi)=\pi((m,l),\{e_{1}\pm e_{2},2e_{1},2e_{2}\},0,0,0,0)$\tabularnewline
\hline 
\end{tabular}
\par\end{center}


\subsection{For $O(3,1)$}

Let $\pi=\pi_{\zeta}((m;),\xi,\O,0,0,\varepsilon,\kappa)\in\mathcal{R}(O(3,1))$
with $0\leqslant m\in\mathbb{Z}$, $\varepsilon=(\varepsilon_{1})$
and $\kappa=(\kappa_{1})$. All $\pi$ with $\theta_{1}(\pi)\neq0$
are in the following table:

\smallskip{}

\begin{center}
\begin{tabular}{c|c|c|c|c}
\hline 
$\zeta$ & $\xi$ & $m$ & $(\varepsilon_{1},\kappa_{1})$ & $\theta_{1}(\pi)$\tabularnewline
\hline 
\hline 
\multirow{2}{*}{$1$} & \multirow{2}{*}{$1$} & $\geqslant0$ & $(1,0)$ & $\pi((m),\{2e_{1}\},0,0,0,0)$\tabularnewline
\cline{3-5} 
 &  & $0$ & $\neq(1,0)$ & $\pi(0,\O,0,0,(-\varepsilon_{1}),(\kappa_{1}))$\tabularnewline
\hline 
\end{tabular}
\par\end{center}

\smallskip{}

All $\pi$ with $\theta_{2}(\pi)\neq0$ are in the following table: 

\smallskip{}

\begin{center}
\begin{tabular}{c|c|c|c|c}
\hline 
$\zeta$ & $\xi$ & $m$ & $(\varepsilon_{1},\kappa_{1})$ & $\theta_{2}(\pi)$\tabularnewline
\hline 
\hline 
\multirow{2}{*}{$1$} & \multirow{3}{*}{$1$} & \multirow{2}{*}{$\geqslant0$} & $\neq(-1,0)$ & $\pi((m),\{2e_{1}\},0,0,(-\varepsilon_{1}),(\kappa_{1}))$\tabularnewline
\cline{4-5} 
 &  &  & $(-1,0)$ & $\pi((m,0),\{e_{1}\pm e_{2},2e_{1},-2e_{2}\},0,0,0,0)$\tabularnewline
\cline{1-1} \cline{3-5} 
$-1$ &  & $\geqslant1$ & $(-1,0)$ & $\pi((m,0),\{e_{1}\pm e_{2},2e_{1},2e_{2}\},0,0,0,0)$\tabularnewline
\hline 
\end{tabular}
\par\end{center}


\subsection{For $O(2,2)$}

Let $\pi=\pi_{\zeta}(\lambda_{d},\xi,\Psi,\mu,\nu,\varepsilon,\kappa)\in\mathcal{R}(O(2,2))$.
 All $\pi$ with $\theta_{1}(\pi)\neq0$ are in the following table:

\smallskip{}\noindent\resizebox{\textwidth}{!}{%

\begin{tabular}{c|c|c|c|c|c|c|c|c|c}
\hline 
$\zeta$ & $\xi$ & $\lambda_{d}$ & $\Psi$ & $\mu$ & $\nu$ & $\varepsilon$ & $\kappa$ & $\theta_{1}(\pi)$ & with\tabularnewline
\hline 
\hline 
\multirow{3}{*}{$1$} & \multirow{3}{*}{$1$} & $(m;0)$ & $\{e_{1}\pm f_{1}\}$ & \multirow{3}{*}{$0$} & \multirow{3}{*}{$0$} & \multirow{2}{*}{$0$} & \multirow{2}{*}{$0$} & $\pi((m),\{2e_{1}\},0,0,0,0)$ & \multirow{2}{*}{$0\leqslant m\in\mathbb{Z}$}\tabularnewline
\cline{3-4} \cline{9-9} 
 &  & $(0;m)$ & $\{\pm e_{1}+f_{1}\}$ &  &  &  &  & $\pi((-m),\{-2e_{1}\},0,0,0,0)$ & \tabularnewline
\cline{3-4} \cline{7-10} 
 &  & $0$ & $\O$ &  &  & $(\varepsilon_{1},1)$ & $(\kappa_{1},0)$ & $\pi(0,\O,0,0,(\varepsilon_{1}),(\kappa_{1}))$ & $(\varepsilon_{1},\kappa_{1})\neq(-1,0)$\tabularnewline
\hline 
\end{tabular}

}\smallskip{}

All $\pi$ with $\theta_{2}(\pi)\neq0$ are in the following table: 

\smallskip{}\noindent\resizebox{\textwidth}{!}{%

\begin{tabular}{c|c|c|c|c|c|c|c|c|c}
\hline 
$\zeta$ & $\xi$ & $\lambda_{d}$ & $\Psi$ & $\mu$ & $\nu$ & $\varepsilon$ & $\kappa$ & $\theta_{2}(\pi)$ & with\tabularnewline
\hline 
\hline 
\multirow{10}{*}{$1$} & \multirow{12}{*}{$1$} & \multirow{4}{*}{$(m;l)$} & \multirow{2}{*}{$\{e_{1}$$\pm f_{1}\}$} & \multirow{4}{*}{$0$} & \multirow{4}{*}{$0$} & \multirow{6}{*}{$0$} & \multirow{6}{*}{$0$} & $\pi((m,-l),\{e_{1}\pm e_{2},$ & \multirow{2}{*}{integers $m\geqslant l\geqslant0$}\tabularnewline
 &  &  &  &  &  &  &  & $2e_{1},-2e_{2}\},0,0,0,0)$ & \tabularnewline
\cline{4-4} \cline{9-10} 
 &  &  & \multirow{2}{*}{$\{\pm e_{1}$$+f_{1}\}$} &  &  &  &  & $\pi((m,-l),\{\pm e_{1}-e_{2},$ & \multirow{2}{*}{integers $l\geqslant m\geqslant0$}\tabularnewline
 &  &  &  &  &  &  &  & $2e_{1},-2e_{2}\},0,0,0,0)$ & \tabularnewline
\cline{3-6} \cline{9-10} 
 &  & \multirow{8}{*}{$0$} & \multirow{8}{*}{$\O$} & \multirow{2}{*}{$(\mu_{1})$} & \multirow{2}{*}{$(\nu_{1})$} &  &  & $\pi(0,\O,(\mu_{1}),(\nu_{1}),$ & $0\leqslant\mu_{1}\in\mathbb{Z}$, $\nu_{1}\in\mathbb{C}$\tabularnewline
 &  &  &  &  &  &  &  & $0,0)$ & $\nu_{1}=0\Rightarrow$$\mu_{1}$ is odd\tabularnewline
\cline{5-10} 
 &  &  &  & \multirow{6}{*}{$0$} & \multirow{6}{*}{$0$} & \multirow{2}{*}{$(\varepsilon_{1},\varepsilon_{2})$} & \multirow{2}{*}{$(\kappa_{1},\kappa_{2})$} & $\pi(0,\O,0,0,$ & \multirow{2}{*}{all $(\varepsilon_{i},\kappa_{i})\neq(-1,0)$}\tabularnewline
 &  &  &  &  &  &  &  & $(\varepsilon_{1},\varepsilon_{2}),(\kappa_{1},\kappa_{2}))$ & \tabularnewline
\cline{7-10} 
 &  &  &  &  &  & \multirow{4}{*}{$(\varepsilon_{1},-1)$} & \multirow{4}{*}{$(\kappa_{1},0)$} & $\pi((0),\{-2e_{1}\},$ & \tabularnewline
 &  &  &  &  &  &  &  & $0,0,(\varepsilon_{1}),(\kappa_{1}))$ & \multirow{2}{*}{$(\varepsilon_{1},\kappa_{1})\neq(1,0)$}\tabularnewline
\cline{1-1} \cline{9-9} 
\multirow{2}{*}{$-1$} &  &  &  &  &  &  &  & $\pi((0),\{2e_{1}\},$ & \tabularnewline
 &  &  &  &  &  &  &  & $0,0,(\varepsilon_{1}),(\kappa_{1}))$ & \tabularnewline
\hline 
\end{tabular}

}\smallskip{}

\section{\label{App:list}Representations of $Sp(6,\mathbb{R})$ with infinitesimal
character $(\beta,0,1)$}

In this appendix, we list all $\pi\in\mathcal{R}(Sp(6,\mathbb{R}))$
with the infinitesimal character $(\beta,0,1)$ (for any $\beta\in\mathbb{C}$)
and calculate $\mathcal{A}(\pi)$ in every case. Let $\pi=\pi(\lambda_{d},\Psi,\mu,\nu,\varepsilon,\kappa)$,
with $MA\cong Sp(2v,\mathbb{R})\times GL(2,\mathbb{R})^{s}\times GL(1,\mathbb{R})^{t}$.
As $v+2s+t=3$, $(v,s,t)$ $=$ $(3,0,0)$, $(2,0,1)$, $(1,0,2)$,
$(0,0,3)$, $(0,1,1)$, or $(1,1,0)$. Let $\sim$ denote the equivalence
up to permutations and sign-changes of coordinates. Without loss of
generality, we assume that \emph{$\beta$, $\nu_{i}$ and $\kappa_{j}$
all lie in $\{z\in\mathbb{C}\mid\mathrm{Re}(z)\geqslant0\}$}.

\subsection{$\mathbf{(v,s,t)=(3,0,0)}$}

$\lambda_{d}\sim(0,0,1)$, $(1,0,1)$ or $(\beta,0,1)$ with $2\leqslant\beta\in\mathbb{Z}$.

\smallskip{}\noindent\resizebox{\textwidth}{!}{%

\begin{tabular}{c|c|c|c}
\hline 
\multirow{2}{*}{$\lambda_{d}=\lambda_{a}$} & $\rho(\mathfrak{u}\cap\mathfrak{p})$ & \multirow{2}{*}{$\Psi$} & \multirow{2}{*}{$\mathcal{A}(\pi)$}\tabularnewline
 & $-\rho(\mathfrak{u}\cap\mathfrak{k})$ &  & \tabularnewline
\hline 
\hline 
\multirow{2}{*}{$(1,0,0)$} & \multirow{2}{*}{$(1,1,1)$} & $\{e_{1}\pm e_{2},e_{2}\pm e_{3},e_{1}\pm e_{3},2e_{1},2e_{2},-2e_{3}\}$ & $\{(2,2,1)\}$\tabularnewline
\cline{3-4} 
 &  & $\{e_{1}\pm e_{2},\pm e_{2}-e_{3},e_{1}\pm e_{3},2e_{1},2e_{2},-2e_{3}\}$ & $\{(2,1,0)\}$\tabularnewline
\hline 
\multirow{2}{*}{$(0,0,-1)$} & \multirow{2}{*}{$(-1,-1,-1)$} & $\{e_{1}\pm e_{2},\pm e_{2}-e_{3},\pm e_{1}-e_{3},2e_{1},-2e_{2},-2e_{3}\}$ & $\{(0,-1,-2)\}$\tabularnewline
\cline{3-4} 
 &  & $\{\pm e_{1}-e_{2},\pm e_{2}-e_{3},\pm e_{1}-e_{3},2e_{1},-2e_{2},2e_{3}\}$ & $\{(-1,-2,-2)\}$\tabularnewline
\hline 
\multirow{4}{*}{$(1,0,-1)$} & \multirow{4}{*}{$(\frac{1}{2},0,-\frac{1}{2})$} & $\{e_{1}\pm e_{2},\pm e_{2}-e_{3},e_{1}\pm e_{3},2e_{1},2e_{2},-2e_{3}\}$ & $\{(2,1,-1)\}$\tabularnewline
\cline{3-4} 
 &  & $\{e_{1}\pm e_{2},\pm e_{2}-e_{3},e_{1}\pm e_{3},2e_{1},-2e_{2},-2e_{3}\}$ & $\{(2,-1,-1)\}$\tabularnewline
\cline{3-4} 
 &  & $\{e_{1}\pm e_{2},\pm e_{2}-e_{3},\pm e_{1}-e_{3},2e_{1},2e_{2},-2e_{3}\}$ & $\{(1,1,-2)\}$\tabularnewline
\cline{3-4} 
 &  & $\{e_{1}\pm e_{2},\pm e_{2}-e_{3},\pm e_{1}-e_{3},2e_{1},-2e_{2},-2e_{3}\}$ & $\{(1,-1,-2)\}$\tabularnewline
\hline 
\multirow{2}{*}{$(\beta,1,0)$} & \multirow{2}{*}{$(1,2,2)$} & $\{e_{1}\pm e_{2},e_{2}\pm e_{3},e_{1}\pm e_{3},2e_{1},2e_{2},2e_{3}\}$ & $\{(\beta+1,3,3)\}$\tabularnewline
\cline{3-4} 
 &  & $\{e_{1}\pm e_{2},e_{2}\pm e_{3},e_{1}\pm e_{3},2e_{1},2e_{2},-2e_{3}\}$ & $(\beta+1,3,1)\}$\tabularnewline
\hline 
\multirow{2}{*}{$(0,-1,-\beta)$} & \multirow{2}{*}{$(-2,-2,-1)$} & $\{\pm e_{1}-e_{2},\pm e_{2}-e_{3},\pm e_{1}-e_{3},2e_{1},-2e_{2},-2e_{3}\}$ & $\{(-1,-3,-\beta-1)\}$\tabularnewline
\cline{3-4} 
 &  & $\{\pm e_{1}-e_{2},\pm e_{2}-e_{3},\pm e_{1}-e_{3},-2e_{1},-2e_{2},-2e_{3}\}$ & $\{(-3,-3,-\beta-1)$\tabularnewline
\hline 
\multirow{2}{*}{$(\beta,0,-1)$} & \multirow{2}{*}{$(1,0,0)$} & $\{e_{1}\pm e_{2},\pm e_{2}-e_{3},e_{1}\pm e_{3},2e_{1},2e_{2},-2e_{3}\}$ & $\{(\beta+1,1,-1)\}$\tabularnewline
\cline{3-4} 
 &  & $\{e_{1}\pm e_{2},\pm e_{2}-e_{3},e_{1}\pm e_{3},2e_{1},-2e_{2},-2e_{3}\}$ & $\{(\beta+1,-1,-1)\}$\tabularnewline
\hline 
\multirow{2}{*}{$(1,0,-\beta)$} & \multirow{2}{*}{$(0,0,-1)$} & $\{e_{1}\pm e_{2},\pm e_{2}-e_{3},\pm e_{1}-e_{3},2e_{1},2e_{2},-2e_{3}\}$ & $\{(1,1,-\beta-1)\}$\tabularnewline
\cline{3-4} 
 &  & $\{e_{1}\pm e_{2},\pm e_{2}-e_{3},\pm e_{1}-e_{3},2e_{1},-2e_{2},-2e_{3}\}$ & $\{(1,-1,-\beta-1)\}$\tabularnewline
\hline 
\end{tabular}

}

\subsection{$\mathbf{(v,s,t)=(2,0,1)}$}

$\varepsilon=(\varepsilon_{1})$, $\kappa=(\kappa_{1})$, and $(\lambda_{d}\mid\kappa)\sim(\beta,0,1)$. 

(F-2): $(\varepsilon_{1},\kappa_{1})\ne(-1,0)$.

(1) $\kappa\sim(\beta)$, $\lambda_{d}\sim(1,0)$, and $(\varepsilon_{1},\beta)\neq(-1,0)$.

\smallskip{}\noindent\resizebox{\textwidth}{!}{%

\begin{tabular}{c|c|c|c|c|c|c|c}
\hline 
\multirow{2}{*}{$\kappa_{1}$} & \multirow{2}{*}{$\lambda_{d}$} & \multirow{2}{*}{$\lambda_{a}$} & $\rho(\mathfrak{u}\cap\mathfrak{p})$ & \multirow{2}{*}{$\Psi$} & \multirow{2}{*}{$\varepsilon_{1}$} & \multirow{2}{*}{$\mathcal{A}(\pi)$} & \multirow{2}{*}{with}\tabularnewline
 &  &  & $-\rho(\mathfrak{u}\cap\mathfrak{k})$ &  &  &  & \tabularnewline
\hline 
\hline 
\multirow{8}{*}{$\beta$} & \multirow{4}{*}{$(1,0)$} & \multirow{4}{*}{$(1,0,0)$} & \multirow{4}{*}{$(1,1,1)$} & \multirow{2}{*}{$\{e_{1}\pm e_{2},2e_{1},2e_{2}\}$} & $1$ & $\{(2,2,2)\}$ & \tabularnewline
\cline{6-8} 
 &  &  &  &  & $-1$ & $\{(2,2,1)\}$ & $\beta\in\mathbb{C}\backslash\{0\}$\tabularnewline
\cline{5-8} 
 &  &  &  & \multirow{2}{*}{$\{e_{1}\pm e_{2},2e_{1},-2e_{2}\}$} & $1$ & $\{(2,0,0)\}$ & \tabularnewline
\cline{6-8} 
 &  &  &  &  & $-1$ & $\{(2,1,0)\}$ & $\beta\in\mathbb{C}\backslash\{0\}$\tabularnewline
\cline{2-8} 
 & \multirow{4}{*}{$(0,-1)$} & \multirow{4}{*}{$(0,0,-1)$} & \multirow{4}{*}{$(-1,-1,-1)$} & \multirow{2}{*}{$\{\pm e_{1}-e_{2},2e_{1},-2e_{2}\}$} & $1$ & $\{(0,0,-2)\}$ & \tabularnewline
\cline{6-8} 
 &  &  &  &  & $-1$ & $\{(0,-1,-2)\}$ & $\beta\in\mathbb{C}\backslash\{0\}$\tabularnewline
\cline{5-8} 
 &  &  &  & \multirow{2}{*}{$\{\pm e_{1}-e_{2},-2e_{1},-2e_{2}\}$} & $1$ & $\{(-2,-2,-2)\}$ & \tabularnewline
\cline{6-8} 
 &  &  &  &  & $-1$ & $\{(-1,-2,-2)\}$ & $\beta\in\mathbb{C}\backslash\{0\}$\tabularnewline
\hline 
\end{tabular}

}\smallskip{}

(2) $\kappa=(0)$, $\varepsilon=(1)$ by (F-2), and $\lambda_{d}\sim(\beta,1)$
with $0\leqslant\beta\in\mathbb{Z}$. When $\beta=0$ this case coincides
with the case (1), so we assume $\beta\geqslant1$.

\smallskip{}\noindent\resizebox{\textwidth}{!}{%

\begin{tabular}{c|c|c|c|c|c|c|c}
\hline 
\multirow{2}{*}{$\kappa_{1}$} & \multirow{2}{*}{$\varepsilon_{1}$} & \multirow{2}{*}{$\beta\in\mathbb{Z}$} & \multirow{2}{*}{$\lambda_{d}$} & \multirow{2}{*}{$\lambda_{a}$} & $\rho(\mathfrak{u}\cap\mathfrak{p})$ & \multirow{2}{*}{$\Psi$} & \multirow{2}{*}{$\mathcal{A}(\pi)$}\tabularnewline
 &  &  &  &  & $-\rho(\mathfrak{u}\cap\mathfrak{k})$ &  & \tabularnewline
\hline 
\hline 
\multirow{6}{*}{$0$} & \multirow{6}{*}{$1$} & \multirow{2}{*}{$1$} & \multirow{2}{*}{$(1,-1)$} & \multirow{2}{*}{$(1,0,-1)$} & \multirow{2}{*}{$(\frac{1}{2},0,-\frac{1}{2})$} & $\{e_{1}\pm e_{2},2e_{1},2e_{2}\}$ & $\{(2,0,-1)\}$\tabularnewline
\cline{7-8} 
 &  &  &  &  &  & $\{\pm e_{1}-e_{2},2e_{1},2e_{2}\}$ & $\{(1,0,-2)\}$\tabularnewline
\cline{3-8} 
 &  & \multirow{4}{*}{$\geqslant2$} & $(\beta,1)$ & $(\beta,1,0)$ & $(1,2,2)$ & $\{e_{1}\pm e_{2},2e_{1},2e_{2}\}$ & $\{(\beta+1,3,2)\}$\tabularnewline
\cline{4-8} 
 &  &  & $(-1,-\beta)$ & $(0,-1,-\beta)$ & $(-2,-2,-1)$ & $\{\pm e_{1}-e_{2},-2e_{1},-2e_{2}\}$ & $\{(-2,-3,-\beta-1)\}$\tabularnewline
\cline{4-8} 
 &  &  & $(\beta,-1)$ & $(\beta,0,-1)$ & $(1,0,0)$ & $\{e_{1}\pm e_{2},2e_{1},-2e_{2}\}$ & $\{(\beta+1,0,-1)\}$\tabularnewline
\cline{4-8} 
 &  &  & $(1,-\beta)$ & $(1,0,-\beta)$ & $(0,0,-1)$ & $\{\pm e_{1}-e_{2},2e_{1},-2e_{2}\}$ & $\{(1,0,-\beta-1)\}$\tabularnewline
\hline 
\end{tabular}

}\smallskip{}

(3) $\kappa\sim(1)$, $\lambda_{d}\sim(\beta,0)$ with $0\leqslant\beta\in\mathbb{Z}$.
When $\beta=1$ this case coincides with (1), so we assume $\beta\neq1$.

\smallskip{}\noindent\resizebox{\textwidth}{!}{%

\begin{tabular}{c|c|c|c|c|c|c|c}
\hline 
\multirow{2}{*}{$\kappa_{1}$} & \multirow{2}{*}{$\beta\in\mathbb{Z}$} & \multirow{2}{*}{$\lambda_{d}$} & \multirow{2}{*}{$\lambda_{a}$} & $\rho(\mathfrak{u}\cap\mathfrak{p})$ & \multirow{2}{*}{$\Psi$ } & \multirow{2}{*}{$\varepsilon_{1}$} & \multirow{2}{*}{$\mathcal{A}(\pi)$}\tabularnewline
 &  &  &  & $-\rho(\mathfrak{u}\cap\mathfrak{k})$ &  &  & \tabularnewline
\hline 
\hline 
\multirow{12}{*}{$1$} & \multirow{4}{*}{$0$} & \multirow{4}{*}{$(0,0)$} & \multirow{4}{*}{$(0,0,0)$} & \multirow{4}{*}{$(0,0,0)$} & \multirow{2}{*}{$\{e_{1}\pm e_{2},2e_{1},-2e_{2}\}$} & $1$ & $\{(1,0,0)\}$\tabularnewline
\cline{7-8} 
 &  &  &  &  &  & $-1$ & $\{(1,1,0)\}$\tabularnewline
\cline{6-8} 
 &  &  &  &  & \multirow{2}{*}{$\{\pm e_{1}-e_{2},2e_{1},-2e_{2}\}$} & $1$ & $\{(0,0,-1)\}$\tabularnewline
\cline{7-8} 
 &  &  &  &  &  & $-1$ & $\{(0,-1,-1)\}$\tabularnewline
\cline{2-8} 
 & \multirow{8}{*}{$\geqslant2$} & \multirow{4}{*}{$(\beta,0)$} & \multirow{4}{*}{$(\beta,0,0)$} & \multirow{4}{*}{$(1,1,1)$} & \multirow{2}{*}{$\{e_{1}\pm e_{2},2e_{1},2e_{2}\}$} & $1$ & $\{(\beta+1,2,2)\}$\tabularnewline
\cline{7-8} 
 &  &  &  &  &  & $-1$ & $\{(\beta+1,2,1)\}$\tabularnewline
\cline{6-8} 
 &  &  &  &  & \multirow{2}{*}{$\{e_{1}\pm e_{2},2e_{1},-2e_{2}\}$} & $1$ & $\{(\beta+1,0,0)\}$\tabularnewline
\cline{7-8} 
 &  &  &  &  &  & $-1$ & $\{(\beta+1,1,0)\}$\tabularnewline
\cline{3-8} 
 &  & \multirow{4}{*}{$(0,-\beta)$} & \multirow{4}{*}{$(0,0,-\beta)$} & \multirow{4}{*}{$(-1,-1,-1)$} & \multirow{2}{*}{$\{\pm e_{1}-e_{2},2e_{1},-2e_{2}\}$} & $1$ & $\{(0,0,-\beta-1)\}$\tabularnewline
\cline{7-8} 
 &  &  &  &  &  & $-1$ & $\{(0,-1,-\beta-1)\}$\tabularnewline
\cline{6-8} 
 &  &  &  &  & \multirow{2}{*}{$\{\pm e_{1}-e_{2},-2e_{1},-2e_{2}\}$} & $1$ & $\{(-2,-2,-\beta-1)\}$\tabularnewline
\cline{7-8} 
 &  &  &  &  &  & $-1$ & $\{(-1,-2,-\beta-1)\}$\tabularnewline
\hline 
\end{tabular}

}

\subsection{$\mathbf{(v,s,t)=(1,0,2)}$}

$\varepsilon=(\varepsilon_{1},\varepsilon_{2})$, $\kappa=(\kappa_{1},\kappa_{2})$,
and $(\lambda_{d}\mid(\kappa_{1},\kappa_{2}))\sim(\beta,0,1)$.

(F-2): $\kappa_{i}=0$ $\Rightarrow$ $\varepsilon_{i}=-1$; $\kappa_{1}=\pm\kappa_{2}$
$\Rightarrow$ $\varepsilon_{1}=\varepsilon_{2}$.

(1) $\lambda_{d}=(0)$. Up to permutations and sign-changes $(\kappa_{1},\kappa_{2})=(\beta,1)$.

\smallskip{}\noindent\resizebox{\textwidth}{!}{%

\begin{tabular}{c|c|c|c|c|c|c}
\hline 
\multirow{2}{*}{$(\kappa_{1},\kappa_{2})$} & \multirow{2}{*}{$\lambda_{d}$} & $\lambda_{a}+\rho(\mathfrak{u}\cap\mathfrak{p})$ & \multirow{2}{*}{$\Psi$} & \multirow{2}{*}{$(\varepsilon_{1},\varepsilon_{2})$} & \multirow{2}{*}{$\mathcal{A}(\pi)$} & \multirow{2}{*}{with}\tabularnewline
 &  & $-\rho(\mathfrak{u}\cap\mathfrak{k})$ &  &  &  & \tabularnewline
\hline 
\hline 
\multirow{6}{*}{$(\beta,1)$} & \multirow{6}{*}{$(0)$} & \multirow{6}{*}{$(0,0,0)$} & \multirow{3}{*}{$\{2e_{1}\}$} & $(1,1)$ & $\{(1,0,0)\}$ & \multirow{2}{*}{}\tabularnewline
\cline{5-6} 
 &  &  &  & $(-1,-1)$ & $\{(1,1,1)\}$ & \tabularnewline
\cline{5-7} 
 &  &  &  & $(1,-1)$ or $(-1,1)$  & $\{(1,1,0)\}$ & $\beta\in\mathbb{C}\backslash\{\pm1\}$\tabularnewline
\cline{4-7} 
 &  &  & \multirow{3}{*}{$\{-2e_{1}\}$} & $(1,1)$ & $\{(0,0,-1)\}$ & \multirow{2}{*}{}\tabularnewline
\cline{5-6} 
 &  &  &  & $(-1,-1)$ & $\{(-1,-1,-1)\}$ & \tabularnewline
\cline{5-7} 
 &  &  &  & $(1,-1)$ or $(-1,1)$  & $\{(0,-1,-1)\}$ & $\beta\in\mathbb{C}\backslash\{\pm1\}$\tabularnewline
\hline 
\end{tabular}

}\smallskip{}

(2) $\lambda_{d}\sim(1)$. Up to permutations and sign-changes $(\kappa_{1},\kappa_{2})=(\beta,0)$.
By (F-2), $\varepsilon_{2}=-1$.

\smallskip{}\noindent\resizebox{\textwidth}{!}{%

\begin{tabular}{c|c|c|c|c|c|c|c|c}
\hline 
\multirow{2}{*}{$(\kappa_{1},\kappa_{2})$} & \multirow{2}{*}{$\varepsilon_{2}$} & \multirow{2}{*}{$\lambda_{d}$} & \multirow{2}{*}{$\Psi$} & \multirow{2}{*}{$\lambda_{a}$} & $\rho(\mathfrak{u}\cap\mathfrak{p})$ & \multirow{2}{*}{$\varepsilon_{1}$} & \multirow{2}{*}{$\mathcal{A}(\pi)$} & \multirow{2}{*}{with}\tabularnewline
 &  &  &  &  & $-\rho(\mathfrak{u}\cap\mathfrak{k})$ &  &  & \tabularnewline
\hline 
\hline 
\multirow{4}{*}{$(\beta,0)$} & \multirow{4}{*}{$-1$} & \multirow{2}{*}{$(1)$} & \multirow{2}{*}{$\{2e_{1}\}$} & \multirow{2}{*}{$(1,0,0)$} & \multirow{2}{*}{$(1,1,1)$} & $1$  & $\{(2,2,1),(2,1,0)\}$ & $\beta\in\mathbb{C}\backslash\{0\}$\tabularnewline
\cline{7-9} 
 &  &  &  &  &  & $-1$ & $\{(2,1,1)\}$ & \tabularnewline
\cline{3-9} 
 &  & \multirow{2}{*}{$(-1)$} & \multirow{2}{*}{$\{-2e_{1}\}$} & \multirow{2}{*}{$(0,0,-1)$} & \multirow{2}{*}{$(-1,-1,-1)$} & $1$  & $\{(0,-1,-2),(-1,-2,-2)\}$ & $\beta\in\mathbb{C}\backslash\{0\}$\tabularnewline
\cline{7-9} 
 &  &  &  &  &  & $-1$ & $\{(-1,-1,-2)\}$ & \tabularnewline
\hline 
\end{tabular}

}\smallskip{}

(3) $\lambda_{d}\sim(\beta)$ with $0\leqslant\beta\in\mathbb{Z}$.
Up to permutations and sign-changes $(\kappa_{1},\kappa_{2})=(1,0)$.
By (F-2), $\varepsilon_{2}=-1$. When $\beta=0$ or $1$, this case
coincides with the cases (1) or (2), so we assume $\beta\geqslant2$.

\smallskip{}\noindent\resizebox{\textwidth}{!}{%

\begin{tabular}{c|c|c|c|c|c|c|c|c}
\hline 
\multirow{2}{*}{$(\kappa_{1},\kappa_{2})$} & \multirow{2}{*}{$\varepsilon_{2}$} & \multirow{2}{*}{$\beta\in\mathbb{Z}$} & \multirow{2}{*}{$\lambda_{d}$} & \multirow{2}{*}{$\Psi$} & \multirow{2}{*}{$\lambda_{a}$} & $\rho(\mathfrak{u}\cap\mathfrak{p})$ & \multirow{2}{*}{$\varepsilon_{1}$} & \multirow{2}{*}{$\mathcal{A}(\pi)$}\tabularnewline
 &  &  &  &  &  & $-\rho(\mathfrak{u}\cap\mathfrak{k})$ &  & \tabularnewline
\hline 
\hline 
\multirow{4}{*}{$(1,0)$} & \multirow{4}{*}{$-1$} & \multirow{4}{*}{$\geqslant2$} & \multirow{2}{*}{$(\beta)$} & \multirow{2}{*}{$\{2e_{1}\}$} & \multirow{2}{*}{$(\beta,0,0)$} & \multirow{2}{*}{$(1,1,1)$} & $1$ & $\{(\beta+1,2,1),(\beta+1,1,0)\}$\tabularnewline
\cline{8-9} 
 &  &  &  &  &  &  & $-1$ & $\{(\beta+1,1,1)\}$\tabularnewline
\cline{4-9} 
 &  &  & \multirow{2}{*}{$(-\beta)$} & \multirow{2}{*}{$\{-2e_{1}\}$} & \multirow{2}{*}{$(0,0,-\beta)$} & \multirow{2}{*}{$(-1,-1,-1)$} & $1$ & $\{(0,-1,-\beta-1),(-1,-2,-\beta-1)\}$\tabularnewline
\cline{8-9} 
 &  &  &  &  &  &  & $-1$ & $\{(-1,-1,-\beta-1)\}$\tabularnewline
\hline 
\end{tabular}

}

\subsection{$\mathbf{(v,s,t)=(0,0,3)}$}

$\varepsilon=(\varepsilon_{1},\varepsilon_{2},\varepsilon_{3})$,
and $\kappa=(\kappa_{1},\kappa_{2},\kappa_{3})\sim(\beta,0,1)$. 

(F-2): $\kappa_{i}=0$ $\Rightarrow$ $\varepsilon_{i}=1$; $\kappa_{i}=\pm\kappa_{j}$
$\Rightarrow$ $\varepsilon_{i}=\varepsilon_{j}$.

Up to permutations and sign-changes $(\kappa_{1},\kappa_{2},\kappa_{3})=(\beta,1,0)$.
By (F-2), $\varepsilon_{3}=1$.

\smallskip{}

\begin{center}
\begin{tabular}{c|c|c|c|c|c}
\hline 
\multirow{2}{*}{$(\kappa_{1},\kappa_{2},\kappa_{3})$} & \multirow{2}{*}{$\varepsilon_{3}$} & $\lambda_{a}+\rho(\mathfrak{u}\cap\mathfrak{p})$ & \multirow{2}{*}{$(\varepsilon_{1},\varepsilon_{2})$} & \multirow{2}{*}{$\mathcal{A}(\pi)$} & \multirow{2}{*}{with}\tabularnewline
 &  & $-\rho(\mathfrak{u}\cap\mathfrak{k})$ &  &  & \tabularnewline
\hline 
\hline 
\multirow{4}{*}{$(\beta,1,0)$} & \multirow{4}{*}{$1$} & \multirow{4}{*}{$(0,0,0)$} & $(1,1)$ & $\{(0,0,0)\}$ & \tabularnewline
\cline{4-6} 
 &  &  & $(-1,-1)$  & $\{(1,1,0),(0,-1,-1)\}$ & $\beta\in\mathbb{C}\backslash\{0\}$\tabularnewline
\cline{4-6} 
 &  &  & $(1,-1)$ & \multirow{2}{*}{$\{(1,0,0),(0,0,-1)\}$} &  $\beta\in\mathbb{C}\backslash\{\pm1\}$ \tabularnewline
\cline{4-4} \cline{6-6} 
 &  &  & $(-1,1)$  &  & $\beta\in\mathbb{C}\backslash\{0,\pm1\}$\tabularnewline
\hline 
\end{tabular}
\par\end{center}


\subsection{$\mathbf{(v,s,t)=(0,1,1)}$}

$\mu=(\mu_{1})$, $\nu=(\nu_{1})$, $\varepsilon=(\varepsilon_{1})$,
$\kappa=(\kappa_{1})$, and
\[
(\frac{\mu_{1}+\nu_{1}}{2},\frac{-\mu_{1}+\nu_{1}}{2},\kappa_{1})\sim(\beta,0,1)
\]

(F-2): $\nu_{1}=0$ $\Rightarrow$ $\mu_{1}$ is odd; $\kappa_{1}=0$
$\Rightarrow$ $\varepsilon_{1}=1$.

(1) $\kappa_{1}=\beta$, and $(\frac{\mu_{1}+\nu_{1}}{2},\frac{-\mu_{1}+\nu_{1}}{2})\sim(0,1)$.
Then $\mu_{1}=1$, $\nu_{1}=1$.

\smallskip{}

\begin{center}
\begin{tabular}{c|c|c|c|c|c|c|c}
\hline 
\multirow{2}{*}{$\mu_{1}$} & \multirow{2}{*}{$\nu_{1}$} & \multirow{2}{*}{$\lambda_{a}$} & $\rho(\mathfrak{u}\cap\mathfrak{p})$ & \multirow{2}{*}{$\kappa_{1}$} & \multirow{2}{*}{$\varepsilon_{1}$} & \multirow{2}{*}{$\mathcal{A}(\pi)$} & \multirow{2}{*}{with}\tabularnewline
 &  &  & $-\rho(\mathfrak{u}\cap\mathfrak{k})$ &  &  &  & \tabularnewline
\hline 
\hline 
\multirow{2}{*}{$1$} & \multirow{2}{*}{$1$} & \multirow{2}{*}{$(\frac{1}{2},0,-\frac{1}{2})$} & \multirow{2}{*}{$(\frac{1}{2},0,-\frac{1}{2})$} & \multirow{2}{*}{$\beta$} & $1$ & $\{(1,0,-1)\}$ & $0\leqslant\beta\in\mathbb{Z}$\tabularnewline
\cline{6-8} 
 &  &  &  &  & $-1$ & $\{(1,1,-1),(1,-1,-1)\}$ & $1\leqslant\beta\in\mathbb{Z}$\tabularnewline
\hline 
\end{tabular}
\par\end{center}

\smallskip{}

(2) $\kappa_{1}=1$, and $(\frac{\mu_{1}+\nu_{1}}{2},\frac{-\mu_{1}+\nu_{1}}{2})\sim(\beta,0)$.
Then $0\leqslant\mu_{1}=\beta=\nu_{1}$ with $0\leqslant\beta\in\mathbb{Z}$,
and $\beta\neq0$ by (F-2). When $\beta=1$ this case coincides with
the case (1), so we assume $\beta\geqslant2$.

\smallskip{}\noindent\resizebox{\textwidth}{!}{%

\begin{tabular}{c|c|c|c|c|c|c}
\hline 
\multirow{2}{*}{$\kappa_{1}$} & $\mu_{1}=\nu_{1}$  & \multirow{2}{*}{$\lambda_{a}$} & $\rho(\mathfrak{u}\cap\mathfrak{p})$ & \multirow{2}{*}{$\varepsilon_{1}$} & \multirow{2}{*}{$\mathcal{A}(\pi)$} & \multirow{2}{*}{with}\tabularnewline
 & $\ \ \ =\beta$ &  & $-\rho(\mathfrak{u}\cap\mathfrak{k})$ &  &  & \tabularnewline
\hline 
\hline 
\multirow{6}{*}{$1$} & \multirow{3}{*}{$2m+1$} & \multirow{3}{*}{$(m+\frac{1}{2},0,-m-\frac{1}{2})$} & \multirow{3}{*}{$(\frac{1}{2},0,-\frac{1}{2})$} & $1$ & $\{(m+1,0,-m-1)\}$ & \multirow{6}{*}{$1\leqslant m\in\mathbb{Z}$}\tabularnewline
\cline{5-6} 
 &  &  &  & \multirow{2}{*}{$-1$} & $\{(m+1,1,-m-1),$ & \tabularnewline
 &  &  &  &  & $\ \ (m+1,-1,-m-1)\}$ & \tabularnewline
\cline{2-6} 
 & \multirow{3}{*}{$2m$} & \multirow{3}{*}{$(m,0,-m)$} & \multirow{3}{*}{$(\frac{1}{2},0,-\frac{1}{2})$} & $1$ & $\{(m+1,0,-m),(m,0,-m-1)\}$ & \tabularnewline
\cline{5-6} 
 &  &  &  & \multirow{2}{*}{$-1$} & $\{(m+1,1,-m),(m+1,-1,-m),$ & \tabularnewline
 &  &  &  &  & $\ \ (m,1,-m-1),(m,-1,-m-1)\}$ & \tabularnewline
\hline 
\end{tabular}

}\smallskip{}

(3) $\kappa_{1}=0$, and $(\frac{\mu_{1}+\nu_{1}}{2},\frac{-\mu_{1}+\nu_{1}}{2})\sim(\beta,1)$.
By (F-2), $\varepsilon_{1}=1$. As $\mu_{1}=|\beta\pm1|\in\mathbb{Z}$,
$0\leqslant\beta\in\mathbb{Z}$. When $\beta=0$ this case coincides
with the case (1), so we assume $\beta\geqslant1$. Notice that in
this case $\mu_{1}$ and $\nu_{1}$ have the same parity, so $\nu_{1}\neq0$
by (F-2).

\smallskip{}\noindent\resizebox{\textwidth}{!}{%

\begin{tabular}{c|c|c|c|c|c|c|c|c}
\hline 
\multirow{2}{*}{$\kappa_{1}$} & \multirow{2}{*}{$\varepsilon_{1}$} & \multirow{2}{*}{$\mu_{1}$} & \multirow{2}{*}{$\nu_{1}$} & \multirow{2}{*}{$\lambda_{a}$} & \multirow{2}{*}{$\beta$} & $\rho(\mathfrak{u}\cap\mathfrak{p})$ & \multirow{2}{*}{$\mathcal{A}(\pi)$} & \multirow{2}{*}{with}\tabularnewline
 &  &  &  &  &  & $-\rho(\mathfrak{u}\cap\mathfrak{k})$ &  & \tabularnewline
\hline 
\hline 
\multirow{7}{*}{$0$} & \multirow{7}{*}{$1$} & \multirow{4}{*}{$\beta-1$} & \multirow{4}{*}{$\beta+1$} & \multirow{4}{*}{$(\frac{\beta-1}{2},0,-\frac{\beta-1}{2})$} & $1$ & $(0,0,0)$ & $\{(1,0,0),(0,0,-1)\}$ & \tabularnewline
\cline{6-9} 
 &  &  &  &  & $2m$ & \multirow{6}{*}{$(\frac{1}{2},0,-\frac{1}{2})$} & $\{(m,0,-m)\}$ & \multirow{6}{*}{$1\leqslant m\in\mathbb{Z}$}\tabularnewline
\cline{6-6} \cline{8-8} 
 &  &  &  &  & \multirow{2}{*}{$2m+1$} &  & $\{(m+1,0,-m),$ & \tabularnewline
 &  &  &  &  &  &  & $\ \ (m,0,-m-1)\}$ & \tabularnewline
\cline{3-6} \cline{8-8} 
 &  & \multirow{3}{*}{$\beta+1$} & \multirow{3}{*}{$\beta-1$} & \multirow{3}{*}{$(\frac{\beta+1}{2},0,-\frac{\beta+1}{2})$} & $2m$ &  & $\{(m+1,0,-m-1)\}$ & \tabularnewline
\cline{6-6} \cline{8-8} 
 &  &  &  &  & \multirow{2}{*}{$2m+1$} &  & $\{(m+2,0,-m-1),$ & \tabularnewline
 &  &  &  &  &  &  & $\ \ (m+1,0,-m-2)\}$ & \tabularnewline
\hline 
\end{tabular}

}

\subsection{$\mathbf{(v,s,t)=(1,1,0)}$}

$\mu=(\mu_{1})$, $\nu=(\nu_{1})$, and $(\lambda_{d}\mid(\frac{\mu_{1}+\nu_{1}}{2},\frac{-\mu_{1}+\nu_{1}}{2}))\sim(\beta,0,1)$. 

(F-2): $\nu_{1}=0$ $\Rightarrow$ $\mu_{1}$ is odd.

(1) $\lambda_{d}=(0)$, and $(\frac{\mu_{1}+\nu_{1}}{2},\frac{-\mu_{1}+\nu_{1}}{2})\sim(\beta,1)$.
Then $\mu_{1}=|\beta\pm1|\in\mathbb{Z}$, and $0\leqslant\beta\in\mathbb{Z}$.
In this case $\mu_{1}$ and $\nu_{1}$ are integers with the same
parity, so $\nu_{1}\neq0$ by (F-2).

\smallskip{}\noindent\resizebox{\textwidth}{!}{%

\begin{tabular}{c|c|c|c|c|c|c|c|c}
\hline 
\multirow{2}{*}{$\lambda_{d}$} & \multirow{2}{*}{$\beta$ } & \multirow{2}{*}{$\mu_{1}$} & \multirow{2}{*}{$\nu_{1}$} & \multirow{2}{*}{$\lambda_{a}$} & $\rho(\mathfrak{u}\cap\mathfrak{p})$ & \multirow{2}{*}{$\Psi$} & \multirow{2}{*}{$\mathcal{A}(\pi)$} & \multirow{2}{*}{with}\tabularnewline
 &  &  &  &  & $-\rho(\mathfrak{u}\cap\mathfrak{k})$ &  &  & \tabularnewline
\hline 
\hline 
\multirow{16}{*}{$(0)$} & \multirow{2}{*}{$0$} & \multirow{2}{*}{$1$} & \multirow{2}{*}{$1$} & \multirow{2}{*}{$(\frac{1}{2},0,-\frac{1}{2})$} & \multirow{2}{*}{$(\frac{1}{2},0,-\frac{1}{2})$} & $\{2e_{1}\}$ & $\{(1,1,-1)\}$ & \multirow{4}{*}{}\tabularnewline
\cline{7-8} 
 &  &  &  &  &  & $\{-2e_{1}\}$ & $\{(1,-1,-1)\}$ & \tabularnewline
\cline{2-8} 
 & \multirow{2}{*}{$1$} & \multirow{2}{*}{$0$} & \multirow{2}{*}{$2$} & \multirow{2}{*}{$(0,0,0)$} & \multirow{2}{*}{$(0,0,0)$} & $\{2e_{1}\}$ & $\{(1,1,0)\}$ & \tabularnewline
\cline{7-8} 
 &  &  &  &  &  & $\{-2e_{1}\}$ & $\{(0,-1,-1)\}$ & \tabularnewline
\cline{2-9} 
 & \multirow{2}{*}{$2m$} & \multirow{6}{*}{$\beta-1$} & \multirow{6}{*}{$\beta+1$} &  & \multirow{12}{*}{$(\frac{1}{2},0,-\frac{1}{2})$} & $\{2e_{1}\}$ & $\{(m,1,-m)\}$ & \multirow{12}{*}{$1\leqslant m\in\mathbb{Z}$}\tabularnewline
\cline{7-8} 
 &  &  &  & \multirow{2}{*}{$(\frac{\beta-1}{2},0,$} &  & $\{-2e_{1}\}$ & $\{(m,-1,-m)\}$ & \tabularnewline
\cline{2-2} \cline{7-8} 
 & \multirow{4}{*}{$2m+1$} &  &  &  &  & \multirow{2}{*}{$\{2e_{1}\}$} & $\{(m+1,1,-m),$ & \tabularnewline
 &  &  &  & \multirow{2}{*}{$-\frac{\beta-1}{2})$} &  &  & $\ \ (m,1,-m-1)\}$ & \tabularnewline
\cline{7-8} 
 &  &  &  &  &  & \multirow{2}{*}{$\{-2e_{1}\}$} & $\{(m+1,-1,-m),$ & \tabularnewline
 &  &  &  &  &  &  & $\ \ (m,-1,-m-1)\}$ & \tabularnewline
\cline{2-5} \cline{7-8} 
 & \multirow{2}{*}{$2m$} & \multirow{6}{*}{$\beta+1$} & \multirow{6}{*}{$\beta-1$} &  &  & $\{2e_{1}\}$ & $\{(m+1,1,-m-1)\}$ & \tabularnewline
\cline{7-8} 
 &  &  &  & \multirow{2}{*}{$(\frac{\beta+1}{2},0,$} &  & $\{-2e_{1}\}$ & $\{(m+1,-1,-m-1)\}$ & \tabularnewline
\cline{2-2} \cline{7-8} 
 & \multirow{4}{*}{$2m+1$} &  &  &  &  & \multirow{2}{*}{$\{2e_{1}\}$} & $\{(m+2,1,-m-1),$ & \tabularnewline
 &  &  &  & \multirow{2}{*}{$-\frac{\beta+1}{2})$} &  &  & $\ \ (m+1,1,-m-2)\}$ & \tabularnewline
\cline{7-8} 
 &  &  &  &  &  & \multirow{2}{*}{$\{-2e_{1}\}$} & $\{(m+2,-1,-m-1),$ & \tabularnewline
 &  &  &  &  &  &  & $\ \ (m+1,-1,-m-2)\}$ & \tabularnewline
\hline 
\end{tabular}

}\smallskip{} 

(2) $\lambda_{d}\sim(1)$, and $(\frac{\mu_{1}+\nu_{1}}{2},\frac{-\mu_{1}+\nu_{1}}{2})\sim(\beta,0)$.
Then $0\leqslant\beta=\mu_{1}=\nu_{1}\in\mathbb{Z}$, and $\beta\neq0$
by (F-2).

\smallskip{}\noindent\resizebox{\textwidth}{!}{%

\begin{tabular}{c|c|c|c|c|c|c}
\hline 
$\mu_{1}=\nu_{1}$  & \multirow{2}{*}{$\lambda_{d}$} & \multirow{2}{*}{$\Psi$} & \multirow{2}{*}{$\lambda_{a}$} & $\rho(\mathfrak{u}\cap\mathfrak{p})$ & \multirow{2}{*}{$\mathcal{A}(\pi)$} & \multirow{2}{*}{with}\tabularnewline
$\ \ \ =\beta$ &  &  &  & $-\rho(\mathfrak{u}\cap\mathfrak{k})$ &  & \tabularnewline
\hline 
\hline 
\multirow{2}{*}{$1$} & $(1)$ & $\{2e_{1}\}$ & $(1,\frac{1}{2},-\frac{1}{2})$ & $(1,\frac{3}{2},\frac{1}{2})$ & $\{(2,2,0)\}$ & \multirow{4}{*}{}\tabularnewline
\cline{2-6} 
 & $(-1)$ & $\{-2e_{1}\}$ & $(\frac{1}{2},-\frac{1}{2},-1)$ & $(-\frac{1}{2},-\frac{3}{2},-1)$ & $\{(0,-2,-2)\}$ & \tabularnewline
\cline{1-6} 
\multirow{2}{*}{$2$} & $(1)$ & $\{2e_{1}\}$ & $(1,1,-1)$ & $(1,1,0)$ & $\{(2,2,-1)\}$ & \tabularnewline
\cline{2-6} 
 & $(-1)$ & $\{-2e_{1}\}$ & $(1,-1,-1)$ & $(0,-1,-1)$ & $\{(1,-2,-2)\}$ & \tabularnewline
\hline 
\multirow{2}{*}{$2m+1$} & $(1)$ & $\{2e_{1}\}$ & $(m+\frac{1}{2},1,-m-\frac{1}{2})$ & $(\frac{1}{2},1,-\frac{1}{2})$ & $\{(m+1,2,-m-1)\}$ & \multirow{2}{*}{$1\leqslant m\in\mathbb{Z}$}\tabularnewline
\cline{2-6} 
 & $(-1)$ & $\{-2e_{1}\}$ & $(m+\frac{1}{2},-1,-m-\frac{1}{2})$ & $(\frac{1}{2},-1,-\frac{1}{2})$ & $\{(m+1,-2,-m-1)\}$ & \tabularnewline
\hline 
\multirow{4}{*}{$2m$} & \multirow{2}{*}{$(1)$} & \multirow{2}{*}{$\{2e_{1}\}$} & \multirow{2}{*}{$(m,1,-m)$} & \multirow{2}{*}{$(\frac{1}{2},1,-\frac{1}{2})$} & $\{(m+1,2,-m),$ & \multirow{4}{*}{$2\leqslant m\in\mathbb{Z}$}\tabularnewline
 &  &  &  &  & $\ \ (m,2,-m-1)\}$ & \tabularnewline
\cline{2-6} 
 & \multirow{2}{*}{$(-1)$} & \multirow{2}{*}{$\{-2e_{1}\}$} & \multirow{2}{*}{$(m,-1,-m)$} & \multirow{2}{*}{$(\frac{1}{2},-1,-\frac{1}{2})$} & $\{(m+1,-2,-m),$ & \tabularnewline
 &  &  &  &  & $\ \ (m,-2,-m-1)\}$ & \tabularnewline
\hline 
\end{tabular}

}\smallskip{}

(3) $\lambda_{d}\sim(\beta)$ with $0\leqslant\beta\in\mathbb{Z}$,
and $(\frac{\mu_{1}+\nu_{1}}{2},\frac{-\mu_{1}+\nu_{1}}{2})\sim(0,1)$.
So $\mu_{1}=\nu_{1}=1$. When $\beta=0$ or $1$, this case coincides
with the cases (1) or (2), so we assume $\beta\geqslant2$.

\smallskip{}\noindent\resizebox{\textwidth}{!}{%

\begin{tabular}{c|c|c|c|c|c|c|c}
\hline 
\multirow{2}{*}{$\lambda_{d}$} & \multirow{2}{*}{$\Psi$} & \multirow{2}{*}{$\mu_{1}$} & \multirow{2}{*}{$\nu_{1}$} & \multirow{2}{*}{$\lambda_{a}$} & $\rho(\mathfrak{u}\cap\mathfrak{p})$ & \multirow{2}{*}{$\mathcal{A}(\pi)$} & \multirow{2}{*}{with}\tabularnewline
 &  &  &  &  & $-\rho(\mathfrak{u}\cap\mathfrak{k})$ &  & \tabularnewline
\hline 
\hline 
$(\beta)$ & $\{2e_{1}\}$ & \multirow{2}{*}{$1$} & \multirow{2}{*}{$1$} & $(\beta,\frac{1}{2},-\frac{1}{2})$ & $(1,\frac{3}{2},\frac{1}{2})$ & $\{(\beta+1,2,0)\}$ & \multirow{2}{*}{$2\leqslant\beta\in\mathbb{Z}$}\tabularnewline
\cline{1-2} \cline{5-7} 
$(-\beta)$ & $\{-2e_{1}\}$ &  &  & $(\frac{1}{2},-\frac{1}{2},-\beta)$ & $(-\frac{1}{2},-\frac{3}{2},-1)$ & $\{(0,-2,-\beta-1)\}$ & \tabularnewline
\hline 
\end{tabular}

}\smallskip{}

\bibliographystyle{amsalpha}
\bibliography{bib}

\end{document}